\documentclass[11pt,a4paper]{article}
\usepackage{a4wide}

\usepackage{amsmath,amssymb, amsfonts, amsthm}
\usepackage{mathtools} 
\usepackage{mathrsfs}
\usepackage{mathabx}  
  \changenotsign
\usepackage{nccmath}   
\usepackage{url}       

\mathchardef\mhyphen="2D 

\usepackage{scalerel,stackengine}
\stackMath
\newcommand\rwidecheck[1]{%
\savestack{\tmpbox}{\stretchto{%
  \scaleto{%
    \scalerel*[\widthof{\ensuremath{#1}}]{\kern-.6pt\bigwedge\kern-.6pt}%
    {\rule[-\textheight/2]{1ex}{\textheight}}
  }{\textheight}%
}{0.5ex}}%
\stackon[1pt]{#1}{\scalebox{-1}{\tmpbox}}%
}

\usepackage{enumerate}  

\newcounter{saveenumerate}
\makeatletter
\newcommand{\enumeratext}[1]{%
\setcounter{saveenumerate}{\value{enum\romannumeral\the\@enumdepth}}
\end{enumerate}
#1
\begin{enumerate}
\setcounter{enum\romannumeral\the\@enumdepth}{\value{saveenumerate}}%
}
\makeatother

\usepackage{fancyhdr}
\usepackage{datetime}
\usepackage{color}     
\usepackage{tikz}             
\usetikzlibrary{positioning}  
\usetikzlibrary{arrows}       
\usetikzlibrary{arrows.meta}

\makeatletter
\renewcommand*\env@matrix[1][c]{\hskip -\arraycolsep
  \let\@ifnextchar\new@ifnextchar
  \array{*\c@MaxMatrixCols #1}}
\makeatother

\makeatletter
  \newcommand{\eqnum}{\leavevmode\hfill\refstepcounter{equation}\textup{\tagform@{\theequation}}} 
\makeatother

\usepackage{ifthen}
\newlength{\leftstackrelawd}
\newlength{\leftstackrelbwd}
\def\leftstackrel#1#2{\settowidth{\leftstackrelawd}%
{${{}^{#1}}$}\settowidth{\leftstackrelbwd}{$#2$}%
\addtolength{\leftstackrelawd}{-\leftstackrelbwd}%
\leavevmode\ifthenelse{\lengthtest{\leftstackrelawd>0pt}}%
{\kern-.5\leftstackrelawd}{}\mathrel{\mathop{#2}\limits^{#1}}}

\newcommand{\fc}{\mathfrak{c}}
\newcommand{\fd}{\mathfrak{d}}
\newcommand{\fC}{\mathfrak{C}}
\newcommand{\C}{\mathscr{C}}
\let\e\epsilon

\newcommand{\Fi}{\mathrm{fi}}
\newcommand{\FI}{\mathrm{FI}}

\newcommand{\I}{\mathcal{I}}
\newcommand{\J}{\mathcal{J}}
\newcommand{\calG}{\mathcal G}

\newcommand{\M}{\mathcal{M}}

\newcommand{\bigoh}{\mathcal{O}}

\newcommand{\RR}{\mathbb{R}}
\newcommand{\R}{\mathcal{R}}

\newcommand{\tp}{^{\mathsf{T}}}
\newcommand{\TV}{\mathrm{TV}}
\newcommand{\V}{\mathcal{V}}
\newcommand{\tjeps}{\tilde\jmath^{\hspace{0.08em}\epsilon}}
\newcommand{\tjdel}{\tilde\jmath^{\hspace{0.08em}\delta}}
\newcommand{\wt}{\widetilde}
\newcommand{\ds}{\displaystyle}
\newcommand{\gy}{\mathbf o}
\newcommand{\Rgy}{\R_\gy}
\newcommand{\Rint}{\R_{\mathrm{int}}}

\newcommand\weakto{\mathrel{\relbar\joinrel\rightharpoonup}}
\newcommand\weakstarto{\stackrel{*}{\weakto}}

\newcommand\narrowto{\mathrel{\xrightharpoonup{\text{narrow}}}}
\DeclareMathOperator{\diag}{diag}
\DeclareMathOperator{\Col}{Col}
\DeclareMathOperator{\Null}{Null}

\newcommand{\slow}{\mathrm{slow}}
\newcommand{\fast}{\mathrm{fast}}
\newcommand{\fcycle}{\mathrm{fcyc}}
\newcommand{\damped}{\mathrm{damp}}

\newcommand{\dcyc}{\mathrm{dcyc}}
\newcommand{\dnocyc}{\mathrm{dnocyc}}
\newcommand{\Rslow}{{\R_\slow}}
\newcommand{\Rfast}{{\R_\fast}}
\newcommand{\Rfcycle}{{\R_\fcycle}}
\newcommand{\Rdamped}{{\R_\damped}}

\newcommand{\Rdcyc}{{\R_\dcyc}}
\newcommand{\Rdnocyc}{{\R_\dnocyc}}

\let\div\relax
\DeclareMathOperator*\div{\mathrm{div}}

\DeclareMathOperator\Prob{Prob}

\usepackage{bbm} 
\usepackage{dsfont}
\def\One{\mathds{1}}

\newtheorem{theorem}{Theorem}[section]
\newtheorem{lemma}[theorem]{Lemma}
\newtheorem{proposition}[theorem]{Proposition}
\newtheorem{corollary}[theorem]{Corollary}

\newenvironment{remark}
  {\par\medbreak\refstepcounter{theorem}%
    \noindent\textbf{Remark~\thetheorem. }}%
  {\qed\par\medskip}

\numberwithin{equation}{section}

\usepackage[bookmarks=false]{hyperref}

\usepackage{amsmath}

\author{Mark A. Peletier and D. R. Michiel Renger}
\title{Fast reaction limits via $\Gamma$-convergence of the Flux Rate Functional}

\date{\today}

\begin{document}
\maketitle
\begin{abstract}
	We study the convergence of a sequence of evolution equations for measures supported on the nodes of a graph. The evolution equations themselves can be interpreted as the forward Kolmogorov equations of  Markov jump processes, or equivalently as the equations for the concentrations in a network of linear reactions. 
	The jump rates or reaction rates are divided in two classes;  `slow' rates are constant, and  `fast' rates are scaled as~$1/\e$, and we prove the convergence in the fast-reaction limit $\e\to0$.
	
We establish a $\Gamma$-convergence result for the rate functional in terms of both the concentration at each node and the flux over each edge (the level-2.5 rate function). The limiting system is again described by a functional, and characterizes both fast and slow fluxes in the system. 
	
	This method of proof has three advantages. First, no condition of detailed balance is required. Secondly, the formulation in terms of concentration and flux leads to a short and simple proof of the $\Gamma$-convergence; the price to pay is a more involved compactness proof. Finally, the method of proof deals with approximate solutions, for which the functional is not zero but small, without any changes. 
\end{abstract}

%
%

\section{Introduction}

The aim of this paper is to prove a fast-reaction limit for a sequence of evolution equations on a graph. We first specify the system. 

\bigskip

Let $\calG = (\V,\R)$ be a finite directed diconnected graph with weights $\kappa^\epsilon:\R\to\lbrack0,\infty)$. For each edge $r\in\R$ we denote $r=(r_-,r_+)$, with $r_-,r_+\in\V$ the corresponding source and target nodes. We consider the classical problem of deriving effective equations for the flow on $(\V,\R)$ with two different rates:
\begin{align}
  \dot \rho^\epsilon(t) = -\div (\kappa^\epsilon \otimes \rho^\epsilon(t)), && \rho^\epsilon(0) \text{ fixed}.
\label{eq:ODE}
\end{align}
with discrete divergence $(\div A)_x := \sum_{r_-=x} A_r - \sum_{r_+=x}A_r$, product $(\kappa^\epsilon\otimes\rho)_{r\in \R}:=\kappa^\epsilon_r\rho_{r_-}$, and $t\in\lbrack0,T\rbrack$, $T>0$. 
 We assume that the space of edges is a disjoint union $\R=\Rslow\cup\Rfast$ so that
\begin{equation}
  \kappa^{\epsilon}_r =
  \begin{cases}
    \kappa_r,                      &r\in\Rslow,\\
    \tfrac{1}{\epsilon} \kappa_r.  &r\in\Rfast.
  \end{cases}
\label{eq:k eps}
\end{equation}
We are  interested in the limiting behaviour as $\epsilon\to0$, where the fast edges equilibrate instanteously onto a slow manifold. Such limits, also known as `Quasi-Steady-State Approximations', have a long history in the literature, see for example~\cite{Tikhonov1952} and \cite{Stiefenhofer1998}.

\subsection{$\Gamma$-convergence of the large-deviations rate}
\label{subsec:LDP}

Often, one is not only interested in convergence of the dynamics, but also in convergence of some variational structure such as a gradient structure, or more generally an `action' functional that is minimised by the dynamics~\eqref{eq:ODE}. Of course this convergence is particularly relevant if this action has a physical meaning. The functional that we study in this paper can be interpreted as an action functional in the following way.

Consider a microscopic system of $n$ independent particles $X^\epsilon_i(t)\in\V, i=1,\hdots,n$ that randomly jump from state $X^\epsilon_i(t_-)=r_-$ to a new state $X^\epsilon_i(t)=r_+$ with Markov intensity $\kappa^\epsilon_r$. This is a typical microscopic model for a (bio)chemical system of unimolecular reactions with multiple time scales. The concentration of particles in state $x$ is then $\rho^{n,\epsilon}_x(t):=n^{-1}\sum_{i=1}^n \mathds1_{\{X^\epsilon_i(t)=x\}}$, and the vector of random concentrations $\rho^{n,\epsilon}(t)$ converges to the deterministic solution $\rho^\epsilon(t)$ of~\eqref{eq:ODE} by Kurtz' classical result~\cite{Kurtz1970}. For  large but finite particle numbers $n$, there is a  small probability that $\rho^{n,\epsilon}(t)$ deviates significantly from $\rho^\epsilon(t)$. These  small probabilities are best understood through a large deviations principle~\cite{Feng1994,Leonard1995,Agazzi2018}:
\begin{subequations}
\begin{align}
{-n^{-1}}\log \Prob\big( &\rho^{n,\epsilon}\approx \rho \big)
    \stackrel{n\to\infty}{\sim}
  \I^\epsilon_0(\rho(0)) + \I^\epsilon(\rho), \quad\text{where}\\
  &\I^\epsilon(\rho):=
  \inf_{\substack{j\in L^1([0,T];\RR^\R):\\ \dot\rho = -\div j}} \quad \sum_{r\in \R} \int_{[0,T]}\!s\big(j_r(t) \mid \kappa^\epsilon_r \rho_{r_-}\!(t)\big)\,dt, 
    \label{eq:concentration RF}\\
  &s(a\mid b) :=
    \begin{cases}
      s(a\mid b)=a\log \mfrac{a}{b}-a+b,  &a,b>0,\\
      s(a \mid b)=b,                      &a=0,b\geq0,\\
      s(a \mid b)=\infty,                 &a<0, \ b< 0, \text{ or } a>0,b = 0,
    \end{cases}
\end{align}
\end{subequations}
and $\I^\epsilon_0$ reflects whatever randomness is taken for the initial concentration $\rho^{n,\epsilon}(0)$. We stress that this formula is typical for Markov jump processes; chosing a different microscopic model for the dynamics could lead to different functionals.

If the network satisfies detailed balance, then the rate functional~\eqref{eq:concentration RF} can be related to a gradient flow~\cite{Onsager1931I,Onsager1953I,MielkePeletierRenger2014,MPPR2017}. We shall revisit the detailed balance condition in Section~\ref{subsec:GF}. For a similar interpretation in terms of an action without the detailed balance condition, see~\cite{Bertini2015MFT,Renger2018a}. 

Note that $\I^\epsilon$ is indeed minimised by solutions $\rho^\epsilon$ of~\eqref{eq:ODE}. This implies that we can consider the equation $\I^\e=0$ as a variational formulation of the equation~\eqref{eq:ODE}; this is the point of view known as `curves of maximal slope'~\cite{Ambrosio2008} or the `energy-dissipation principle'~\cite{Mielke16}. An important advantage of this choice of formulation is that $\Gamma$-convergence of $\I^\epsilon$ implies converge of the minimising dynamics (see~\cite[Cor.~7.24]{DalMaso1993} and~\cite{Mielke16}); in other words, one can prove convergence of the solutions by proving $\Gamma$-convergence of the functionals. This is also the method that we adopt in this paper.

\subsection{$\Gamma$-convergence of the flux large-deviations rate}
\label{subsec:flux LDP}

One difficulty in proving $\Gamma$-convergence of the functional $\I^\epsilon_0+\I^\epsilon$, however,  is that $\I^\epsilon$ is implicitly defined by a constrained minimisation problem. The constrained infimum of the sum in~\eqref{eq:concentration RF} is an infimal convolution (see~\cite[Sec.~3.4]{MPPR2017}). This shows that the evolution of the concentrations in different nodes are strongly intertwined, which considerably complicates the mathematical analysis. For example, the related work~\cite{DisserLieroZinsl2018} requires an orthogonality assumption to decouple the concentrations.

We can however avoid this difficulty by considering a different functional instead. Observe that the variable $j_r(t)$ in \eqref{eq:concentration RF} has the interpretation of  a flux: it measures how much mass is transported through edge $r$ at time $t$. Naturally, one can rephrase~\eqref{eq:ODE} in terms of this flux as the coupled system
\begin{align}
  \dot \rho^\epsilon(t) = -\div j^\epsilon(t) \quad\text{and}\quad j^\epsilon(t)=\kappa^\epsilon \otimes \rho^\epsilon(t), && \rho^\epsilon(0) \text{ fixed}.
\label{eq:flux ODE}
\end{align}
On the level of the microscopic particle system one can also define the random particle flux $J^{n,\epsilon}$, which yields the large-deviation principle~\cite{Baiesi2009,Renger2018a,PattersonRenger2019}:
\begin{align}
&-n^{-1}\log \Prob\big( (\rho^{n,\epsilon},J^{n,\epsilon})\approx (\rho,j) \big)
    \stackrel{n\to\infty}{\sim}
  \I^\epsilon_0(\rho(0)) + \J^\epsilon(\rho,j), \qquad\text{where}\\[3\jot]
  &\J^\epsilon(\rho,j):= \begin{cases}  
  \ds \smash[b]{\sum_{r\in\R} \int_{[0,T]}}\!s\big(j_r(t)\mid \kappa^\epsilon_r \rho_{r_-}\!(t)\big)\,dt, &\text{if }\rho\in W^{1,1}([0,T];\RR^\V), j\in L^1([0,T];\RR^\R),  \\
              &\qquad \text{ and }\dot \rho=-\div j,\\[2\jot]
   \infty, &\text{otherwise}.
 \end{cases}
\label{eq:flux RF}
\end{align}
Indeed, the functional $\J^\e$ is related to~\eqref{eq:concentration RF} by $\I^\epsilon(\rho)=\inf_{\dot\rho=-\div j} J^\epsilon(\rho,j)$, which is consistent with the `contraction principle' in large-deviations theory. Its minimiser~\eqref{eq:flux ODE} follows the same evolution as the minimiser~\eqref{eq:ODE}, but provides with more information: the flux. From a physics perspective, this additional information is important to understand non-equilibrium thermodynamics; see for example~\cite{Bertini2015MFT}, \cite{MPPR2017} and \cite[Sec.~4]{Renger2018a}. From a mathematical perspective, we will use the property that the flux functional $\J^\epsilon$ is  a sum over edges to decompose networks into separate components.

The goal of this paper is to prove convergence of the functional $\I^\epsilon_0+\J^\epsilon$ to a limit functional, whose minimiser describes the effective dynamics for \eqref{eq:flux ODE}. As a consequence, we obtain $\Gamma$-convergence of the functional $\I^\epsilon_0+\I^{\epsilon}$, convergence of solutions of the flux ODE~\eqref{eq:flux ODE}, and convergence of solutions of the ODE~\eqref{eq:ODE}.

In order to track diverging fluxes and  vanishing concentrations, we shall introduce a number of rescalings before taking the $\Gamma$-limit, as we explain in the next section.

\subsection{Network decomposition: nodes}
\label{subsec:network decomposition nodes}

We decompose the network into different components according to their scaling behaviour. To explain the main ideas, consider the  example of Figure~\ref{fig:network}. Recall from \eqref{eq:k eps} that we assume that $\R=\R_\slow\cup\R_\fast$, where the slow edges have rates of order $1$, and the fast edges of order $1/\e$.

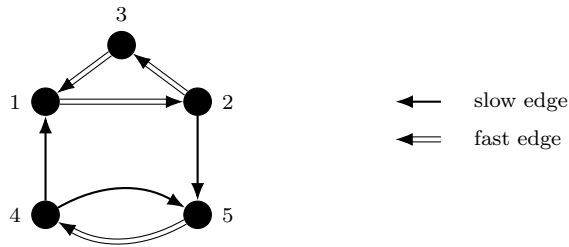
\begin{figure}[h!]
  \centering
  \begin{tikzpicture}[scale=0.5,baseline]
    \tikzstyle{every node}=[font=\scriptsize]

    \node[label=left:$1$](y1) at (0,0) [circle,draw,fill]{};
    \node[label=right:$2$](y2) at (4,0) [circle,draw,fill]{};
    \node[label=above:$3$](y3) at (2,1.5) [circle,draw,fill]{};
    \node[label=left:$4$](y4) at (0,-3) [circle,draw,fill]{}; 
    \node[label=right:$5$](y5) at (4,-3) [circle,draw,fill]{};
    \draw[-{Latex},double distance=1.5] (y1)-- (y2); 
    \draw[-{Latex},double distance=1.5] (y2)-- (y3); 
    \draw[-{Latex},double distance=1.5] (y3)-- (y1); 
    \draw[-{Latex},thick]               (y4)-- (y1); 
    \draw[-{Latex},thick]               (y2)-- (y5); 
    \draw[-{Latex},thick]               (y4) to [bend left] (y5); 
    \draw[-{Latex},double distance=1.5] (y5) to [bend left] (y4); 

    \draw[-{Latex},thick] (10.4,0)--(9.2,0);
      \node[anchor=west] at (11,0){slow edge};
    \draw[-{Latex},double distance=1.5] (10.4,-1)--(9.2,-1);
      \node[anchor=west] at (11,-1){fast edge};
  \end{tikzpicture}
\caption{An example of a network with slow and fast edges.}
\label{fig:network}
\end{figure}

The first step in the decomposition is to categorise the nodes. In the example, node $5$ is expected to have low concentration,  since any mass at node 5  will be quickly transported to node $4$. We make this statement precise by considering the equilibrium concentration. Since we assume the network to be diconnected, there  exists a unique equilibrium concentration $0<\pi^\epsilon\in\RR^\V$ for the dynamics~\eqref{eq:ODE}; we will always assume that $\pi^\e$ is normalized, i.e.\ $\sum_{x\in\V} \pi^\e_x = 1$. 
We use the equilibrium concentrations to  subdivide the nodes into two classes, $\V=\V_0\cup\V_1$, where
\begin{align}
  \V_0:=\{x\in\V: \pi^\epsilon_x \xrightarrow{\epsilon\to0} \pi_x>0\}, &&\text{and}&&
  \V_1:=\{x\in\V: \tfrac1\epsilon\pi^\epsilon_x \xrightarrow{\epsilon\to0} \tilde\pi_x>0\},
\label{eq:pi convergence}
\end{align}
and the tilde is used to stress that the quantity is rescaled. 
This decomposition implies an assumption that $\pi^\epsilon_x$ is either of order $1$ or of order $\epsilon$. In fact, one can construct networks with $\R=\R_\slow\cup\R_\fast$ with stationary states $\pi^\epsilon_x$ of order $\e^2$, $\e^3$, or higher, but in this paper such networks will be ruled out by our assumption that there are no `leaked' fluxes (see below). We introduce a further subdivision of the nodes after categorising the fluxes.

\subsection{Network decomposition: fluxes}
\label{subsec:network decomposition fluxes}

We expect that $j^\epsilon_r$ is comparable to $\kappa^\epsilon_r \rho^\epsilon_{r_-}$, which in turn we expect to be comparable to $\kappa^\epsilon_r \pi^\epsilon_{r_-}$. Hence the flux or amount of mass being transported through an edge $r$ not only depends on the order of $\kappa^\epsilon_r$, but also on the amount of available mass in the source node $r_-$, of order $\pi^\epsilon_{r_-}$. Therefore the scaling behaviour of the flux falls into one of the following four different categories:
\begin{center}
\begin{tabular}{r || l l | l l |}
$j^\epsilon_r$  & \multicolumn{2}{c|}{$r_-\in\V_0$} & \multicolumn{2}{c|}{$r_-\in\V_1$} \\
\hline
$r\in\Rslow$ & $\bigoh(1)$  &``slow'' & $\bigoh(\epsilon)$ & ``leak'' \\
$r\in\Rfast$ & $\bigoh(1/\epsilon)$ &``fast cycle''  & $\bigoh(1)$       & ``damped'' \\
\hline
\end{tabular}
\end{center}
In this paper we rule out ``leak'' fluxes by assumption, so that $\R=\Rslow\cup\Rfcycle\cup\Rdamped$, with
\begin{align*}
  \Rfcycle:=\{r\in \Rfast: r_- \in \V_0\}  &&\text{and}&&  \Rdamped:=\{r\in\Rfast: r_-\in \V_1\}.
\end{align*}

\begin{figure}[h!]
  \centering
  \begin{tikzpicture}[scale=0.5,baseline=-4.5em]
    \tikzstyle{every node}=[font=\scriptsize]

    \node[label=left:$1$](x1) at (0,0) [circle,draw,fill]{};
    \node[label=right:$2$](x2) at (4,0) [circle,draw,fill]{};
    \node[label=above:$3$](x3) at (2,1.5) [circle,draw,fill]{}; 
    \node[label=left:$4$](x4) at (0,-3) [circle,draw,fill]{}; 
    \node[label=right:$5$](x5) at (4,-3) [circle,draw,fill=gray]{}; 
    \draw[-{Latex},double distance=1.5] (x1)-- (x2);
    \draw[-{Latex},double distance=1.5] (x2)-- (x3); 
    \draw[-{Latex},double distance=1.5] (x3)-- (x1);
    \draw[-{Latex},thick]               (x4)-- (x1);
    \draw[-{Latex},thick]               (x2)-- (x5) ;
    \draw[-{Latex},thick]               (x4) to [bend left] (x5);
    \draw[-{Latex},double distance=1.5,dashed] (x5) to [bend left] (x4);
      \begin{scope}
        \path[clip] (x5) to [bend left] (x4) -- (x4.center)--(x5.center)--cycle; 
        \draw[-{Latex},double distance=1.5] (x5) to [bend left] (x4);
      \end{scope}
      \draw[white,line width=1.2pt,{-{Latex}[black]}] (x5) to [bend left] (x4);

    \node(y1) at (9.75,1.5) [circle,draw,fill]{};
      \node[anchor=west] at (11,1.5){$\V_0$-node};
    \node(y1) at (9.75,0.5) [circle,draw,fill=gray]{};
      \node[anchor=west] at (11,0.5){$\V_1$-node};
    \draw[-{Latex},thick] (10.4,-0.5)--(9.1,-0.5);
      \node[anchor=west] at (11,-0.5){slow flux};
    \draw[-{Latex},double distance=1.5] (10.4,-1.5)--(9.1,-1.5);
      \node[anchor=west] at (11,-1.5){fast cycle flux};
    \draw[-{Latex},double distance=1.5,dashed] (10.4,-2.5) to (9.1,-2.5);
      \begin{scope}
        \path[clip] (9.6,-2.5) -- (10.4,-2.5) -- (10.4,-2.3) -- (9.1,-2.3)-- (9.1,-2.5);
        \draw[-{Latex},double distance=1.5] (10.4,-2.5) to (9.1,-2.5);
      \end{scope}
      \draw[white,line width=1pt,{-{Latex}[black]}] (10.4,-2.5) to (9.1,-2.5);
      \node[anchor=west] at (11,-2.5){damped flux};
  \end{tikzpicture}
  \qquad\quad
    \begin{tikzpicture}[scale=0.5]
    \tikzstyle{every node}=[font=\footnotesize]
    \node(x1) at (0,0.5) [circle,draw,fill,inner sep=2]{};
    \node(x2) at (1.5,0.5) [circle,draw,fill,inner sep=2]{};
    \node(x3) at (0.75,1.2) [circle,draw,fill,inner sep=2]{}; 
    \draw[double distance=0.75,-{latex}] (x1)-- (x2);
    \draw[double distance=0.75,-{latex}] (x2)-- (x3); 
    \draw[double distance=0.75,-{latex}] (x3)-- (x1);

    \node[label=right:$123$,minimum size=3em](x123) at (0.75,0.7) [circle,draw]{};
    \node[label=left:$4$](x4) at (0.375,-3) [circle,draw,fill]{};
    \node[label=right:$5$](x5) at (1.125,-3) [circle,draw,fill=gray]{}; 
    \draw[-{Latex},thick]            (x4) to[bend left] (x123);
    \draw[-{Latex},thick]             (x123) to[bend left] (x5);
    \draw[-{Latex},thick]            (x4) to [out=90,in=90,looseness=3] (x5);

    \draw[-{Latex},double distance=1.5,dashed] (x5) to [out=270,in=270,looseness=3] (x4);
      \begin{scope}
        \path[clip] (x5) to [out=270,in=270,looseness=3] (x4) -- (x4.center)--(x5.center)--cycle;
        \draw[-{Latex},double distance=1.5] (x5) to [out=270,in=270,looseness=3] (x4);
      \end{scope}
      \draw[{-{Latex}[black]},white,line width=1pt] (x5) to [out=270,in=270,looseness=3] (x4);
  \end{tikzpicture}
\caption{The example from Figure~\ref{fig:network}, redrawn using the categorisation of nodes and fluxes (left); the final reduction to a two-node network in the limit $\e\to0$ (right). \label{fig:network2}}
\end{figure}
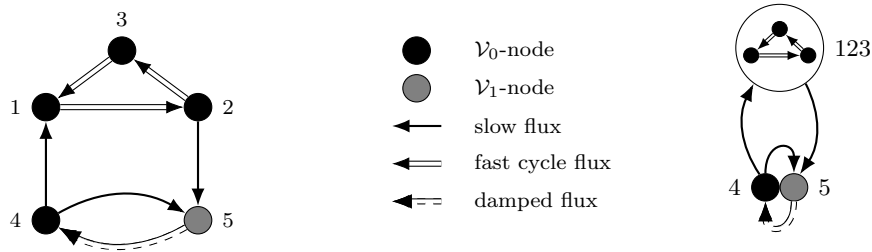

Let us now explain these four categories in more detail by considering the example network of Figure~\ref{fig:network}, which can now be redrawn as Figure~\ref{fig:network2}. 

\begin{enumerate}
\item What we shall call the \emph{slow fluxes} are fluxes through a slow edge that start at a node in $\V_0$. Typically, these slow fluxes will be of order $\bigoh(1)$, and they depend on $\epsilon$ only indirectly through dependence on the other fluxes. 

\item For the fast edges however, there is a fundamental difference between the fluxes $1\to2\to3\to1$ and the flux $5\to4$. The three fluxes $1\to2\to3\to1$ constitute a cycle of fast edges, with fluxes of order $\bigoh(1/\epsilon)$. Therefore mass will rotate very fast through this cycle, and in the limit $\epsilon\to0$, the mass present in the cycle will instanteneously equilibrate over these three edges. Moreover, any mass inserted into this cycle through the slow flux $4\to1$ will also instantaneously equilibrate over the nodes in the cycle, and any mass removed from the cycle through the slow flux $2\to5$ may be withdrawn from any node in the cycle. Practically this means that in the limit the cycle/diconnected component $1\to2\to3\to1$ acts as one node $\fc:=\{1,2,3\}$. We shall see in Lemma~\ref{lem:fast cycles} that all edges with $r\in\Rfast$ and $r_-\in\V_0$ are indeed  part of a cycle, which justifies the name \emph{fast cycle}. 

\item By contrast, the fast edge $5\to4$ is not part of a fast diconnected component. One does expect mass in node $5$ to be transported very fast into node $4$, but since there is no fast inflow, the mass in node $5$ will be strongly depleted after the initial time. After this, the amount of mass that will be actually transported through edge $5\to4$ is fully subject to the amount of inflow of mass into node $5$ by the slow fluxes $2\to5$ and $4\to5$, and will therefore be of $\bigoh(1)$. We shall call the flux $5\to4$ a \emph{damped flux}; its corresponding edge is fast, but the flux is damped by the fact that there is not enough mass available in the source node $5$. In the limit, any mass that is inserted into node $5$ from node $2$ or $4$ will be immediately pushed into node $4$. 

\item Now imagine a flux $5\to1$, not drawn in the picture. Since there is a damped flux going out of node $5$, almost all mass from node $5$ will follow that flux into node $4$, whereas very little mass from node $5$ would leak away into node $1$. We shall call such fluxes \emph{leak fluxes}. Since they contribute little to the behaviour of the whole network we rule out this possibility by assumption. This also rules out the possibility of higher orders of $\pi^\epsilon_x$ as mentioned above.
\end{enumerate}
An even further subdivision of $\Rdamped$ will be discussed in Section~\ref{subsec:spikes}, but this will not be needed in the general discussion.

\subsection{Network decomposition: connected components}
\label{subsec:connected components}

After categorising the fluxes, we now further subdivide the nodes of $\V_0$ into $\V_0=\V_{0\fcycle}\cup\V_{0\slow}$, consisting of nodes that are part of a fast cycle and the remainder:
\begin{align*}
  \V_{0\fcycle}:=\{x\in\V_0: \exists r\in\Rfcycle, r_-=x \},
  &&\text{and}&&
  \V_{0\slow}:=\V_0\backslash\V_\fcycle.
\end{align*}
The notation reflects the expectation that the concentration in the nodes in $\V_{0\fcycle}$ will instantenously equilibrate over the diconnected components of the graph $(\V_{0},\Rfcycle)$. We collect these components in the set
\begin{equation*}
  \fC:=\left\{ \fc\subset \V_{0} :\  \forall x,y \in\fc, \  \exists (r^k)_{k=1}^K\subset\Rfcycle,\  r^1_-=x, r^k_+=r^{k+1}_-,r^K_+=y\right\}.
\end{equation*}
To each $\fc\in \fC$ corresponds the equilibrium mass
\begin{equation}
  \pi_\fc^\e:=\sum_{x\in\fc} \pi_x^\e, \qquad c\in\fC.
\label{eq:pic}
\end{equation}
We will see in Lemma~\ref{lem:fast cycles} that a component $\fc\in\fC$ can be considered a union of cycles in the graph $(\V_{0\fcycle},\Rfcycle)$. Consequently, if there exists a fast-cycle path from $x$ to $y$ then  there also exists a fast-cycle path from $y$ to $x$. This remark also implies that each fast component $\fc$ is a subset of $\V_{0\fcycle}$.

Observe that, as illustrated in Figure~\ref{fig:network2}, we do not combine the nodes in $\V_1$ and $\V_0$ into single nodes; instead we preserve the nodes, and we  keep track of the fast cycle as well as the damped fluxes. This is motivated by our Theorem~\ref{th:equicoercivity}, which yields sufficient compactness in the $\V_1$-concentrations, damped fluxes and fast cycle fluxes.


\subsection{Rescaled flux and initial functionals}
\label{sec:rescaled}


In Sections~\ref{subsec:network decomposition nodes} and \ref{subsec:network decomposition fluxes} we categorised the nodes and fluxes by their typical scaling behaviour. We shall prove that the scaling behaviour of these categories is not only typical for the effective dynamics but actually for any dynamics with finite large-deviation cost. In order to do so we rescale all concentrations and fluxes according to their respective scalings.

We expect concentrations $\rho^\epsilon_x$ to follow $\pi_x^\e$, and therefore to be of order order $1$ on $\V_0$ and of order $\epsilon$ on $\V_1$. This motivates the rescaling the concentrations by working with the densities $u^\e$, defined by
\begin{equation*}
  u^\e_x(t):=\frac{\rho^\e_x(t)}{\pi^\e_x(t)},
\end{equation*}
where $x\in\V$ or $x\in\tilde\V_0\cup\V_1$, using~\eqref{eq:pic}. Although $\V_{0\fcycle}=\bigcup\fC$, we study $u^\e_x(t)$ for $x\in\V_{0\slow}\cup\V_{0\fcycle}\cup\fC$, assuming that $u^\e_\fc$ and $u^\e_x, x\in\fc$ are related by
\begin{equation}
  \pi^\e_\fc u^\e_\fc(t) = \sum_{x\in\fc} \pi^\e_x u^\e_x(t),
\label{eq:sum density eps}
\end{equation}
which we consider as a special continuity equation, additional to $\dot\rho^\e=-\div j^\e$. The distinction between $u^\e_x$ and $u^\e_\fc$ allows for two different notions of compactness: a weaker compactness for $u^\e_x$ with $x\in\V_{0\fcycle}$, and a stronger compactness for $u^\e_\fc$ for any $\fc\in\fC$.


As explained in Section~\ref{subsec:network decomposition fluxes}, the fluxes are expected to scale as $j^\epsilon_r(t)=\bigoh(\kappa^\epsilon_r\pi^\epsilon_{r_-})$. The slow and damped fluxes are of order $1$ and therefore need not be rescaled. For fast cycle fluxes, of order $1/\epsilon$, we introduce the rescaled flux $\tjeps_r$, defined by
\begin{align*}
  j^\epsilon_r(t)=: \kappa_r^\epsilon\rho^\epsilon_{r_-}\!(t) + \mfrac{1}{\sqrt{\epsilon}}\tjeps_r = \mfrac{1}{\epsilon}\kappa_r \pi^\epsilon_{r_-}\!(t) u^\epsilon_{r_-}\!(t) + \mfrac{1}{\sqrt{\epsilon}}\tjeps_r,
  &&r\in\Rfcycle.
\end{align*}
It turns out that this deviation from $\kappa_r^\e\rho_r^\e$ of order $1/\sqrt\e$ is the right choice for sequences along which $\I_0^\e + \J^\e$ is bounded, since this scaling is natural in the context of the compactness and $\Gamma$-limit results that we prove below.

To shorten the expressions we shall write
\begin{align*}
  u_{\V_{0\slow}}:= (u_x)_{x\in\V_{0\slow}}, \quad
  u_{\V_{0\fcycle}}:= (u_x)_{x\in \V_{0\fcycle}}, \quad
  u_{\fC}:= (u_\fc)_{\fc\in\fC}, \quad
  u_{\V_1}:= (u_x)_{x\in\V_1},\\
  j_{\Rslow}:=(j_r)_{r\in\Rslow}, \quad
  j_{\Rdamped}:=(j_r)_{r\in\Rdamped}, \quad
\tilde\jmath_{\Rfcycle}:=(\tilde\jmath_r)_{r\in\Rfcycle},
\end{align*} 
and finally by a slight abuse of notation $(u,j):=( u_{\V_{0\slow}},  u_{\V_{0\fcycle}},  u_{\fC}, u_{\V_1},  j_{\Rslow},  j_{\Rdamped}, \tilde\jmath_{\Rfcycle})$.
With these rescalings and notation we now rewrite the large-deviations rate functional~\eqref{eq:flux RF} as:
\begin{align}
  \tilde\J^\epsilon(u,j)
  &=\tilde\J^\epsilon\big(u_{\V_{0\slow}},u_{\V_{0\fcycle}},u_{\fC},u_{\V_1},j_{\Rslow},j_{\Rdamped},\tilde\jmath_{\Rfcycle}\big) \notag\\
  &\qquad:=\J^\epsilon\big(\pi^\epsilon u, (j_\Rslow,j_\Rdamped,\epsilon^{-1}\kappa\otimes\pi^\epsilon u^\epsilon + \epsilon^{-1/2}\tilde\jmath_\Rfcycle)\big) \notag\\
  &\qquad=\underbrace{\sum_{r\in\Rslow} \int_{[0,T]}\!s\big(j_r(t)\mid \kappa_r \pi^\epsilon_{r_-}\!u_{r_-}\!(t)\big)\,dt}_{=:\tilde\J^\epsilon_\slow(u_{\V_{0\slow}},u_{\V_{0\fcycle}})} \notag\\
  &\qquad\quad+\underbrace{\sum_{r\in\Rdamped} \int_{[0,T]}\!s\big(j_r(t)\mid \tfrac{1}{\epsilon}\kappa_r\pi^\epsilon_{r_-}\!u_{r_-}\!(t)\big)\,dt}_{=:\tilde\J^\epsilon_\damped(u_{\V_1},j_\Rdamped)} \notag\\
  &\qquad\quad+\underbrace{\sum_{r\in\Rfcycle} \int_{[0,T]}\!s\Big(\tfrac{1}{\epsilon}\kappa_r \pi^\epsilon_{r_-}\!u_{r_-}\!(t) + \mfrac{1}{\sqrt{\epsilon}}\tilde\jmath_r(t)  \Bigm| \tfrac1\epsilon\kappa_r\pi^\epsilon_{r_-}\!u_{r_-}\!(t)\Big)\,dt}_{=:\tilde\J^\epsilon_\fcycle(u_{\V_{0\fcycle}},\tilde\jmath_\Rfcycle)},
\label{eq:flux RF rescaled}
\end{align}
where $\tilde\J^\epsilon=\infty$ if one of the conditions of~\eqref{eq:flux RF} and \eqref{eq:sum density eps} is violated. Recall that $\pi^\epsilon_{r_-}\approx\pi_{r_-}$ for $r\in\Rslow$ and $\pi^\epsilon_{r_-}\approx \epsilon\tilde\pi_{r_-}$ for $r\in\Rdamped$, so that the two functionals $\tilde\J^\epsilon_\slow$ and $\tilde\J^\epsilon_\damped$ are very similar.

In order to control the initial condition we include the initial large-deviation rate function~$\I^\epsilon_0$ in the analysis. As mentioned in Section~\ref{subsec:LDP}, this function depends on the choice of the initial probability. As is common, we  choose the random dynamics to start independently at the invariant measure. Since linear reactions correspond to independent copies of the process, the particles modelled by the invariant measure are also independent, and hence $\I^\epsilon_0\big(\rho(0)\big) = \sum_{x\in\V} s\big(\rho_x(0)\mid \pi^\epsilon\big)$ by Sanov's Theorem~\cite[Th.~6.2.10]{Dembo1998}. We again rescale this functional to work with densities instead:
\begin{align}
\label{def:I0}
  \tilde\I^\epsilon_0\big(u(0)\big) &= \tilde\I^\epsilon_0\big(u_{\V_{0\slow}}(0),u_{\V_{0\fcycle}}(0),u_{\fC}(0),u_{\V_1}(0)\big) :=
\I^\epsilon_0\big(\pi^\epsilon\otimes u(0)\big) \notag\\
  &:=\sum_{x\in\V_{0\slow}} s\big(\pi^\epsilon_x u_x(0) \mid \pi^\epsilon_x \big) + \sum_{x\in\V_{0\fcycle}} s\big(\pi^\epsilon_x u_x(0) \mid \pi^\epsilon_x \big) + \sum_{x\in\V_1} s\big(\pi^\epsilon_x u_x(0) \mid \pi^\epsilon_x \big).
\end{align}
The minimiser of $\tilde \I_0$ is the vector of densities all equal to one. 


\subsection{Main results: compactness and $\Gamma$-convergence}
\label{sec:main results}

We now focus on the $\Gamma$-limit of the rescaled functional $\tilde\I^\epsilon_0+\tilde\J^\epsilon$, in the space 
\begin{multline*}
  \Theta:=  
  C([0,T];\RR^{\V_{0\slow}})
    \times
  L^\infty([0,T];\RR^{\V_{0\fcycle}})
    \times
  C([0,T];\RR^{\fC})  
    \times
  \M([0,T];\RR^{\V_1})\\
    \times
  L^\C([0,T];\RR^{\Rslow})
    \times
  \M([0,T];\RR^{\Rdamped})
    \times
  L^\C([0,T];\RR^{\Rfcycle}),
\end{multline*}
where $C$ is the space of continuous functions, $\M$ denotes spaces of bounded measures, and $L^\C$ denote Orlicz spaces corresponding to the nice Young function (see Section~\ref{subsec:Orlicz}):
\begin{equation*}
  \C(a):=\inf_{p-q=a} s(p \mid 1) + s(q \mid 1).
\end{equation*}
We always make the implicit assumption that $u_\fC$ and $u_{\V_{0\fcycle}}$ are connected by~\eqref{eq:sum density eps}.

We make $\Theta$ into a topological space by equipping each space $C$ with the uniform topology, each $L^\infty$ and $L^\C$ with their weak-\textasteriskcentered\ topologies and each measure space $\M$ with the narrow topology (defined by duality with continuous functions).

Of course  $\Gamma$-convergence properties strongly depend on the chosen topology. In fact, it is known that different topologies may lead to different $\Gamma$-limits~\cite[Ch.~6]{DalMaso1993}, \cite[Sec.~2.6]{Mielke2016}. The choice of this particular topological space $\Theta$ is motivated by our first main result:
\begin{theorem}[Equicoercivity]
Let $(u^\epsilon,j^\epsilon)_{\epsilon>0}\subset\Theta$ such that
\begin{equation*}
  \tilde\I^\epsilon_0\big(u^\epsilon(0)\big) + \tilde\J^\epsilon(u^\epsilon,j^\epsilon)\leq C \qquad\text{for some } C>0.
\end{equation*}
Then there exists a $\Theta$-convergent subsequence.
\label{th:equicoercivity}
\end{theorem}

This equicoercivity identifies a topology that is generated by the sequence of functionals itself, and therefore natural for the $\Gamma$-convergence. Note that the topologies for $u_{\V_0}$ and $u_\fC$ are much stronger than the other ones. This will be needed to interchange limits $\lim_{\e\to0} \lim_{t\downarrow 0} u^\e_{\V_{0\slow}}(t)$ and $\lim_{t\downarrow 0} \lim_{\e\to0} u^\e_{\V_{0\slow}}(t)$ in order to converge in the continuity equation later on. By contrast, such strong compactness is not to be expected for $u^\epsilon_{\V_1}$, nor is it needed, since the $u^\epsilon_{\V_1}(0)$ will not play a role in the limit due to instantaneous equilibration.

Our second main result is the $\Gamma$-convergence:
\begin{theorem} In the topological space $\Theta$:
\begin{equation*}
  \tilde\I^\epsilon_0 + \tilde\J^\epsilon_\slow + \tilde\J^\epsilon_\fcycle + \tilde\J^\epsilon_\damped =:\tilde\I^\epsilon_0 + \tilde\J^\epsilon\xrightarrow[\epsilon\to0]{\Gamma} \tilde\I^0_0 + \tilde\J^0 := \tilde\I^0_0 + \tilde\J^0_\slow + \tilde\J^0_\fcycle + \tilde\J^0_\damped,
\end{equation*}
where, setting $u_{r_-}:=u_\fc$ for any $r_-\in\fc$,
\begin{align*}
  \tilde\I_0^0\big(u(0)\big) 
  &:=\sum_{x\in\V_{0\slow}} s\big(\pi_x u_x(0) \mid \pi_x \big) + \sum_{\fc\in\fC} s\big(\pi_\fc u_\fc(0) \mid \pi_\fc \big),\\
  \tilde\J^0_\slow(u_{\V_{0\slow}},u_{\V_{0\fcycle}},j_\Rslow) &:= \sum_{r\in\Rslow} \int_{[0,T]}\!s\big(j_r(t)\mid \kappa_r \pi_{r_-}\!u_{r_-}\!(t)\big)\,dt,\\
  \tilde\J^0_\damped(u_{\V_1},j_\Rdamped) &:= \sum_{r\in\Rdamped} \int_{[0,T]}\!s\big(j_r\mid \kappa_r \tilde\pi_{r_-}\!u_{r_-}\!\big)\,(dt),\\
  \tilde\J^0_\fcycle(u_{\V_{0\fcycle}},\tilde\jmath_\Rfcycle) &:=
   \displaystyle \mfrac12\sum_{r\in\Rfcycle} \, \int_{[0,T]}\!\mfrac{\tilde\jmath_r(t)^2}{\kappa_r\pi_{r_-}\!u_{r_-}\!(t)}\,dt,
\end{align*}
and we set $\tilde\J^0=\infty$ if the limit continuity equations~\eqref{eq:limit cont eq} are violated.
\label{th:main result}
\end{theorem}

The explicit form~\eqref{eq:limit cont eq} of the limit continuity equations will be derived in Lemma~\ref{lem:limit cont eq}, after the required notions are introduced and the required results about the network and continuity equations are proven. In our third main result, explained in the next section, we show that both the densities $u_{\V_1}$ and the damped fluxes $j_\Rdamped$ may become measure-valued in time; therefore we use a slight generalisation of the function $s$ to measure-valued trajectories, i.e.:
\begin{equation}
  \int_{[0,T]}\!s\big(j_r\mid \kappa_r \tilde\pi_{r_-}\!u_{r_-}\!\big)\,(dt) :=
  \begin{cases}
    \displaystyle \int_{[0,T]}\!s\Big( \mfrac{dj_r}{\kappa_r \tilde\pi_{r_-}du_{r_-}}(t) \mid 1 \Big) \kappa_r \tilde\pi_{r_-}\!u_{r_-}\!(dt),  &\text{if } j_r \ll u_{r_-},\\
    \infty, &\text{otherwise}.
  \end{cases}
\label{eq:measure-valued s}
\end{equation}

Comparing Theorem~\ref{th:main result} with Figure~\ref{fig:network}, we see that the limit functional contains additional information about the $\V_1$ nodes that contract to a single node in the limit, and about all slow, fast cycle and damped fluxes. Due to this additional information, the proof of the $\Gamma$-convergence is relatively straightforward, e.g. without the need of unfolding techniques. This illustrates our `philosophical' message that the mathematics becomes easier if one takes fluxes into account, which was also observed in~\cite{PattersonRenger2019} where the large-deviation principle~\eqref{eq:flux RF} was proven.

\subsection{Main result: the development of spikes}
\label{subsec:spikes}


The equicoercivity of $u^\epsilon_{\V_1}$ and $j^\epsilon_\damped$ will be derived by uniform $L^1$-bounds in Lemmas~\ref{l:bdd-densities} and~\ref{l:bdd-damped-fluxes}. From these bounds one can only extract compactness as measures, in the narrow sense, so that $u^\epsilon_{\V_1}$ and $j^\epsilon_\damped$ may develop measure-valued singularities or \emph{spikes} in time. 

For the densities $u_{\V_1}^\e$, such spikes can not be ruled out, regardless of the network structure. This is easy to see from the fact that these densities become fully uncoupled in the limit continuity equation~\eqref{eq:limit cont eq V1}. From~\eqref{eq:measure-valued s} one sees that one may choose  large $u_r$ for $r\in\V_1$, provided  $j_r \ll u_{r_-}$.

For the fluxes  $j^\epsilon_\damped$, the occurrence of spikes is  related to the presence of \emph{damped cycles}, i.e.\ cycles of damped reactions. The example of Figures~\ref{fig:network} and \ref{fig:network2} has no such damped cycles, but Figure~\ref{fig:damped fluxes} illustrates the concept. 

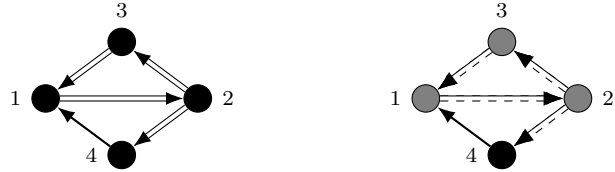
\begin{figure}[h!]
  \centering
  \begin{tikzpicture}[scale=0.5,baseline]
    \tikzstyle{every node}=[font=\scriptsize]

    \node[label=left:$1$](x1) at (0,0) [circle,draw,fill]{};
    \node[label=right:$2$](x2) at (4,0) [circle,draw,fill]{};
    \node[label=above:$3$](x3) at (2,1.5) [circle,draw,fill]{}; 
    \node[label=left:$4$](x4) at (2,-1.5) [circle,draw,fill]{}; 

    \draw[-{Latex},double distance=1.5] (x1)-- (x2);
    \draw[-{Latex},double distance=1.5] (x2)-- (x3); 
    \draw[-{Latex},double distance=1.5] (x3)-- (x1);
    \draw[-{Latex},double distance=1.5] (x2)-- (x4);
    \draw[-{Latex},thick] (x4)-- (x1);

    \node[label=left:$1$](x5) at (10,0) [circle,draw,fill=gray]{};
    \node[label=right:$2$](x6) at (14,0) [circle,draw,fill=gray]{};
    \node[label=above:$3$](x7) at (12,1.5) [circle,draw,fill=gray]{}; 
    \node[label=left:$4$](x8) at (12,-1.5) [circle,draw,fill]{}; 

    \draw[-{Latex},double distance=1.5,dashed] (x5) to (x6);
      \begin{scope}
        \path[clip] (x6) -- (x5) -- (x5.center) -- (x6.center)-- (x6);
        \draw[-{Latex},double distance=1.5] (x5) to (x6);
      \end{scope}
      \draw[{-{Latex}[black]},white,line width=1pt] (x5) to (x6);
    \draw[-{Latex},double distance=1.5,dashed] (x6) to (x7);
      \begin{scope}
        \path[clip] (x7) -- (x6) -- (x6.center) -- (x7.center)-- (x7);
        \draw[-{Latex},double distance=1.5] (x6) to (x7);
      \end{scope}
      \draw[{-{Latex}[black]},white,line width=1pt] (x6) to (x7);
    \draw[-{Latex},double distance=1.5,dashed] (x7) to (x5);
      \begin{scope}
        \path[clip] (x7) -- (x5) -- (x5.center) -- (x7.center)-- (x7);
        \draw[-{Latex},double distance=1.5] (x7) to (x5);
      \end{scope}
      \draw[{-{Latex}[black]},white,line width=1pt] (x7) to (x5);
    \draw[-{Latex},double distance=1.5,dashed] (x6) to (x8);
      \begin{scope}
        \path[clip] (x6) -- (x8) -- (x8.center) -- (x6.center)-- (x6);
        \draw[-{Latex},double distance=1.5] (x6) to (x8);
      \end{scope}
      \draw[{-{Latex}[black]},white,line width=1pt] (x6) to (x8);
    \draw[-{Latex},thick]               (x8)-- (x5);
  \end{tikzpicture}
\caption{An example of a network with a cycle of damped fluxes. On the left the network with the distinction between fast and slow edges; on the right the redrawn network with the node and flux classification as in Figure~\ref{fig:network2}.}
\label{fig:damped fluxes}
\end{figure}

To study this we further subdivide $\Rdamped$ into damped cycles and the rest, $\Rdamped=\Rdcyc\cup\Rdnocyc$, where
\begin{align*}
  \Rdcyc:=\big\{r^0\in\Rdamped: \exists (r^k)_{k=1}^K\subset\Rdamped, r^k_+=r^{k-1}_-, r^K_+=r^0_-\big\},
  &&
  \Rdnocyc:=\Rdamped\backslash\Rdcyc.
\end{align*}
The relation between damped cycles and spikes in the damped fluxes is summarised in our third main result:
\begin{theorem}\quad
\vspace{-0.2cm}
\begin{enumerate}[(i)]
\item For any sequence $(u^\e,j^\e)_{\e>0}\subset\Theta$ such that $\tilde\I^\e_0\big(u^\e(0)\big) + \tilde\J^\e(u^\e,j^\e)\leq C$ for some $C>0$ and
 $(u^\e,j^\e)\xrightarrow{\Theta}(u,j)$, we have  $j_\Rdnocyc \in L^\C([0,T];\RR^\Rdnocyc)$.
\item If $\Rdcyc \not=\emptyset$ then there exists a sequence $(u^\e,j^\e)_{\e>0}\subset\Theta$ with
$\tilde\I^\e_0\big(u^\e(0)\big) + \tilde\J^\e(u^\e,j^\e)\leq C$ for some $C>0$
and $(u^\e,j^\e)\xrightarrow{\Theta}(u,j)$ such that
\begin{equation*}
  j_\Rdcyc\in\M([0,T];\RR^{\Rdcyc})\backslash L^1([0,T];\RR^\Rdcyc).
\end{equation*}
\end{enumerate}
\label{th:spikes}
\end{theorem}

As a consequence of Theorem~\ref{th:main result} and Theorem~\ref{th:spikes}(ii), if $\Rdcyc\neq\emptyset$ then there is a $(u,j)\in\Theta$ with $j_\Rdcyc\in\M([0,T];\RR^{\Rdcyc})\backslash L^1([0,T];\RR^\Rdcyc)$ for which
$
  \tilde\I^0_0\big(u(0)\big) + \tilde\J^0(u,j)\leq \infty
$.

\subsection{Related literature}
\label{subsec:GF}

As mentioned in the introduction, this work is related to classical quasi-steady state approximation theory; see e.g.~\cite[Sec.~4.2]{HeinrichSchuster96} or~\cite[Sec.~3.1]{Kuehn15}. We mention two recent works~\cite{DisserLieroZinsl2018, MielkeStephan2019TR} that study fast-reaction limits in connection with another underlying structure, namely a gradient structure. A gradient structure consists of an energy $\tfrac12\I_0^\epsilon(c)$ and a non-negative convex dissipation potential $\Psi^\epsilon_c(\xi)$ such that the evolution equation~\eqref{eq:ODE} can be rewritten as $\Psi^\epsilon_{c(t)}(\dot c(t)) = -D\tfrac12\I^\epsilon_0(c(t))$. 
Both studies work on the level of concentrations rather than fluxes, under the assumption that the $\epsilon$-dependent evolution equation~\eqref{eq:ODE} satisfies detailed balance, and under the assumption that damped fluxes do not occur. The detailed balance condition is needed for the $\epsilon$-dependent equation to have a gradient structure, and the absence of damped fluxes guarantees that the gradient structure is not destroyed in the limit.

Disser, Liero, and Zinsl~\cite{DisserLieroZinsl2018} study general, possibly non-linear reaction networks with mass-action kinetics. Under the detailed balance assumption such equations have a gradient structure with quadratic dissipation potential, as discovered in \cite{Maas2011, Mielke2012a}. The authors show the convergence of that gradient structure by the notion of E-convergence as defined in~\cite{Mielke2016}. In order to do so they assume linearly independent stoichiometric coefficients, which can be seen as a decoupling or orthogonality between the slow and the fast reactions. In this paper we do not need such an assumption because the flux setting automatically decouples the reactions.

Mielke and Stephan~\cite{MielkeStephan2019TR} study the linear setting, similarly to the current paper. Contrary to Disser \emph{et al.}, they use the gradient structure that is related to the large-deviation principle~\eqref{eq:concentration RF} in the sense of \cite{MielkePeletierRenger2014}, again under the detailed balance assumption. They prove the convergence of that gradient structure, using the stronger notion of tilted EDP-convergence; see~\cite{Mielke2012a,LieroMielkePeletierRenger2017,MielkeMontefuscoPeletier20TR}. This result implies convergence of the large-deviation rate functions~$\I^\epsilon$, under the more restrictive assumptions mentioned above, but also for a wide range of tilted energies simultaneously. In a paper that is soon to appear, Mielke, Peletier, and Stephan generalise this to the case of nonlinear systems, modelled on the class of chemical reactions with mass-action kinetics that satisfy the detailed balance condition.


\subsection{Overview}

Section~\ref{sec:preliminaries} contains preliminaries that are needed throughout the paper. In Section~\ref{sec:network}, we study properties of the network, the continuity equations, and their limits, and we derive equicoercivity in $\Theta$. In Section~\ref{sec:Gamma convergence} we prove our main $\Gamma$-convergence result, Theorem~\ref{th:main result}. In Section~\ref{sec:spikes} we prove the relation between spikes and damped cycles, Theorem~\ref{th:spikes}. Finally, in Section~\ref{sec:density ldp and effective dynamics} we derive implications  for $\Gamma$-convergence of the density large deviations, and for convergence of solutions to the effective dynamics.

\section{Preliminaries}
\label{sec:preliminaries}

We first provide a list of basic  facts that will be used throughout the paper. After this we introduce the Orlicz space $L^\C$. Next we recall a FIR inequality that bounds the free energy and Fisher information by the rate functional which will be needed to derive compactness of densities later on. Finally, we state a number of convex dual formulations of a number of relevant functionals.

\subsection{Basic properties}

We will use the following properties of the functions $s(\cdot|\cdot)$ and $\C$. For any $a,b\geq0$ and $p\in\RR$, we have:
\begin{align}
   s(a\mid b) &:= a\log \frac ab-a+b \notag\\
      &=(1-\alpha) b + a\log \alpha + \sum_{n=2}^\infty \mfrac{1}{n(n-1)} b\alpha\Big(\mfrac{b\alpha-a}{b\alpha}\Big)^n 
      \geq (1-\alpha) b + a\log \alpha \quad\forall\alpha>0, \label{eq:s linearisation}\\
  s(a\mid b) &\leq \frac{\,a^2}{b} -2a +b \qquad\text{(using } \log x\leq x-1), \label{eq:s bdd quadr}\\
  \C(a)&:=\inf_{p-q=a} s(p \mid 1) + s(q \mid 1) \notag\\
      &= s\big(\tfrac12 a + \sqrt{1+a^2/4} \mid 1\big) + s\big(-\tfrac12 a + \sqrt{1+a^2/4} \mid 1\big)\notag\\
      &=\int_0^a\!\sinh^{-1}(\hat a/2)\,d\hat a =2\big(\cosh^*(a/2)+1\big),\notag\\
  \C^*(p) &:= \sup_{a\in\RR} p a-\C(a)\label{def:C*}\\
    &=2\big(\cosh(p)-1\big), \label{eq:dual form of C}\\
  s(a \mid b) &= b\big(s(a/b \mid 1) + s(1\mid 1)\big) \geq b\,\C(\tfrac{a-b}b) \label{eq:C bdd by s},\\
  \C(\delta p) &= \delta\int_0^p\sinh^{-1}(\delta q/2)\,dq \stackrel{\text{(concave)}}{\geq} \delta^2 \int_0^p\sinh^{-1}(q/2)\,dq = \delta^2\C(p) \qquad\forall\delta\in\lbrack0,1\rbrack,\label{eq:C sqrt scaling}\\
  \C(\delta p) &\stackrel{\text{(convex)}}{\leq} \delta\C(p)\hspace{22.7em}\forall\delta\in\lbrack0,1\rbrack.\label{eq:C linear scaling}
\end{align}

\subsection{Orlicz space}
\label{subsec:Orlicz}

The functions $\C,\C^*$ defined above form a convex dual pair of N-functions (``nice Young functions''~\cite[Sec.~1.3]{RaoRen1991}). The primal function $\C$ satisfies the $\Delta_2$ property: $\C(2p) \leq 4 \C(p)$ (but $\C^*$ does not). We shall use the corresponding  Orlicz space  (see \cite[Th.~3.3.13]{RaoRen1991}):
\begin{align}
  &L^\C([0,T];\RR^\R):=\Big\{ \textstyle j:[0,T]\to\RR^\R, \, \exists a>0 \text{ such that } \sum_{r \in \R} \int_{[0,T]}\!\C\big(\tfrac1a j_r(t)\big)\,dt<\infty\Big\}, \notag\\
  &\lVert j\rVert_{L^\C} := \sup_{\substack{\zeta\in L^{\!\C^*\!\!}([0,T];\RR^\R):\\\sum_{r\in \R} \int_{[0,T]}\!\C^*(\lvert \zeta_r(t)\rvert)\,dt\leq 1}} \, \int_{[0,T]}\! \lvert j(t)\cdot \zeta(t)\rvert \,dt 
    = \inf_{a>0} \textstyle \, \mfrac1a\!\left(1+\sum_{r\in \R}\int_{[0,T]}\!\C\big(a\, j_r(t)\big)\,dt\right)\!.
\label{eq:C Orlicz norm}
\end{align}
The final characterization above implies that 
\begin{equation}
\label{eq:est-Orlicz}
a\|j\|_{L^{\C}} \leq 1 + \sum_r \int_{[0,T]} \C\bigl(a\,j_r(t)\bigr)\,dt \qquad \text{for all }a>0.
\end{equation}
We also introduce the space (see \cite[Prop.~3.4.3]{RaoRen1991})
\begin{align*}
  M^{\C^*\!\!}([0,T];\RR^\R)&:=\Big\{\textstyle \zeta:[0,T]\to\RR^\R,\,\, \forall a>0 \text{ there holds } \sum_{r\in \R}\int_{[0,T]}\!\C^*\big(\tfrac1a\zeta_r(t)\big)\,dt<\infty\Big\}\\
    &= \overline{\mathrm{span}\Big\{ \text{step functions } \zeta\in L^{\C^*\!\!}([0,T];\RR^\R) \Big\}}^{L^{\C^*}} \!\! \subsetneq L^{\C^*\!\!}([0,T];\RR^\R).
\end{align*}
Then $\big(M^{\C^*\!\!}([0,T];\RR^\R)\big)^* \simeq L^\C([0,T];\RR^\R)$ \cite[Thms~4.1.6 \&~4.1.7]{RaoRen1991}, and, since $\C$ satisfies the $\Delta_2$-property, also $\big(L^\C([0,T];\RR^\R)\big)^*\simeq(L^{\C^*\!\!}([0,T];\RR^\R)$ \cite[Cor.~4.1.9]{RaoRen1991}. In particular, the first of these isomorphisms defines the weak-\textasteriskcentered\ topology on $L^\C([0,T];\RR^\R)$.


\subsection{An FIR inequality}


There are various related notions of Fisher information for discrete systems in the literature~\cite{BobkovTetali06,Maas2017,Fathi2018}. The notion that we use is:
\begin{equation}
  \FI^\epsilon(u):=\mfrac12\sum_{r\in\R}\int_{[0,T]}\!\kappa^\epsilon_r\pi^\epsilon_{r_-} \left( \sqrt{u_{r_-}\!(t)} - \sqrt{u_{r_+}\!(t)}\,\right)^2\,dt,
\label{eq:FI}
\end{equation}
where $\kappa^\epsilon_r\pi^\epsilon_{r_-} u_{r^+}(t)=\kappa^\epsilon_r \tfrac{\pi^\epsilon_{r_-}}{\pi^\epsilon_{r_+}}\rho_{r_+}(t)$ appears as the backward jump rate for the time-reversed process. 

Recall the definitions of $\tilde\I_0$ from~\eqref{def:I0} and $\tilde \J^\e$ from~\eqref{eq:flux RF rescaled}. Using arguments from Macroscopic Fluctuation Theory, one can show the following inequality, that is sometimes known as the FIR inequality in the literature~\cite{HilderPeletierSharmaTse2019TR,KaiserJackZimmer2018,RengerZimmer2019TR}:
\begin{lemma}[FIR inequality]
\label{l:FIR}
Let $(u^\epsilon_{\V_0},u^\epsilon_{\V_1},j^\epsilon_\Rslow,j^\epsilon_\Rdamped,\tjeps_\Rfcycle)\in \Theta$ be such that $\tilde \I_0^\e(u^\e(0)) + \J^\epsilon(u^\e\pi^\epsilon ,j^\e)< \infty$.
Then    
\begin{equation}
 \sup_{0\leq t\leq T} \tfrac12\tilde\I^\epsilon_0(u^\e(t))+ \FI^\epsilon(u^\e) \leq \tfrac12\tilde\I^\epsilon_0(u^\e(0))  + \J^\epsilon(u^\e\pi^\epsilon ,j^\e).
\label{eq:FIR inequality}
\end{equation}
\end{lemma}
The proof is a simple rewriting of the results of \cite{HilderPeletierSharmaTse2019TR}, \cite[Cor.~4]{KaiserJackZimmer2018} and \cite{RengerZimmer2019TR}, and we omit it. From the boundedness of $\tilde\I^\epsilon_0(u^\e(0)) + \J^\epsilon(u^\e\pi^\epsilon ,j^\e)$ assumed above, the inequality~\eqref{eq:FIR inequality} implies boundedness of both $\tilde\I^\epsilon_0(u^\e(T))$ and $\FI^\epsilon(u^\e\pi^\epsilon )$; this will be important in deducing compactness for the densities $u^\e_{\V_1}$.

\subsection{Dual formulations}
\label{subsec:dual formulations}

We recall convex dual formulations for the entropic and quadratic functionals and the Fisher information.

\begin{lemma}[{\cite[Prop.~3.5]{PattersonRenger2019},\cite[Lemma~9.4.4]{Ambrosio2008}}] If $u\in L^1([0,T])$,
\begin{equation*}
  \sup_{\zeta\in C([0,T])} \int_{[0,T]}\!\big\lbrack\zeta(t)j(t) - u(t)(e^{\zeta(t)}-1)\big\rbrack\,dt =
    \begin{cases}
      \int_{[0,T]} s(j(t) \mid u(t))\,dt, &j\in L^1([0,T]), j\ll u,\\
      \infty,                         &\text{otherwise},
    \end{cases}
\end{equation*}
and if $u\in\M([0,T])$,
\begin{equation*}
  \sup_{\zeta\in C([0,T])} \int_{[0,T]}\!\big\lbrack\zeta(t)j(dt) - u(dt)(e^{\zeta(t)}-1)\big\rbrack =
    \begin{cases}
      \int_{[0,T]} s(j \mid u)(dt),   &j\in\M([0,T]), j\ll u,\\
      \infty,                     &\text{otherwise}.
    \end{cases}
\end{equation*}
\label{lem:S dual}
\end{lemma}


\begin{lemma}[{\cite[Lemma~9.4.4]{Ambrosio2008}}] If $u\in L^1([0,T])$,
\begin{align*}
  \sup_{\zeta\in C([0,T])} \int_{[0,T]}\!\Big\lbrack\zeta(t)\tilde\jmath(t) - \tfrac12 u(t) \zeta(t)^2\Big\rbrack\,dt
    =
    \begin{cases}
      \tfrac12 \int_{[0,T]}\! \frac{\tilde\jmath(t)^2}{u(t)}\,dt, &\text{if }  \tilde\jmath\in L^2_{1/u}([0,T]), \tilde\jmath\ll u,\\
      \infty, &\text{otherwise}.
    \end{cases}
\end{align*}
\label{lem:square dual}
\end{lemma}

\begin{proposition} For $u\in L^1([0,T];\RR^\V)$,
\begin{align*}
  \FI^\epsilon(u) = \sup_{
    \substack{p\in C([0,T];\RR^{2\R}):\\p_{r_-}<1,\ p_{r_+}<1,\\(p_{r_-}-1)(p_{r_+}-1)>1}
    } 
    \sum_{r\in\R} \kappa^\epsilon_{r_-}\!\pi^\epsilon_{r_-}\! \int_{[0,T]}\!\kappa^\epsilon_{r_-}\!\pi^\epsilon_{r_-}\! \big\lbrack p_{r_-}\!(t)u_{r_-}\!(t) + p_{r_+}\!(t)u_{r_+}\!(t)\big\rbrack\,dt.
\end{align*}
\label{prop:Fisher dual}
\end{proposition}
\begin{proof}
Note that upon writing $\Fi(a,b)$ for the argument in the integral in~\eqref{eq:FI}, 
\[
\Fi(a,b)   :=(\sqrt{a}-\sqrt{b})^2
\]
we can characterize the function $\Fi$ by
\begin{align}
  \Fi(a,b)   &=\sup_{p,q} ap+bq - \Fi^*(p,q) = \sup_{p<1,q<1, (p-1)(q-1)>1} ap+bq
\notag\\
\intertext{where}
  \Fi^*(p,q) &:=\sup_{a,b\geq0} ap+bq - \Fi(a,b) = \chi\{p<1,q<1, (p-1)(q-1)>1\}.
\end{align}
We use this to write for $u\in L^1([0,T];\RR^\V_+)$,
\begin{align*}
  \FI^\epsilon(u)
    &= 
     \sup_{
        \substack{p\in\M([0,T];\RR^{2\R}):\\p_{r_-}<1,\ p_{r_+}<1,\\(p_{r_-}-1)(p_{r_+}-1)>1}
      } \,
    \sum_{r\in\R} \kappa^\epsilon_{r_-}\!\pi^\epsilon_{r_-}\! \int_{[0,T]}\!(p_{r_-}\!u_{r_-} + p_{r_+}\!u_{r_+})\,dt.
\end{align*}
After checking that a cut-off from below and a convolution leave the conditions invariant, the result follows by a standard approximation argument.
\end{proof}

\begin{remark} The definition of the Fisher information can easily be extended to measures if we use the dual formulation. In fact, the supremum remains finite when the measure is finite:
\begin{align*}
  &\sup_{
    \substack{p\in C([0,T];\RR^{2\R}):\\p_{r_-}<1,p_{r_+}<1,\\(p_{r_-}-1)(p_{r_+}-1)>1}
       } 
  \sum_{r\in\R} \kappa^\epsilon_{r_-}\!\pi^\epsilon_{r_-}\! \int_{[0,T]}\!\left\lbrack p_{r_-}\!(t)u_{r_-}\!(dt) + p_{r_+}\!(t)u_{r_+}\!(dt)\right\rbrack
  \\
  &\qquad=
  \sup_{
    \substack{q\in C([0,T];\RR^{2\R}):\\q_{r_-}>0,q_{r_+}>0,\\q_{r_-}q_{r_+}>1}
       } 
  \sum_{r\in\R} \kappa^\epsilon_{r_-}\!\pi^\epsilon_{r_-}\!\int_{[0,T]}\!\left\lbrack (1-q_{r_-}\!(t))u_{r_-}\!(dt) + (1-q_{r_+}\!(t))u_{r_+}\!(dt)\right\rbrack \\
  &\qquad
  \leq \sum_{r\in\R} \kappa^\epsilon_{r_-}\!\pi^\epsilon_{r_-}\!\big(\lVert u_{r_-}\rVert_\TV + \lVert u_{r_+}\!\rVert_\TV\big).
\end{align*}
This shows that a uniform bounded Fisher information does not rule out the development of singularities in the densities, as explained in Section~\ref{subsec:spikes}.
\end{remark}

\section{Network properties and compactness}
\label{sec:network}

In this section we study the network decomposition introduced in Sections~\ref{subsec:connected components}, \ref{subsec:network decomposition fluxes}, \ref{subsec:network decomposition nodes}, and in particular the implications for the continuity equation. We derive estimates for sublevel sets of the rate functional and deduce compactness of these sublevel sets in the topological space~$\Theta$ as defined in Section~\ref{sec:main results}. We then use that topology to derive the limiting continuity equations. In addition, we show that any sequence of bounded cost will equilibrate over the fast cycle components, and then prove a stronger equilibration result that will be needed in the construction of the recovery sequence in Section~\ref{sec:Gamma convergence}.

\subsection{Network properties and the continuity equations}
\label{subsec:network properties}

Recall that we assumed that any node $x$ is either in $\V_0$ (when $\pi^\epsilon_x=\bigoh(1)$) or in $\V_1$ (when $\pi^\epsilon_x=\bigoh(\epsilon)$), and that leak edges, through which the non-equilibrium steady state flux is of order $\epsilon$, do not occur. Moreover, we further decomposed $\V_0$ into $\V_{0\fcycle}$ and $\V_{0\slow}$, where $\V_{0\fcycle}$ is defined as all nodes $x\in \V_0$ such there is at least least one fast reaction that leaves $x$.

The name $\V_{0\fcycle}$ (`fast cycle') reflects the fact that all nodes in this set belong to a cycle of fast fluxes, as the following simple lemma shows: 
\begin{lemma}
  The subgraph $(\V_{0\fcycle},\Rfcycle)$ consists purely of cycles. More explicitly, let $x^1\in\V_{0\fcycle}$. Then there exists a cycle $(r^k)_{k=1}^K\subset\Rfcycle, r^k_+=r^{k+1}_-, r^1_-=x^1=r^K_+$. Similarly any $r\in\Rfcycle$ is part of such a fast cycle.
\label{lem:fast cycles}
\end{lemma}
\begin{proof}
Let $r^1\in\Rfcycle$ with $r^1_-=x^1$, which exists by assumption $x^0\in\V_{0\fcycle}$, and let $x^2:=r^1_+$. The equilibrium equation in $x^2$ reads:
\begin{equation*}
  \pi^\epsilon_{x^2} \sum_{r\in\R:r_-=x^2} \kappa^\epsilon_r = \sum_{r\in\R:r_+=x^2}\kappa^\epsilon_r\pi^\epsilon_{r_-} \geq \kappa^\epsilon_{r^1} \pi^\epsilon_{x^1}.
\end{equation*}
The right-hand side is of order $1/\epsilon$, and so for the left-hand side $\pi^\epsilon_{x^2}$ must be order $1$ (or higher, which is ruled out by assumption), and the sum contains at least one $r^2:=r\in\Rfast$. It follows that $x^2\in\V_{0\fcycle}$ and $r^2\in\Rfcycle$. We then repeat the same argument, which only terminates when $x^{K+1}=x^1$. The second claim is true by the same argument.
\end{proof}

We can then enumerate all possible edges from and to $\V_{0\slow}$, $\V_{0\fcycle}$, and $\V_1$.
\begin{lemma}\quad
\vspace{-0.2cm}
\begin{enumerate}[(i)]
\item If $x\in\V_{0\slow}$, then all incoming edges $r\in\R,r_+=x$ are either in $\Rslow$ or in $\Rdamped$, and all outgoing edges $r\in\R,r_+=x$ are in $\Rslow$.
\item If $x\in\V_{0\fcycle}$, then the incoming edges could be of any type, and all outgoing edges $r\in\R,r_-=x$ are either in $\Rslow$ or in $\Rfcycle$. 
\item If $x\in\V_1$, then all incoming fluxes $r\in\R,r_+=x$ are either in $\Rslow$ or $\Rdamped$, and all outgoing fluxes $r\in\R,r_-=x$ are in $\Rdamped$.
\end{enumerate}
\label{lem:node categorisation}
\end{lemma}
\begin{proof}
For $x\in\V_{0\slow}$ or $\V_{0\fcycle}$, the statement follows immediately from the definitions of $\Rslow,\Rdamped$ and $\Rfcycle$. For $x\in\V_1$ any slow outgoing edge will be of leak type that we ruled out by assumption and any fast outgoing edge is damped. Since all outgoing edges are of order $1$, an incoming fast cycle edge of order $1/\epsilon$ would imply that $\pi^\epsilon_x$ is of order $1/\epsilon$, which is ruled out by the conservation of mass.
\end{proof}


We can now write down the rescaled continuity equations. Although for $\epsilon>0$ all densities $u$ and fluxes $j$ and $\tilde\jmath$ have $W^{1,1}$ and $L^1$ regularity respectively, provided  the rate functional~\eqref{eq:flux RF} is finite, some of this regularity is lost in the regime $\epsilon\to0$. Therefore it will be useful to write the continuity equations in a different form. In the following we will say that
\begin{equation*}
  \pi_x^\epsilon \dot u^\epsilon_x = \sum_r j_r^\epsilon \qquad\text{in the weak sense},
\end{equation*}
whenever
\begin{equation}
  -\int_{[0,T]}\!\dot\phi(t)\,\pi_x^\epsilon u^\epsilon_x(dt) = \sum_r \int_{[0,T]}\!\phi(t) j^\epsilon_r(dt) \qquad\text{ for all } \phi\in C^1_0([0,T]),
\label{eq:weak sense}
\end{equation}
where we identify $u_x(dt)=u_x(t)\,dt$ and $j_r(dt)=j_r(t)\,dt$ wherever possible. If for a fixed $\epsilon>0$ we have $\tilde\I_0^\epsilon+\tilde\J^\epsilon<\infty$, then by~\eqref{eq:flux RF} we know that all densities are absolutely continuous and all fluxes have $L^1$-densities. We will then say that
\begin{equation*}
  \pi_x^\epsilon \dot u^\epsilon_x = \sum_r j^\epsilon_r \qquad\text{in the mild sense},
\end{equation*}
whenever for all $0\leq t_0\leq t_1\leq T$,
\begin{equation}
  \pi_x^\epsilon u^\epsilon_x(t_1) - \pi_x^\epsilon u^\epsilon_x(t_1) = \sum_r j^\epsilon_r[t_0,t_1],
\label{eq:mild sense}
\end{equation}
using the notation
\begin{equation*}
  j^\epsilon_r[t_0,t_1]:=
  \begin{cases}
    \int_{[t_0,t_1]}\!j^\epsilon_r(t)\,dt, & j^\epsilon_r\in L^1(\lbrack 0,T\rbrack),\\
    \int_{[t_0,t_1]}\!j^\epsilon_r(dt),    & j^\epsilon_r\in \M(\lbrack 0,T\rbrack).
  \end{cases}
\end{equation*}

\begin{corollary} 
\label{cor:rescaled-e-cont-eq}
After rescaling, the continuity equations \eqref{eq:flux ODE} and \eqref{eq:sum density eps} are, for fixed $\epsilon>0$, in the weak sense,

\begin{subequations}
\begin{align}
  \pi^\e_\fc u^\e_\fc &= \sum_{x\in\fc} \pi^\e_x u^\e_x
    &\text{for } \fc\in\fC,
    \label{eq:cont eq sum density eps}\\
  \pi^\epsilon_x \dot u^\epsilon_x = 
    &\sum_{\substack{r\in\Rslow:\\r_+=x}} j^\epsilon_r
      + 
    \sum_{\substack{r\in\Rdamped:\\r_+=x}} j^\epsilon_r
      -
    \sum_{\substack{r\in\Rslow:\\r_-=x}} j^\epsilon_r
      &\text{for } x\in\V_{0\slow},
\label{eq:cont eq V0s}\\
  \pi^\epsilon_x\dot u^\epsilon_x = 
    &\sum_{\substack{r\in\Rslow:\\r_+=x}} j^\epsilon_r
      + 
    \sum_{\substack{r\in\Rdamped:\\r_+=x}} j^\epsilon_r
      -
    \sum_{\substack{r\in\Rslow:\\r_-=x}} j^\epsilon_r\notag\\
      & +
    \sum_{\substack{r\in\Rfcycle:\\r_+=x}} \Big( \tfrac{1}{\epsilon}\kappa_r \pi^\epsilon_{r_-} u^\epsilon_{r_-} + \tfrac{1}{\sqrt{\epsilon}}\tjeps_r \Big) 
       -
    \sum_{\substack{r\in\Rfcycle:\\r_-=x}} \Big( \tfrac{1}{\epsilon}\kappa_r \pi^\epsilon_{r_-} u^\epsilon_{r_-} + \tfrac{1}{\sqrt{\epsilon}}\tjeps_r \Big),
    &\text{for } x\in\V_{0\fcycle},
\label{eq:cont eq V0f}\\
  \pi^\epsilon_x \dot u^\epsilon_x = &
    \sum_{\substack{r\in\Rslow:\\r_+=x}} j^\epsilon_r
      + 
    \sum_{\substack{r\in\Rdamped:\\r_+=x}} j^\epsilon_r -
    \sum_{\substack{r\in\Rdamped:\\r_-=x}} j^\epsilon_r ,
    &\text{for } x\in\V_1.
\label{eq:cont eq V1}
\end{align}
If in addition $\tilde\I^\e_0(u^\epsilon(0))+\tilde\J^\e(u^\e,j^\e)<\infty$, then these equations also hold in the mild sense of~\eqref{eq:mild sense}.
\label{eq:cont eq}
\end{subequations}

\end{corollary}

\subsection{Boundedness of densities and fluxes}
\label{subsec:estimates}

The aim of this section is to prove uniform bounds that are needed to derive the equicoercivity Theorem~\ref{th:equicoercivity} later on.

\begin{lemma}[Boundedness of densities]
\label{l:bdd-densities}
Let $(u^\e,j^\e)_{\epsilon>0}\subset\Theta$ such that
$\tilde\I^\epsilon_0\big(u^\e(0)\big) + \tilde\J^\epsilon(u^\e,j^\e)\leq C$
for some $C>0$. Then 
\begin{enumerate}
\item $(u^\e_{\V_{0\slow}},u^\e_\fC)$ and $u^\e_{\V_{0\fcycle}}$ are uniformly bounded in $C([0,T];\RR^{\V_{0\slow}\cup\fC})$ and $L^\infty([0,T];\RR^{\V_{0\fcycle}})$;
\item $(u^\e_{\V_{0\slow}},u^\e_{\V_{0\fcycle}},u^\e_{\V_1})$ is uniformly bounded in $L^1([0,T];\RR^{\V_{0\slow}\cup\V_{0\fcycle}\cup\V_1})$;
\item \label{l:e:eps-u-V1-to-zero}
$\e\|u_x^\e\|_{C([0,T])} \longrightarrow 0$ for all $x\in \V_1$ as $\e\to0$.
\end{enumerate}
\end{lemma}

\begin{proof}
From \eqref{eq:s linearisation} and mass conservation we derive a uniform bound on the total mass for each $t\in\lbrack0,T\rbrack$:
\begin{equation}
  \sum_{x\in\V} \pi^\epsilon_x u_x^\epsilon(t) = \sum_{x\in\V} \pi^\epsilon_x u_x^\epsilon(0) \leq \tilde\I^\e_0(u^\e(0)) + (e-1)\sum_x \pi_x^\epsilon\leq C + e-1.
\label{eq:Linf bound}
\end{equation}
This implies the $C$-bounds on $u^\e_{\V_{0\slow}},u^\e_\fC$, and the $L^\infty$ bound on $u^\e_{\V_{0\fcycle}}$.

From the FIR inequality~\eqref{eq:FIR inequality} we deduce that
\begin{align}
  C &\geq \tfrac12\tilde\I^\epsilon_0\big(u^\epsilon(0)\big) + \tilde\J^\epsilon\big(u^\epsilon_{\V_0},u^\epsilon_{\V_1},j^\epsilon_{\Rslow},j^\epsilon_{\Rdamped},\tjeps_{\Rfcycle}\big) \notag\\
    &\geq \FI^\epsilon(u^\epsilon)
  = \mfrac12\sum_{r\in\R}\kappa^\epsilon_{r}\pi^\epsilon_{r_-}\!\int_{[0,T]}\!\left( \sqrt{u^\epsilon_{r_-}\!(t)} - \sqrt{u^\epsilon_{r_+}\!(t)}\,\right)^2\,dt.
\label{eq:FIR application}
\end{align}
Hence by \eqref{eq:pi convergence}, for $\epsilon$ sufficiently small and any $r\in\R$:
\begin{equation*}
  \int_{[0,T]}\!\left( \sqrt{u^\epsilon_{r_-}\!(t)} - \sqrt{u^\epsilon_{r_+}\!(t)}\,\right)^2\,dt \leq
    \begin{cases}
      \frac{4C}{\kappa_r\pi_{r_-}},  &r\in\Rslow,\\
      \frac{4C}{\kappa_r\tilde\pi_{r_-}},  &r\in\Rdamped,\\
      \frac{4C}{\kappa_r\pi_{r_-}}\epsilon,  &r\in\Rfcycle.
    \end{cases}
\end{equation*}

Since $\V$ is finite, $\V_0$ cannot be empty, since otherwise the total mass in the system would vanish. Take an arbitrary $x^0\in\V_0$; by~\eqref{eq:Linf bound} we have $\lVert u^\epsilon_{x^0}\rVert_{L^\infty(0,T)}\leq 2(C+e-1)/\pi_{x^0}$ for sufficiently small $\epsilon$. Now take an arbitrary $y\in\V$. By irreducibility of the graph $(\V,\R)$ there exists a sequence of edges $x^0 \xrightarrow{r^{01}} x^1 \xrightarrow{r^{12}} \hdots\to x^n=y$. For the first edge we find, using the inequality $a\leq 2(\sqrt a - \sqrt b)^2 + 2b$, 
\begin{equation*}
  \int_{[0,T]}\!u^\epsilon_{x^1}(t)\,dt \leq 2\int_{[0,T]}\!\left( \sqrt{u^\epsilon_{x^0}(t)} - \sqrt{u^\epsilon_{x^1}(t)}\,\right)^2\,dt + 2\int_{[0,T]}\!u_{x^0}^\epsilon(t)\,dt
    \leq \frac{8C}{\kappa_{01}\pi_{x^0}} + 4T\frac{C+e-1}{\pi_{x^0}}.
\end{equation*}
Repeating this procedure for all edges yields that $u^\epsilon_y$ is uniformly bounded in $L^1(0,T)$.

Finally we prove the vanishing of $\e u_{\V_1}^\e$. We also  deduce from~\eqref{eq:FIR inequality} that for all $0\leq t\leq T$, 
\begin{multline*}
C \geq \frac12 \tilde\I_0^\e\bigl(u^\e(t)\bigr) \geq \frac12 \sum_{x\in \V_1} s(\e u^\e_x(t)\tilde \pi_x^\e \, | \, \e\tilde \pi^\e_x)
= \frac12 \sum_{x\in \V_1} \e \tilde \pi_x^\e \,s(u^\e_x(t) \, | \, 1)
\geq  \sum_{x\in \V_1} \e \tilde \pi_x^\e \, \eta\bigl(u_x^\e(t)\bigr), \\
 \text{with } 
\eta(\tau) := \begin{cases}
\tfrac12 \bigl[\tau\log \tau - \tau + 1\bigr]& \text{if } \tau\geq 1,\\
0 & \text{if }0\leq \tau \leq 1.
\end{cases}
\end{multline*}
Since the $\tilde \pi^\e_x$ are bounded away from zero, we find that 
\[
\eta\bigl(u_x^\e(t)\bigr) \leq \frac C\e \qquad\Longrightarrow \qquad
0\leq u_x^\e(t) \leq \eta^{-1}\Bigl(\frac C\e\Bigr),
\]
where $\eta^{-1}$ is the right-continuous generalized inverse of $\eta$. Since $\eta$ is superlinear at infinity, $\e\eta^{-1}(C/\e)\to0$ as $\e\to0$, and we find that $\e\|u_{\V_1}^\e\|_{C([0,T])} \longrightarrow 0$ as $\e\to0$. 
\end{proof}

\begin{lemma}[Boundedness of slow fluxes]
\label{lem:bound slow fluxes}
Let $(u^\e,j^\e)_{\epsilon>0}\subset\Theta$ such that
$\tilde\I^\epsilon_0\big(u^\e(0)\big) + \tilde\J^\epsilon(u^\e,j^\e)\leq C$
for some $C>0$. Then the slow fluxes $j^\epsilon_{\Rslow}$ are uniformly bounded in $L^\C([0,T];\RR^{\Rslow})$. It follows that there is a non-decreasing function $\omega:[0,\infty)\to[0,\infty)$ with $\lim_{\sigma\downarrow0}\omega(\sigma) = 0$ such that for all $0\leq t_0\leq t_1\leq T$, (using the notation from~\eqref{eq:mild sense})
\begin{equation}
\label{est:jslow}
\sup_{\e>0} \sum_{r\in \Rslow} j_r^\e([t_0,t_1])\leq \omega(t_1-t_0).
\end{equation}
\end{lemma}
\begin{proof}
Again by~\eqref{eq:flux RF} we know that $j^\epsilon_\Rslow$ and $\rho^\epsilon_{\V_0}$ both have $L^1$-densities. Writing $Z:=(C+e-1)\sum_{r\in\Rslow} \kappa_r$,
\begin{align}
  C \;&\!\!\stackrel{\eqref{eq:flux RF rescaled}}{\geq} \sum_{r\in\Rslow} \int_{[0,T]}\! \big\lbrack \underbrace{s\big( j_r^\epsilon(t) \mid \pi^\epsilon_{r_-}u^\epsilon_{r_-}\kappa_r) \big) - \pi^\epsilon_{r_-}u^\epsilon_{r_-}\!(t)\kappa_r}_{\text{non-increasing in } \pi^\epsilon_{r_-}u^\epsilon_{r_-}\!(t)\kappa_r} + \underbrace{\pi^\epsilon_{r_-}u^\epsilon_{r_-}\!(t)\kappa_r}_{\geq0}\big\rbrack\,dt \notag\\
    &\leftstackrel{\eqref{eq:Linf bound}}{\geq} \sum_{r\in\Rslow} \int_{[0,T]}\!\big\lbrack s\big(j_r^\epsilon(t) \mid Z \big) - Z\big\rbrack\,dt 
    \notag\\
    &\leftstackrel{\eqref{eq:C bdd by s}}{\geq}
    Z \sum_{r\in\Rslow} \int_{[0,T]}\!\C\big( \mfrac{j^\epsilon_r(t) - Z}{Z} \big)\,dt - Z \lvert\Rslow\rvert T \notag\\
    &\leftstackrel{\eqref{eq:est-Orlicz}}{\geq} \lVert j^\epsilon - Z \rVert_{L^\C([0,T];\RR^{\Rslow})} - Z (\lvert\Rslow\rvert T+1). \notag
\end{align}

The proof of estimate~\eqref{est:jslow} follows from the definition~\eqref{eq:C Orlicz norm} of the Orlicz norm and the superlinearity of $\C$. Define the function
\[
\omega:[0,\infty)\to[0,\infty), \qquad
\omega(\sigma):= \inf_{\beta>0}  \Bigl\{ \frac{\wt C }\beta:\  |\Rslow|\,\C^*(\beta) \leq \frac 1\sigma \Bigr\},
\]
where $\wt C$ is the bound on $j^\epsilon_{\Rslow}$  in $L^\C([0,T];\RR^{\Rslow}_+)$. The function $\omega$ is non-decreasing by construction, and $\lim_{\sigma\downarrow 0} \omega(\sigma) = 0$ because  $\C^*$ is finite on all of $\RR$.

Fix $0\leq t_0\leq t_1\leq T$ and take $\beta>0$ such that $(t_1-t_0)|\Rslow|\,\C^*(\beta)\leq 1$. Set 
\[
\zeta(\hat t) := 
\begin{cases}
  \mathds1_{[t_0,t_1]}(\hat t) &\text{if } r\in \Rslow,\\
0 & \text{otherwise},
\end{cases}
\]
and use this function $\zeta$ in~\eqref{eq:C Orlicz norm} to estimate,
\[
\sum_{r\in\Rslow} j_r^\e[t_0,t_1] 
= 
\sum_{r\in\Rslow} \frac1\beta \int_{[0,T]} j_r^\e(\hat t)\zeta(\hat t) \, d\hat t
\leq 
\frac1\beta \, \|j^\epsilon_{\Rslow}\|_{L^\C([0,T];\RR^{\Rslow})}
\leq 
\frac{\wt C}\beta .
\]
The estimate~\eqref{est:jslow} follows from taking the infimum over $\beta$. 
\end{proof}

Although the form of the rate functional is almost the same for the slow and damped fluxes, the damped fluxes lack an $C(\lbrack 0,T\rbrack)$-bound on the corresponding densities. Therefore we obtain a weaker bound on the damped fluxes:

\begin{lemma}[Boundedness of damped fluxes]
\label{l:bdd-damped-fluxes}
Let $(u^\e,j^\e)_{\epsilon>0}\subset\Theta$ such that
$\tilde\I^\epsilon_0\big(u^\e(0)\big) + \tilde\J^\epsilon(u^\e,j^\e)\leq C$
for some $C>0$. Then the damped fluxes $j^\e_{\Rdamped}$ are uniformly bounded in $L^1([0,T];\RR^{\Rdamped})$.
In addition, for all $\sigma>0$, 
\begin{equation}
\label{est:weak-AA}
\limsup_{\e\to0}\; \sup_{|t_1-t_0|<\sigma} \;\sum_{\substack{r\in \Rdamped:\\ r_+\in \V_{0}}} j_r^\e[t_0,t_1] \leq    \omega(\sigma),
\end{equation}
where $\omega$ is the modulus of continuity of Lemma~\ref{lem:bound slow fluxes}.
\end{lemma}

\begin{proof}
Again by~\eqref{eq:flux RF rescaled} we can assume that $u^\epsilon_{\V_1}$ and $j^\epsilon_\Rdamped$ have $L^1$-densities, at least for $\epsilon>0$. This allows us to write
\begin{align*}
  C &\stackrel{\eqref{eq:flux RF rescaled}}{\geq} \sum_{r\in\Rdamped} \int_{[0,T]}\!s\big(j^\epsilon_r(t)\mid \tfrac{1}{\epsilon}\kappa_r\pi^\epsilon_{r_-}\!u^\epsilon_{r_-}\!(t)\big)\,dt \\
    &\stackrel{\eqref{eq:s linearisation}}{\geq}  \sum_{r\in\Rdamped} \int_{[0,T]}\!\big((1-e)\tfrac{1}{\epsilon}\kappa_r\pi^\epsilon_{r_-}\!u^\epsilon_{r_-} + j^\epsilon_r(t)\big)\,dt,
\end{align*}
and so $\lVert j^\epsilon_\Rdamped\rVert_{L^1([0,T];\Rdamped)} \leq C +(e-1) \lVert u^\epsilon_{\V_1}\rVert_{L^1([0,T];\RR^{\V_1}_+)}\sup_{\epsilon>0,r\in\Rdamped}\tfrac1\epsilon\kappa_r\pi^\epsilon_{r-}$, which is uniformly bounded by Lemma~\ref{l:bdd-densities} and the assumption $\tfrac1\epsilon\pi^\epsilon_{r_-}\to\tilde\pi_{r_-}$.

Next we prove the estimate~\eqref{est:weak-AA} by summing the mild formulation of the continuity equations~\eqref{eq:cont eq V1} over all $x\in \V_1$, for arbitrary $0\leq t_0\leq t_1\leq T$:
\[
\sum_{\substack{r\in\Rdamped:\\r_-\in \V_1}} j_r^\e[t_0,t_1] 
\;-
\sum_{\substack{r\in\Rdamped:\\r_+\in \V_1}} j_r^\e[t_0,t_1] 
\;=  \sum_{\substack{r\in\Rslow:\\r_+\in \V_1}} j_r^\e[t_0,t_1] 
  - \sum_{x\in\V_1} \pi_x^\e \bigl(u_x^\e(t_1)-u_x^\e(t_0)\bigr).
\]
Since the first two sums have common terms corresponding to $r_-,r_+\in \V_1$, we can remove them to find
\[
\sum_{\substack{r\in\Rdamped:\\r_-\in \V_1\\r_+\in \V_0}} j_r^\e[t_1,t_0] 
\;- 
\sum_{\substack{r\in\Rdamped:\\r_-\in \V_0\\r_+\in \V_1}} j_r^\e[t_1,t_0] 
\;= \sum_{\substack{r\in\Rslow:\\r_+\in \V_1}} j_r^\e[t_1,t_0] 
  - \sum_{x\in\V_1} \pi_x^\e \bigl(u_x^\e(t_1)-u_x^\e(t_0)\bigr).
\]
The second sum is a sum over the empty set, and applying the estimate~\eqref{est:jslow}  we find 
\begin{equation*}
\sum_{\substack{r\in\Rdamped:\\r_-\in \V_1\\r_+\in \V_0}} j_r^\e[t_1,t_0] 
\;\leq 
\omega(t_1-t_0) + \sum_{x\in \V_1} \pi_x^\e \|u_x^\e\|_{C(\lbrack0,T\rbrack}.
\end{equation*}
The estimate~\eqref{est:weak-AA} then follows from part~\ref{l:e:eps-u-V1-to-zero} of Lemma~\ref{l:bdd-densities} together with $\tfrac1\epsilon\pi^\epsilon_{r_-}\to\tilde\pi_{r_-}$.
\end{proof}

\begin{lemma}[Boundedness of fast fluxes]
Let $(u^\e,j^\e)_{\epsilon>0}\subset\Theta$ such that
$\tilde\I^\epsilon_0\big(u^\e(0)\big) + \tilde\J^\epsilon(u^\e,j^\e)\leq C$
for some $C>0$. Then the fast cycle fluxes $\tjeps_{\Rfcycle}$ are uniformly bounded in $L^\C([0,T];\RR^{\Rfcycle})$.
\label{lem:bound fast fluxes}
\end{lemma}

\begin{proof} Similar to the proof of Lemma~\ref{lem:bound slow fluxes} we write $Z:=(C+e-1)\sum_{r\in\Rfcycle}\kappa_r$, so that $\kappa_r \pi^\epsilon_{r_-} u^\epsilon_{r_-}(t)/Z\leq 1$ for each $r\in\Rfcycle$ due to the total mass estimate~\eqref{eq:Linf bound}. Again using the existence of $L^1$-densities:
\begin{align*}
  C \;&\!\!\stackrel{\eqref{eq:flux RF rescaled}}{\geq}
    \sum_{r\in\Rfcycle} \int_{[0,T]}\!s\Big(\tfrac{1}{\epsilon}\kappa_r \pi^\epsilon_{r_-} u^\epsilon_{r_-}\!(t) + \mfrac{1}{\sqrt{\epsilon}}\tjeps_r(t)  \Bigm\vert \tfrac1\epsilon\kappa_r\pi^\epsilon_{r_-}u^\epsilon_{r_-}\!(t)\Big)\,dt \\
    &\leftstackrel{\eqref{eq:C bdd by s}}{\geq}
    \sum_{r\in\Rfcycle} \int_{[0,T]}\! \tfrac1\epsilon\kappa_r\pi^\epsilon_{r_-}u^\epsilon_{r_-}\!(t)\,\C\!\left( \mfrac{ \tfrac{1}{\sqrt{\epsilon}}\tjeps_r(t) }{ \tfrac1\epsilon\kappa_r\pi^\epsilon_{r_-}u^\epsilon_{r_-}\!(t)}\right)\,dt \\
    &\leftstackrel{\eqref{eq:C sqrt scaling}}{\geq}
    \sum_{r\in\Rfcycle} \int_{[0,T]}\! Z\frac{\kappa_r\pi^\epsilon_{r_-} u^\epsilon_{r_-}\!(t)}{Z}\,\C\!\left( \mfrac{ \tjeps_r(t)/Z }{ \kappa_r\pi^\epsilon_{r_-}u^\epsilon_{r_-}\!(t)/Z}\right)\,dt \\
    &\leftstackrel{\eqref{eq:C linear scaling}}{\geq}
    \sum_{r\in\Rfcycle} \int_{[0,T]}\! Z\,\C\!\left( \mfrac{ \tjeps_r(t) }{ Z}\right)\,dt \\
    &\leftstackrel{\eqref{eq:est-Orlicz}}{\geq} 
    \lVert \tjeps \rVert_{L^\C([0,T];\RR^\Rfcycle)} - Z.
\end{align*}

\end{proof}

%
%
%
%
%

\begin{lemma}
[Equicontinuity of $u^\e_{\V_{0\slow}}$ and $u^\e_\fC$.]
\label{l:equicontinuity}
Let $(u^\e_{\V_0},j^\e)_{\e>0}\subset\Theta$ such that
$\tilde\I^\e_0\big(u^\e(0)\big) + \tilde\J^\epsilon(u^\e,j^\e)\leq C$
for some $C>0$. Then there exists a continuous non-decreasing function $\overline \omega:[0,\infty)\to[0,\infty)$ with $\lim_{\sigma\downarrow0}\overline\omega(\sigma)=0$ such that for all $0\leq t_0\leq t_1\leq T$,
\begin{equation}
\label{est:eq-ct-total}
\limsup_{\e\to0} \sum_{\fc\in \fC} |u_\fc^\e(t_1) - u_\fc^\e(t_0)| + \sum_{x\in \V_{0\slow}} |u_x^\e(t_1) - u_x^\e(t_0)| \leq \overline\omega(t_1-t_0).
\end{equation}
\end{lemma}

\begin{proof}
Fix $0\leq t_0\leq t_1 \leq T$. 
Take $x\in \V_{0\slow}$ and note that by~\eqref{eq:cont eq V0s} and~\eqref{est:jslow}
\[
\pi_x^\e\bigl(u_x^\e(t_1)-u_x^\e(t_0)\bigr)
\geq -\sum_{ \substack{r\in\Rslow:\\r_-=x}} j_r^\e[t_1,t_0] \geq -\omega(t_1-t_0),
\]
where we again used the mild formulation of the continuity equations. To estimate the difference from the other side we write
\begin{align*}
\pi_x^\e\bigl(u_x^\e(t_1)-u_x^\e(t_0)\bigr)
\;&\leq\; 
\sum_{ \substack{r\in\Rslow:\\r_+=x}} j_r^\e[t_0,t_1]
+ \sum_{ \substack{r\in\Rdamped:\\r_+=x}} j_r^\e[t_0,t_1]
\;\leq \;\omega(t-s) + \sum_{ \substack{r\in\Rdamped:\\r_+=x}} j_r^\e[t_0,t_1], 
\end{align*}
and by~\eqref{est:weak-AA} one part of~\eqref{est:eq-ct-total} follows. 

The same line of reasoning leads to a corresponding statement about $|u_\fc^\e(t_1) - u_\fc^\e(t_0)|$ for any $\fc\in \fC$, after one sums the continuity equations~\eqref{eq:cont eq V0f} over all $x\in \fc$ to find
\[
\sum_{x\in \fc} \pi_x^\e \bigl(u_x^\e(t_1)-u_x^\e(t_0)\bigr) 
= \sum_{ \substack{r\in\Rslow:\\r_+\in \fc}} j_r^\e[t_0,t_1]
+ \sum_{ \substack{r\in\Rdamped:\\r_+\in \fc}} j_r^\e[t_0,t_1]
- \sum_{ \substack{r\in\Rslow:\\r_-\in \fc}} j_r^\e[t_0,t_1].
\]
We omit the details.
\end{proof}

\subsection{Compactness of densities and fluxes}
\label{subsec:compactness}

In this brief section we derive the compactness of level sets, and hence the equicoercivity of Theorem~\ref{th:equicoercivity}.

\begin{corollary} 
\label{c:compactness}
Let $(u^\e,j^\e)_{\epsilon>0}\subset\Theta$ such that
$\tilde\I^\epsilon_0\big(u^\e(0)\big) + \tilde\J^\e(u^\e,j^\e)\leq C$
for some $C>0$. Then one can choose a sequence $\e_n\to 0$ and a limit point $(u,j)\in\Theta$ such that 
\begin{subequations}
\label{eq:conv-subseq}
\begin{align}
  u^{\e_n}_{\V_{0\slow}} \longrightarrow u_{\V_{0\slow}}      &\qquad\text{in } C([0,T];\RR^{\V_{0\slow}}),
  \label{eq:conv-subseq-Vslow}
\\
  u^{\e_n}_{\V_{0\fcycle}} \weakstarto u_{\V_{0\fcycle}}            &\qquad\text{in } L^\infty([0,T];\RR^{\V_{0\fcycle}}),
  \label{eq:conv-subseq-V0f}
\\
  u^{\e_n}_{\fC} \longrightarrow u_{\fC}            &\qquad\text{in } C([0,T];\RR^{\fC}), 
  \label{eq:conv-subseq-fc}
\\
  u^{\e_n}_{\V_1} \narrowto u_{\V_1}               &\qquad\text{in } \M([0,T];\RR^{\V_1}),
  \label{eq:conv-subseq-V1}
\\
  \e \,u_{\V_1}^\e \longrightarrow 0 & \qquad \text{in }C([0,T];\RR^{\V_1}),
  \label{eq:conv-subseq-epsV1}
\\
  j^{\e_n}_{\Rslow} \weakstarto j_{\Rslow}  &\qquad\text{in } L^\C([0,T];\RR^\Rslow),  \label{eq:conv-subseq-jslow}
\\
  j^{\e_n}_{\Rdamped} \narrowto j_{\Rdamped} &\qquad\text{in } \M([0,T];\RR^\Rdamped),
  \label{eq:conv-subseq-jdamped}
\\
  \tilde\jmath^{\hspace{0.08em}\e_n}_{\Rfcycle} \weakstarto \tilde\jmath_{\Rfcycle}
                                                  &\qquad\text{in } L^\C([0,T];\RR^\Rfcycle).
\label{eq:conv-subseq-fcyc}
\end{align}
\end{subequations}
It follows that  $u_{\V_{0\slow}}$ and $u_\fC$ are continuous. 
\label{cor:compactness}
\end{corollary}


\begin{proof}
The boundedness given by Lemmas~\ref{l:bdd-densities}, \ref{lem:bound slow fluxes}, and~\ref{l:bdd-damped-fluxes} immediately implies the weak-\textasteriskcentered\ and narrow compactness of~\eqref{eq:conv-subseq-V0f}, \eqref{eq:conv-subseq-V1}, \eqref{eq:conv-subseq-epsV1}, \eqref{eq:conv-subseq-jslow}, \eqref{eq:conv-subseq-jdamped}, and~\eqref{eq:conv-subseq-fcyc}; we extract a subsequence that converges in this sense.  

The additional uniform convergences of~\eqref{eq:conv-subseq-Vslow} and~\eqref{eq:conv-subseq-fc} follow from an alternative version of the classical Arzel\`a-Ascoli theorem, which we state and prove in the appendix. This version applies to sequences that are uniformly bounded  and \emph{asymptotically uniformly equicontinuous}. The uniform boundedness of~$u^{\e_n}_{\V_{0\slow}}$ and $u^{\e_n}_\fC$ follow by Lemma~\ref{l:bdd-densities}, and the asymptotic uniform equicontinuity is the statement of Lemma~\ref{l:equicontinuity}.
The uniform convergences~\eqref{eq:conv-subseq-Vslow} and~\eqref{eq:conv-subseq-fc} then  follow by Theorem~\ref{th:mod-AA} (up to extraction of a subsequence).
\end{proof}
%
%

From now on we shall consider sequences that converge in the sense of \eqref{eq:conv-subseq}. 
%

\subsection{Equilibration on fast cycle components}

In this section we prove that all mass on fast cycles will instaneously spread over each node in the fast cycle component.

\begin{lemma}
Let $(u^\e,j^\e)_{\epsilon>0}\subset\Theta$ such that 
$\tilde\I^\epsilon_0\big(u^\e(0)\big) + \tilde\J^\epsilon(u^\e,j^\e)\leq C$ converge to $(u,j)$ in $\Theta$ in the sense of \eqref{eq:conv-subseq}. Then $u_x(t)\equiv u_\fc(t)$ on each component $\fc\in\fC$ and $\div\tilde\jmath_\Rfcycle=0$.
\label{lem:equilibration}
\end{lemma}
\begin{proof} For any $x\in\fc\in\fC$ and $t\in[0,T]$, the mild formulation of the continuity equation is:
\begin{align}
  \pi^\epsilon_x u^\epsilon_x(t) - \pi^\epsilon_x u^\epsilon_x(0) = 
    &\sum_{\substack{r\in\Rslow:\\r_+=x}} j^\epsilon_r[0,t]
      + 
    \sum_{\substack{r\in\Rdamped:\\r_+=x}} j^\epsilon_r[0,t]
      -
    \sum_{\substack{r\in\Rslow:\\r_-=x}} j^\epsilon_r[0,t] \notag\\
      & +
    \sum_{\substack{r\in\Rfcycle:\\r_+=x}} \tfrac{1}{\epsilon}\kappa_r \pi^\epsilon_{r_-}\! u^\epsilon_{r_-}[0,t] + \tfrac{1}{\sqrt{\epsilon}}\tjeps_r[0,t] \notag\\
      &-
    \sum_{\substack{r\in\Rfcycle:\\r_-=x}} \tfrac{1}{\epsilon}\kappa_r \pi^\epsilon_{r_-}\! u^\epsilon_{r_-}[0,t] + \tfrac{1}{\sqrt{\epsilon}}\tjeps_r[0,t].
  \label{eq:mild cont eq on fc}
\end{align}
All terms in the first line are uniformly bounded in $L^1(0,T)$, and the same holds for the $u^\epsilon$ and $\tilde\jmath$ in the second and third lines. First multiplying the equation by $\epsilon$, and then letting $\epsilon\to0$ thus yields:
\begin{equation*}
  \sum_{\substack{r\in\Rfcycle:\\r_+=x}}\kappa_r \pi_{r_-}\! u_{r_-}\!(t)
    =
  \sum_{\substack{r\in\Rfcycle:\\r_-=x}}\kappa_r \pi_{r_-}\! u_{r_-}\!(t).
\end{equation*}
Without the $u_{r_-}(t)$ factors, this is exactly the equation for the steady state $\pi$ for a network consisting only of the fast edges. Since the component $\fc$ containing $x$ is diconnected, this equation has a unique solution up to a multiplicative constant, i.e. $u_x(t)\equiv a_\fc(t)$ on $\fc$, for some $a_\fc\in L^\infty([0,T])$. To identify $a_\fc$, use \eqref{eq:cont eq sum density eps} together with the convergences~\eqref{eq:conv-subseq-V0f} and \eqref{eq:conv-subseq-fc} to find for the limit
\begin{equation*}
  \pi_\fc u_\fc \leftarrow \pi^\e_\fc u^\e_\fc = \sum_{x\in\fc} \pi^\e_x u^\e_x \rightharpoonup \sum_{x\in\fc} \pi_x u_x = \pi_\fc a_\fc, 
\end{equation*}
so that indeed $u_x\equiv a_\fc = u_\fc$.

The same argument, multiplying \eqref{eq:mild cont eq on fc} by $\sqrt{\epsilon}$ and letting $\epsilon\to0$, shows that
\begin{equation*}
  \sum_{\substack{r\in\Rfcycle:\\r_+=x}}\tilde\jmath_r(t)
    =
  \sum_{\substack{r\in\Rfcycle:\\r_-=x}}\tilde\jmath_r(t).
  \qedhere
\end{equation*}
\end{proof}

\begin{remark}
Alternatively, the fact that $u$ is constant on $\fc$ can also be seen from the FIR inequality~\eqref{eq:FIR application} together with the lower semicontinuity that follows from Proposition~\ref{prop:Fisher dual}. The FIR inequality can however not be used to make a similar statement about divergence-free fast fluxes.
\end{remark}

\subsection{The limiting continuity equations}

We again place ourselves in the setting of Section~\ref{subsec:compactness} and derive the continuity equations satisfied in the limit.

\begin{lemma}
Let $(u^\e,j^\e)_{\epsilon>0}\subset\Theta$ be such that 
$\tilde\I^\epsilon_0\big(u^\e(0)\big) + \tilde\J^\epsilon(u^\e,j^\e)\leq C$, and assume that $(u^\e,j^\e)$ converges to $(u,j)$ in the sense of \eqref{eq:conv-subseq}.
Then the limit satisfies the continuity equations
\begin{subequations}
\begin{align}
  \pi_x \dot u_x = 
    &\sum_{\substack{r\in\Rslow:\\r_+=x}} j_r
      + 
    \sum_{\substack{r\in\Rdamped:\\r_+=x}} j_r
      -
    \sum_{\substack{r\in\Rslow:\\r_-=x}} j_r,        &\text{for } x\in\V_{0\slow},
\label{eq:limit cont eq V0s}\\
  \pi_\fc \dot u_\fc = 
    &\sum_{\substack{r\in\Rslow:\\r_+\in\fc}} j_r
      + 
    \sum_{\substack{r\in\Rdamped:\\r_+\in\fc}} j_r
      -
    \sum_{\substack{r\in\Rslow:\\r_-\in\fc}} j_r, &\text{for } \fc\in\fC,
\label{eq:limit cont eq C}\\
    u_x=&u_\fc\hspace{1em}
     \text{and}\qquad
    \sum_{\substack{r\in\Rfcycle:\\r_+=x}}\tilde\jmath_r
      =
    \sum_{\substack{r\in\Rfcycle:\\r_-=x}}\tilde\jmath_r,            &\text{for } x\in\fc\in\fC,
\label{eq:limit cont eq V0f}\\
  0 = &
    \sum_{\substack{r\in\Rslow:\\r_+=x}} j_r
      + 
    \sum_{\substack{r\in\Rdamped:\\r_+=x}} j_r -
    \sum_{\substack{r\in\Rdamped:\\r_-=x}} j_r, &\text{for } x\in\V_1.
\label{eq:limit cont eq V1}
\end{align}
\label{eq:limit cont eq}
\end{subequations}
These equations hold in the sense of distributions on $[0,T]$ as in~\eqref{eq:weak sense}.
\label{lem:limit cont eq}
\end{lemma}

\begin{proof}
Equation~\eqref{eq:limit cont eq V0s} follows directly from equation~\eqref{eq:cont eq V0s} by the convergence properties of Corollary~\ref{c:compactness}. 
For fixed $\fc\in \fC$ we sum equation~\eqref{eq:cont eq V0f} over all $x\in \fc$ to find
\[
\sum_{x\in \fc} \pi_x^\e \dot u_x^\e
=   \sum_{\substack{r\in\Rslow:\\r_+\in\fc}} j_r^\e
      + 
    \sum_{\substack{r\in\Rdamped:\\r_+\in\fc}} j_r^\e
      -
    \sum_{\substack{r\in\Rslow:\\r_-\in\fc}} j_r^\e.
\]
Note that the final two sums in~\eqref{eq:cont eq V0f} cancel by Lemma~\ref{lem:fast cycles}. 
The left-hand side equals $\pi_\fc^\e \dot u_\fc^\e $ and converges in distributional sense by~\eqref{eq:conv-subseq-fc}; the remaining terms also converge by~\eqref{eq:conv-subseq-jslow} and~\eqref{eq:conv-subseq-jdamped}.
The limit equation is~\eqref{eq:limit cont eq C}.

Equation~\eqref{eq:limit cont eq V0f} is the content of Lemma~\ref{lem:equilibration}. Finally, to prove~\eqref{eq:limit cont eq V1} we write~\eqref{eq:cont eq V1} for $x\in \V_1$ as 
\[
  \e\tilde \pi^\epsilon_x \dot u^\epsilon_x = 
    \sum_{\substack{r\in\Rslow:\\r_+=x}} j^\epsilon_r
      + 
    \sum_{\substack{r\in\Rdamped:\\r_+=x}} j^\epsilon_r -
    \sum_{\substack{r\in\Rdamped:\\r_-=x}} j^\epsilon_r. 
\]
The left-hand side converges to zero in distributional sense by~\eqref{eq:conv-subseq-epsV1}, and the right-hand side again converges by~\eqref{eq:conv-subseq-jslow} and~\eqref{eq:conv-subseq-jdamped}.
\end{proof}

As an immediate consequence, the $\Gamma$-limit $\tilde\I^0_0+\tilde\J^0$ from Theorem~\ref{th:main result} can only be finite if these limit continuity equations~\eqref{eq:limit cont eq} hold.

Note that although the densities $u_{\V_1}$ do appear in the limit rate functional $\tilde\J^0_\damped$, they become decoupled from the other variables in the sense that they have vanished completely from the continuity equations. Furthermore, if one does not take fluxes into account, the mass flowing into a $\V_1$ node will be instantaneously distributed over the next nodes, which would lead to a contracted network as drawn on the right of Figure~\ref{fig:network2}. At the level of fluxes this is contraction is reflected in~\eqref{eq:limit cont eq V1}.

\begin{remark}
Note that $L^\infty([0,T]) \ni u_x \stackrel{\text{a.e.}}{=} u_\fc \in C([0,T])$, so that in general $u_x(0)\neq u_\fc(0)$; the mass that is initially present will be  spread out over the component $\fc$ at every positive time $t>0$, but not at $t=0$. The same principle can be seen seen in the strengthened equilibration in the next section, which only holds in the time interval $[t_0,T]$ for any $t_0>0$.
\label{rem:instaneous spread}
\end{remark}

\begin{remark}
If there are no damped cycles, as in Section~\ref{subsec:spikes}, then Lemmas~\ref{l:bdd-damped-fluxes} and ~\ref{lem:limit cont eq} show that $u_{\V_{0\slow}}\in W^{1,\C}([0,T];\RR^{\V_{0\slow}})$ and similarly $u_\fC \in W^{1,\C}([0,T];\RR^{\fC})$.
\label{rem:improved regularity u}
\end{remark}

\subsection{Strengthened equilibration on fast cycle components}
\label{subsec:strengthened equil}

In the previous sections we derived that for a sequence with uniformly bounded cost $\tilde\I_0^\e+\tilde\J^\e$, concentrations $u_x^\e$ in a fast cycle $x\in\fc\in\fC$ converge weakly-* in $L^\infty([0,T])$, whereas the weighted sum $u_\fc^\e$ converges uniformly in $C([0,T])$. We now show that the convergence of $u_x^\e$ can be strengthened to uniform convergence as well, as long as one does not include time $0$ in the interval. This result will be needed later on for the construction of the recovery sequence, see Section~\ref{subsec:Gamma rec seq}.

Recall from Section~\ref{subsec:estimates} that sequences with bounded cost have uniformly bounded fluxes in $L^1$. Together with the continuity equations, this will be the only requirement of the following result.

\begin{lemma}
Let $(u^\e,j^\e)_{\epsilon>0}$ in $\Theta$ such that each $(u^\e,j^\e)$ satisfy the continuity equations~\eqref{eq:cont eq}, and assume that all fluxes $j^\e_\Rslow,j^\e_\Rdamped,\tjeps_\Rfcycle$ are $L^1$-valued and uniformly bounded in $L^1(0,T;\RR^\Rslow)$, $L^1(0,T;\RR^\Rdamped)$ and $L^1(0,T;\RR^\Rfcycle)$, and that $u^\epsilon_\fC\to u_\fC$ in $C([0,T];\RR^\fC)$. Then for all $t_0>0$,
\begin{align*}
  u^\e_x \to u_\fc &&\text{strongly in } L^\infty([t_0,T]) &&\text{for each } x\in\fc\in\fC.
\end{align*}
If in addition,
\begin{align*}
 -\big(\div j^\e(t)\big)_x 
      &:=
    \sum_{\substack{r\in\Rslow:\\r_+=x}} j^\epsilon_r(t)
      + 
    \sum_{\substack{r\in\Rdamped:\\r_+=x}} j^\epsilon_r(t)
      -
    \sum_{\substack{r\in\Rslow:\\r_-=x}} j^\epsilon_r(t),
\intertext{and}
  -\mfrac1{\sqrt{\epsilon}}\big(\div \tjeps(t)\big)_x 
      &:=
    \mfrac1{\sqrt{\epsilon}}\sum_{\substack{r\in\Rfcycle:\\r_+=x}} \tjeps_r(t)
     -
    \mfrac1{\sqrt{\epsilon}}\sum_{\substack{r\in\Rfcycle:\\r_-=x}} \tjeps_r(t)
\end{align*}
are both uniformly bounded in $L^\infty(0,T;\RR^{\V_{0\fcycle}})$ and $u_x^\epsilon(0)=u_\fc^\epsilon(0)$ for each $x\in\fc\in\fC$, then  
\begin{align*}
  u^\e_x \to u_\fc &&\text{strongly in } L^\infty([0,T]) &&\text{for each } x\in\fc\in\fC.
\end{align*}
\label{lem:strengthened equilibration}
\end{lemma}

\begin{proof}
We prove the result for one fast cycle $\fc\in\fC$. To exploit the stochastic structure we temporarily write $\rho^\e_x(t):=\pi^\e_x u^\e_x(t)$, and 
\begin{align*}
  \big(A\tp \rho^\epsilon(t)\big)_x
      &:=
    \sum_{\substack{r\in\Rfcycle:\\r_+=x}} \kappa_r \rho^\epsilon_{r_-}(t)
     -
    \rho^\epsilon_x(t) \sum_{\substack{r\in\Rfcycle:\\r_-=x}} \kappa_r,
\end{align*}
so that $A$ is simply the generator matrix of the Markov chain that consists of the irreducible fast cycle $\fc$, which does not depend on $\e$.
Recall from~\eqref{eq:cont eq V0f} that for each $x\in\fc$: 
\begin{equation*}
  \dot\rho^\epsilon_x(t) = -\big(\div j^\e(t)\big)_x + \tfrac{1}{\epsilon} \big(A\tp \rho^\epsilon(t)\big)_x - \tfrac{1}{\sqrt{\epsilon}} \big(\div\tjeps(t)\big)_x.
\end{equation*}
The vector $\rho^\e(t)\in\RR^\fc$ can be orthogonally decomposed into $\rho^{\e,0}(t)\in \Null(A\tp)$ and $\rho^{\e,\bot} \in \Col(A)$. For the column space part we estimate:
\begin{align*}
  \frac{d}{dt}  \mfrac12 \lvert \rho^{\e,\bot}(t)\rvert^2_2 &= \rho^{\e,\bot}(t)\cdot \dot \rho^{\e,\bot}(t) = \rho^{\e,\bot}(t)\cdot \dot \rho^\e(t)\\
    &=-\rho^{\e,\bot}(t)\cdot \div j^\e(t) + \mfrac{1}{\epsilon} \rho^{\e,\bot}(t)\cdot A\tp \rho^{\epsilon,\bot}(t) - \mfrac{1}{\sqrt{\epsilon}} \rho^{\e,\bot}\cdot \div\tjeps(t)\\
    &\leq \lvert\div j^\e(t)\rvert_1 + \mfrac{1}{\sqrt{\epsilon}}\lvert \div\tjeps(t)\rvert_1 + \mfrac{2\lambda}{\e} \mfrac12\lvert \rho^{\e,\bot}\rvert^2_2,
\end{align*}
using $\lvert\rho^\e\rvert_1\leq1$ and Lemma~\ref{lem:definite generator} with largest negative eigenvalue $\lambda<0$ of $A$. By Gronwall:
\begin{equation}
  \mfrac12 \lvert \rho^{\e,\bot}(t)\rvert^2_2 \leq
    \Big({\textstyle
      \mfrac12 \lvert \rho^{\e,\bot}(0)\rvert^2_2
      + \int_0^t\lvert\div j^\e(s)\rvert_1\,ds + \tfrac{1}{\sqrt{\epsilon}}\int_0^t\lvert\div\tjeps(s)\rvert_1\,ds
    }\Big) e^{2\lambda t/\e}
\label{eq:Gronwall}
\end{equation}
Since the $L^1$-norms of the fluxes are uniformly bounded, $\rho^{\e,\bot}\to0$ strongly in $L^\infty([t_0,T];\RR^\fc)$.

We now focus on the other part $\rho^{\e,0}\in\Null(A\tp)$. Since the fast cycle $\fc$ is irreducible, $\Null(A\tp)=\mathrm{span}\{ (\pi_x)_{x\in\fc} \}$, so we may write $\rho^{\e,0}_x(t) = \pi_x a^\e(t)$ for some $a^\e(t)\in\RR$. Summing over the cycle gives
\begin{align*}
  \pi_\fc a^\e(t) = \sum_{x\in\fc} \rho^{\e,0}_x(t) = \sum_{x\in\fc}\big(\rho^\e_x(t) - \rho^{\e,\bot}_x(t)\big) = \pi^\e_\fc u^\e_\fc(t) - \sum_{x\in\fc} \rho^{\e,\bot}_x(t).
\end{align*}
By assumption the first term on the right-hand side converges uniformly to $\pi_\fc u_\fc$, and we just proved above that the second term vanishes uniformly on $[t_0,T]$. This implies that $a^\e\to u_\fc$ uniformly, and so $\pi^\e_x u^\e_x=\rho^\e_x=\pi_x a^{\e}+ \rho^{\e,\bot}_x \to \pi_x u_\fc$ uniformly on $[t_0,T]$.

Now assume that $ -\big(\div j^\e(t)\big)_x$ and $-\big(\div \tjeps(t)\big)_x/\sqrt{\epsilon}$ are uniformly bounded and $u^\epsilon_x(0)=u^\epsilon_\fc(0)$. In that case $\rho^{\epsilon,\bot}(0)=0$, and so \eqref{eq:Gronwall} becomes:
\begin{equation*}
  \mfrac12 \lvert \rho^{\e,\bot}(t)\rvert^2_2
    \leq
  \Big({\textstyle
     \lVert\div j^\e\rVert_{L^\infty} + \lVert \tfrac{1}{\sqrt{\epsilon}}\div\tjeps\rVert_{L^\infty}
  }\Big) te^{2\lambda t/\e}
    \leq
   -\Big({\textstyle
     \lVert\div j^\e\rVert_{L^\infty} + \lVert \tfrac{1}{\sqrt{\epsilon}}\div\tjeps\rVert_{L^\infty}
  }\Big)\mfrac{\epsilon}{2\lambda e},
\end{equation*}
showing that $\rho^{\e,\bot}\to0$ uniformly in $[0,T]$. The uniform convergence of $\rho^{\epsilon,0}_x=\pi^\epsilon_x a^\epsilon$ follows by the same argument as above.
\end{proof}

\section{$\Gamma$-convergence}
\label{sec:Gamma convergence}

This section is devoted to the proof of the main $\Gamma$-convergence Theorem~\ref{th:main result}, which consists of the lower bound, Proposition~\ref{prop:lower bound}, and the existence of a recovery sequence in Proposition~\ref{prop:rec seq}.

\subsection{$\Gamma$-Lower bounds}

The $\Gamma$-lower bound is summarised in the following.

\begin{proposition}[$\Gamma$-lower bound] For any sequence $(u^\e,j^\e)\to(u,j)$ in $\Theta$,
\begin{equation*}
  \liminf_{\e\to0} \tilde\I^\epsilon_0\big(u^\e(0)\big)+\tilde\J^\epsilon(u^\e,j^\e) \geq \tilde\I^0_0\big(u(0)\big)+\tilde\J^0(u,j).
\end{equation*}
\label{prop:lower bound}
\end{proposition}
\begin{proof}
We treat each functional $\tilde\I^\epsilon_0,\tilde\J^\epsilon_\slow$, $\tilde\J^\epsilon_\damped$ and $\tilde\J^\epsilon_\fcycle$ separately, and without loss of generality we may always assume that $\tilde\I^\epsilon_0+\tilde\J^\epsilon\leq C$ for some $C\geq0$ and hence the continuity equations~\eqref{eq:cont eq} hold; otherwise the lower bound is trivial. This is carried out in the next Lemmas~\ref{lem:Gamma lower bound initial}, \ref{lem:Gamma lower bound slow} and \ref{lem:Gamma lower bound fcycle}.
\end{proof}

For the initial condition, recall the definitions of $\tilde\I^\e_0$ and $\tilde\I_0$ from~\eqref{def:I0} and Section~\ref{sec:main results}, and observe that the first one depends on $u_{\V_{0\fcycle}}(0)$ whereas the second depends on $u_\fC(0)$, which may be different, see Remark~\ref{rem:instaneous spread}. Hence the $\Gamma$-convergence of $\tilde\I^\e_0$ to $\tilde\I_0$ does not hold in $\RR^{\V_{0\slow}}\times\RR^{\V_{0\fcycle}}\times\RR^\fC\times\RR^{\V_1}$, but only in the path-space convergence of~\eqref{eq:conv-subseq}.

\begin{lemma}[$\Gamma$-lower bound for the initial condition]
Let $(u^\e,j^\e)_{\epsilon>0}\subset\Theta$ such that 
$\tilde\I^\epsilon_0\big(u^\e(0)\big) + \tilde\J^\epsilon(u^\e,j^\e)\leq C$ converge to $(u,j)\in\Theta$ in the sense of \eqref{eq:conv-subseq}. Then:
\begin{equation*}
  \liminf_{\epsilon\to0}\, \tilde\I^\epsilon_0\big(u^\epsilon(0)\big) \geq \tilde\I^0_0\big(u(0)\big).
\end{equation*}
\label{lem:Gamma lower bound initial}
\end{lemma}
\begin{proof}
By uniform convergence, $u^\epsilon_{\V_{0\slow}}(0)\to u_{\V_{0\slow}}(0)$ so that
\begin{equation*}
  \sum_{x\in\V_{0\slow}} s\big(\pi^\epsilon_x u^\epsilon_x(0)\mid \pi^\epsilon_x \big) \to \sum_{x\in\V_{0\slow}} s\big(\pi_x u_x(0)\mid \pi_x \big),
\end{equation*}
and clearly
\begin{equation*}
  \sum_{x\in\V_1} s\big(\pi^\epsilon_x u^\epsilon_x(0)\mid \pi^\epsilon_x \big)\geq0.
\end{equation*}
Lemma~\ref{lem:fast cycles} shows that every $x\in\V_{0\fcycle}$ is part of exactly one component $\fc\in\fC$.
From \eqref{eq:pic}, Jensen's inequality and the continuity equation~\eqref{eq:cont eq sum density eps},
\begin{align*}
  \sum_{x\in\V_{0\fcycle}} s\big(\pi^\epsilon_x u^\epsilon_x(0)\mid \pi^\epsilon_x \big)
    &= \sum_{\fc\in\fC}  \pi^\epsilon_\fc \sum_{x\in \fc} s\big(u^\epsilon_x(0)\mid 1 \big) \frac{\pi^\epsilon_x}{\pi^\epsilon_\fc}\\
    &\geq \sum_{\fc\in\fC}  \pi^\epsilon_\fc  s\big( {\textstyle \sum_{x\in \fc} u^\epsilon_x(0) \tfrac{\pi^\epsilon_x}{\pi^\epsilon_\fc}   } \mid 1 \big) \\
    &= \sum_{\fc\in\fC}  s\big(\pi^\epsilon_\fc u^\epsilon_\fc(0) \mid \pi^\epsilon_\fc \big) \to \sum_{\fc\in\fC}  s\big(\pi_\fc u_\fc(0) \mid \pi_\fc \big),
\end{align*}
again by uniform convergence of $u^\e_\fC$.
\end{proof}

\begin{lemma}[$\Gamma$-lower bound for the slow and damped fluxes] 

Let $(u^\e,j^\e)_{\epsilon>0}\subset\Theta$ such that 
$\tilde\I^\epsilon_0\big(u^\e(0)\big) + \tilde\J^\epsilon(u^\e,j^\e)\leq C$ converge to $(u,j)\in\Theta$ in the sense of \eqref{eq:conv-subseq}. Then:
\begin{align*}
  \liminf_{\epsilon\to0}\, \tilde\J^\epsilon_\slow(u^\epsilon_{\V_{0\slow}},u^\epsilon_{\V_{0\fcycle}},j^\epsilon_\Rslow) \geq \tilde\J^0_\slow(u_{\V_{0\slow}},u_{\V_{0\fcycle}},j_\Rslow),
\intertext{and}
  \liminf_{\epsilon\to0}\, \tilde\J^\epsilon_\damped(u^\epsilon_{\V_1},j^\epsilon_\Rdamped) \geq \tilde\J^0_\damped(u_{\V_1},j_\Rdamped).
\end{align*}
\label{lem:Gamma lower bound slow}
\end{lemma}

\begin{proof}
Recall the uniform $L^1$-bounds on the slow and damped fluxes from Lemmas~\ref{lem:bound slow fluxes} and~\ref{l:bdd-damped-fluxes}. The statement for slow fluxes follows directly from rewriting
\begin{equation}
  \tilde\J^0_\slow(u^\epsilon_{\V_0},j^\epsilon_\Rslow) = \sum_{r\in\Rslow} \int_{[0,T]}\!\Big\lbrack s\Big(j^\epsilon_r(t)\Bigm| \kappa_r \pi_{r_-}\!u^\epsilon_{r_-}\!(t)\Big) + j^\epsilon_r(t)\log\mfrac{\pi_{r_-}}{\pi^\epsilon_{r_-}} -\pi^\epsilon_{r_-} + \pi_{r_-} \Big\rbrack\,dt,
\label{eq:J0slow change pi}
\end{equation}
together with the joint lower semicontinuity from Lemma~\ref{lem:S dual}, and $\pi^\epsilon\to\pi>0$. The argument for the damped fluxes is the same after generalising to possible measure-valued trajectories in time.
\end{proof}

\begin{lemma}[$\Gamma$-lower bound for the fast cycle fluxes]
Let $(u^\e,j^\e)_{\epsilon>0}\subset\Theta$ such that 
$\tilde\I^\epsilon_0\big(u^\e(0)\big) + \tilde\J^\epsilon(u^\e,j^\e)\leq C$ converge to $(u,j)\in\Theta$ in the sense of \eqref{eq:conv-subseq}. Then:
\begin{equation*}
  \liminf_{\epsilon\to0}\, \tilde\J^\epsilon_\fcycle(u^\epsilon_{\V_{0\fcycle}},\tjeps_\Rfcycle) \geq \tilde\J^0_\fcycle(u_{\V_{0\fcycle}},\tilde\jmath_\Rfcycle).
\end{equation*}
\label{lem:Gamma lower bound fcycle}
\end{lemma}

A similar statement is proven in \cite[Th.~2]{BonaschiPeletier16}.

\begin{proof}
To simplify notation we prove the statement for one arbitrary $r\in\Rfcycle$. We first note that
\begin{align*}
  \pi^\epsilon_{r_-}\!u^\epsilon_{r_-} \weakto \pi_{r_-}\!u^\epsilon_{r_-} \quad\text{in } L^1([0,T]) 
  \qquad\text{and}\qquad
  \sup_{\epsilon>0} \, \lVert \pi^\epsilon_{r_-}\!u^\epsilon_{r_-} \rVert_{L^1}<\infty,
\end{align*}
and that for any test function $\zeta\in C([0,T])$,
\begin{align*}
  \tfrac1\epsilon e^{\sqrt{\epsilon}\zeta}-\tfrac1\epsilon-\tfrac1{\sqrt{\epsilon}}\zeta \xrightarrow[\epsilon\to0]{L^\infty} \tfrac12\zeta^2.
\end{align*}
It then follows that the following integral converges:
\begin{align*}
  \int_{[0,T]}\!\kappa_r\pi^\epsilon_{r_-}\!u^\epsilon_{r_-}\!(t)\big(\tfrac1\epsilon e^{\sqrt{\epsilon}\zeta(t)}-\tfrac1\epsilon-\tfrac1{\sqrt{\epsilon}}\zeta(t)\big)\,dt
    \to
  \tfrac12 \int_{[0,T]}\!\kappa_r\pi_{r_-}\!u_{r_-}\!(t)\zeta(t)^2\,dt.
\end{align*}
Using the dual formulations of Lemmas~\ref{lem:S dual} and \ref{lem:square dual},
\begin{align*}
  &\liminf_{\epsilon\to0}\int_{[0,T]}\!s\Big(\tfrac{1}{\epsilon}\kappa_r \pi^\epsilon_{r_-}\!u^\epsilon_{r_-}\!(t) + \mfrac{1}{\sqrt{\epsilon}}\tjeps_r(t)  \Big| \tfrac1\epsilon\kappa_r\pi^\epsilon_{r_-}\!u^\epsilon_{r_-}\!(t)\Big)\\
  &\qquad \geq
  \sup_{\zeta\in C([0,T])} \liminf_{\epsilon\to0} \int_{[0,T]}\!\Big\lbrack\zeta(t)\tjeps_r(t)
  -
  \kappa_r\pi^\epsilon_{r_-}\!u^\epsilon_{r_-}\!(t)\big(\tfrac1\epsilon e^{\sqrt{\epsilon}\zeta(t)}-\tfrac1\epsilon-\tfrac1{\sqrt{\epsilon}}\zeta(t)\big)\Big\rbrack\,dt\\
  &\qquad =
  \sup_{\zeta\in C([0,T])} \int_{[0,T]}\!\Big\lbrack\zeta(t)\tilde\jmath_r(t) - \tfrac12\kappa_r\pi_{r_-}\!u_{r_-}\!(t)\zeta(t)^2\Big\rbrack\,dt\\
  &\qquad =
    \begin{cases}
    \displaystyle
      \tfrac12 \int_{[0,T]}\! \frac{\tilde\jmath_r(t)^2}{\kappa_r\pi_{r_-}\!u_{r_-}\!(t)}\,dt, &\text{if } \tilde\jmath_r\in L^2_{1/\kappa_r\pi_{r_-}\!u_{r_-}}\!\!([0,T]),\\
      \infty, &\text{otherwise}.
    \end{cases}
\end{align*}
\end{proof}

\subsection{$\Gamma$-recovery sequence}
\label{subsec:Gamma rec seq}

For each of the four functionals separately, convergence is easily shown using a constant sequence $(u^\e,j^\e)\equiv(u,j)$. However, such a constant sequence is not a valid recovery sequence as it violates the continuity equations~\eqref{eq:cont eq}. The construction of the recovery sequence is summarised in the following proposition.

\begin{proposition}[$\Gamma$-recovery sequence] For any $(u,j)$ in $\Theta$ there exists a sequence $(u^\e,j^\e)_\e\subset\Theta$ such that $(u^\e,j^\e)\to(u,j)$ in $\Theta$ and 
\begin{equation*}
  \lim_{\e\to0} \tilde\I^\epsilon_0\big(u^\e(0)\big)+\tilde\J^\epsilon(u^\e,j^\e) = \tilde\I^0_0\big(u(0)\big)+\tilde\J^0(u,j).
\end{equation*}
\label{prop:rec seq}
\end{proposition}
\begin{proof}
In Lemma~\ref{lem:regularise uj} we first show that $(u,j)$ can be approximated by a regularised $(u^\delta,j^\delta)$ such that the limit functional converges, i.e.\ such that $\tilde\I^0\big(u^\delta(0))+\tilde\J^0(u^\delta,j^\delta)\to\tilde\I^0\big(u(0))+\tilde\J^0(u,j)$ as $\delta\to0$. In Lemma~\ref{lem:rec seq regularised} we construct a recovery sequence $(u^\epsilon,j^\epsilon)$ corresponding to such regularised $(u^\delta,j^\delta)$, and then use a diagonal argument to construct a recovery sequence for arbitrary $(u,j)$, see for example~\cite[Prop.~6.2]{Duong2013a}. 
\end{proof}

\begin{remark} So far, we only assumed $\sum_{x\in\V}\pi^\epsilon_x=1$, whereas the total mass $\sum_{x\in\V} \pi^\epsilon_x u^\epsilon_x(t)$ is only bounded above by~\eqref{eq:Linf bound}. All arguments in this paper can be extended to the case where the total mass is fixed. In that case the construction of the recovery sequence becomes slightly more involved, since adding mass to certain nodes must be balanced by subtracting mass from other nodes.
\end{remark}

\begin{lemma}[Approximation of the limit functional]
Let $(u,j)\in\Theta$ such that $\tilde\I_0^0(u(0))+\tilde\J^0(u,j)<\infty$, so $(u,j)$ satisfies the limit continuity equations~\eqref{eq:limit cont eq}. Then there exists a sequence $(u^\delta,j^\delta)_{\delta>0} \in \Theta$ such that for each $\delta>0$,
\begin{enumerate}
\item $(u^\delta,j^\delta)\in C_b^\infty\big([0,T];\RR^{\V_{0\slow}}\times\RR^{\V_{0\fcycle}}\times\RR^\fC\times\RR^{\V_1}\times\RR^{\R_\slow}\times\RR^{\R_\damped}\times \RR^{\R_\fcycle}\big)$, \label{it:smoothed uj}
\item $(u^\delta,j^\delta)$ satisfies the limit continuity equations~\eqref{eq:limit cont eq}, \label{it:conteq uj}
\item $\inf_{t\in[0,T]} u^\delta_x(t)>0$ for all $x\in\V_{0\slow}\cup\V_{0\fcycle}\cup\fC\cup\V_1$, \label{it:mass bdd below}
\item $j^\delta_r \geq \sum_{x\in\V_1} \delta \tilde \pi_x \lVert \dot u^\delta_x \rVert_{L^\infty}$ for all $r\in\Rslow\cup\Rdamped$, \label{it:added uj}
\enumeratext{and as $\delta\to0$,}
\item $(u^\delta,j^\delta) \to (u,j)$ in $\Theta$,
\item $\tilde\I_0^0(u^\delta(0))+\tilde\J^0(u^\delta,j^\delta)\to \tilde\I_0^0(u(0))+\tilde\J^0(u,j)$.
\end{enumerate}
\label{lem:regularise uj}
\end{lemma}
\begin{proof}
We construct the approximation in three steps.

\paragraph{Step 1: convolution.}
Note that for each $x\in \V_0$ the concentration $t\mapsto u_x(t)$ is continuous; for $x\in \V_{0\slow}$ this follows from the definition of $\Theta$, and for $x\in \V_{0\fcycle}$ this follows from the continuity of $t\mapsto u_{\fC}(t)$ in $\Theta$ and the continuity equation~\eqref{eq:limit cont eq V0f}. We first extend $u_{\V_0}$ beyond $[0,T]$ by constants, and  $u_{\V_1}$ and $j$ by zero. Observe  that with this extension the pair $(u,j)$ satisfies the continuity equation~\eqref{eq:limit cont eq} in the sense of distributions on the whole time interval $\RR$ (which is a stronger statement than the usual interpretation~\eqref{eq:weak sense}). We then approximate $(u,j)$ by convoluting with the heat kernel: $(u^\delta,j^\delta):=(u*\theta^\delta,j*\theta^\delta)$, where $\theta^\delta(t):=(4\pi\delta)^{-1/2} e^{-t^2/(4\delta)}$. Since $(u,j)$ satisfies the linear continuity equations~\eqref{eq:limit cont eq} in the sense of distributions on~$\RR$, they are also satisfied for the convolution $(u^\delta,j^\delta)$.

It is easily checked that $(u^\delta,j^\delta)\big|_{[0,T]} \to (u,j)\big|_{[0,T]}$ in $\Theta$. The initial conditions $u_x^\delta(0)$ converge for $x\in \V_{0\slow}\cup \fC$ and so by continuity $\tilde\I_0^0(u^\delta(0))\to\tilde\I_0^0(u(0))$.  The bound $\liminf_{\delta\to0}\tilde\J^0(u^\delta,j^\delta)\geq \tilde\J^0(u,j)$ is for free because of  lower semicontinuity (see Section~\ref{subsec:dual formulations}). The bound in the other direction is obtained by exploiting the joint convexity of $(u,j)\mapsto\tilde\J^0(u,j)$ and applying Jensen's inequality to the probability measure $\theta^\delta$; see~\cite[Lem.~3.12]{Renger2018a}. 

\paragraph{Step 2: add constants to the densities.}
For the next step we further approximate the sequence $(u^\delta,j^\delta)$, but to reduce clutter we now assume that the procedure above is already applied so that we are given a smooth and bounded $(u,j)$. We make all densities positive by adding a constant $\delta>0$, i.e.\
\[
u_x^\delta (t) := u_x(t) + \delta \qquad \text{for }0\leq t\leq T, \ x\in \V.
\]
It follows automatically that  $u_\fc^\delta = u_\fc + \delta$.
We leave the fluxes $j$ invariant, and the resulting pair $(u^\delta,j^\delta)$ again satisfies the limiting continuity equations~\eqref{eq:limit cont eq}.
The following lemma shows that the limit functional $\tilde\I_0^0+\tilde\J^0$ converges along the sequence $(u^\delta,j^\delta)$.

\begin{lemma}
\label{l:approx-delta}
	Let $a,b\in \mathcal M_{\geq0}([0,T])$ satisfy 
	\[
	\int_{[0,T]} s(a|b)\,(dt) < \infty.
	\]
Then setting $b^\delta(dt)  := b(dt) + \delta dt $ we have
\[
\lim_{\delta\to0} \int_{[0,T]} s(a|b^\delta)\,(dt)  = \int_{[0,T]} s(a|b)\,(dt).
\]
\end{lemma}
\begin{proof}[Proof of Lemma~\ref{l:approx-delta}]
We write
\[
s(a|b^\delta)(dt) = a(dt) \log \frac{\mathrm d a}{\mathrm d b^\delta}(t) - a(dt) + b^\delta(dt).
\]
After integration over $[0,T$] the final term $b^\delta([0,T])$ converges to $b([0,T])$ as $\delta\to0$;  in the first term the argument of the logarithm is decreasing in $\delta$, and therefore the first term converges by the Monotone Convergence Theorem.
\end{proof}

\paragraph{Step 3: add constant fluxes.}
Again to reduce clutter we may assume that we are given an $(u,j)$ satisfying properties~\ref{it:smoothed uj}, \ref{it:conteq uj}, and \ref{it:mass bdd below} of the Lemma.
By irreducibility of the network there exists a cycle $(r^k)_{k=1}^K\subset\Rfcycle, r^k_+=r^{k+1}_-, r^1_-=x^1=r^K_+$, such that each damped flux $r\in\Rdamped$ is contained in the cycle at least once. Note that some fluxes $r$ may occur multiple times, namely $n(r):=\#\{k=1,\hdots,K: r^k=r\}$ times in the cycle. For each $k=1,\hdots,K$ we define the new approximation:
\begin{equation*}
\begin{cases}
  j^\delta_{r^k} := j_{r^k} + n(r^k)\sum_{y\in\V_1} \delta \tilde \pi_y \lVert \dot u_y \rVert_{L^\infty}, &r^k\in\Rdamped\cup\Rslow,\\
  \tjdel_{r^k}:=\tilde\jmath_{r^k} + \sqrt{\e} n(r^k) \sum_{y\in\V_1} \delta \tilde \pi_y \lVert \dot u_y \rVert_{L^\infty}, &r^k\in \Rfcycle.
\end{cases}
\end{equation*}
Substituting these modified fluxes into the limit continuity equations~\eqref{eq:limit cont eq} shows that the concentrations are left unchanged, since some extra mass is being pushed around in cycles. Since the fluxes are only changed by adding a constant, it is easily checked that $(u,j^\delta) \to (u,j)$ in $\Theta$, and by Lemma~\ref{l:approx-delta} we find $\tilde\J^0(u,j^\delta)\to \tilde\J^0(u,j)$ as $\delta\to 0$.
\end{proof}

We now construct a recovery sequence $(u^\epsilon,j^\epsilon)$ for a $(u,j)\in\Theta$ that is regularised by Lemma~\ref{lem:regularise uj}. The difficulty is to construct the sequence such that the continuity equations hold in the $\V_1$ and $\V_{0\fcycle}$ nodes. The problem with the $\V_1$ nodes is that the continuity equations~\eqref{eq:cont eq V1} and \eqref{eq:limit cont eq V1} are different, but $u^\epsilon_{\V_1}$ needs to converge to $u_{\V_1}$. This will be done by transporting exactly the right amount of mass from certain $\V_0$-nodes to the $\V_1$-nodes. To satisfy the continuity equations in the $\V_{0\fcycle}$ nodes, we define $u^\epsilon_{V_{0\fcycle}}$ through the continuity equations, and use the strengthened convergence result of Section~\ref{subsec:strengthened equil} to pass to the limit. 

\begin{lemma}[Recovery sequence for regularised paths] Let $(u,j)\in\Theta$ satisfy properties \ref{it:smoothed uj}, \ref{it:conteq uj}, \ref{it:mass bdd below} and \ref{it:added uj} of Lemma~\ref{lem:regularise uj}. Then there exists a sequence $(u^\epsilon,j^\epsilon)\in\Theta$ such that:
\begin{enumerate}
\item $(u^\epsilon,j^\epsilon)$ satisfies the $\epsilon$-dependent continuity equations~\eqref{eq:cont eq};
\item $(u^\epsilon_{\V_{0\slow}}, u^\epsilon_{\R_\fcycle}, u^\epsilon_\fC, u^\epsilon_{\V_1}, j^\epsilon_\Rslow, j^\epsilon_\Rdamped, \tjeps_\Rfcycle) \to (u_{\V_{0\slow}}, u_{\R_\fcycle}, u_\fC, u_{\V_1}, j_\Rslow, j_\Rdamped, \tilde\jmath_\Rfcycle)$ uniformly on $[0,T]$;
\item $\tilde\I_0^\epsilon\big(u^\epsilon(0)\big)\to \tilde\I_0^0\big(u(0)\big)$ and $\tilde\J^\epsilon(u^\epsilon,j^\epsilon) \to   \tilde\J^0(u,j)$.
\end{enumerate}
\label{lem:rec seq regularised}
\end{lemma}

\begin{proof}
For ease of notation we pick only one node $\hat x\in\V_{0\slow}$, whose density is bounded from below by assumption. We will approximate all fluxes such that a little mass is transported from node $\hat x$ to all $\V_1$-nodes, as follows. Since the network is irreducible, there exists, for each $y\in\V_1$, a connecting chain $Q(\hat x,y):=(r^{k,y})_{k=1}^{K_y}\subset\R$, $r^{k,y}_+=r^{k+1,y}_-, r^{1,y}_-=\hat y, $ and $r^{K_y}=y$. For these connecting chains we may assume without loss of generality that no $r\in\R$ occurs multiple times in a chain $Q(\hat x,y)$. Define for all $r\in\R$:
\begin{equation*}
\begin{cases}
   j^\e_r := j_r + \sum_{y\in\V_1:r\in Q(\hat x,y)}\,  \pi_y^\e \dot u_y, &r\in\Rslow\cup\Rdamped,\\
  \tjeps_r := \tilde\jmath_r + \tfrac1{\sqrt{\e}}\sum_{y\in\V_1:r\in Q(\hat x,y)}\,  \pi_y^\e \dot u_y, &r\in\Rfcycle.
\end{cases}
\end{equation*}
Note that by the assumed properties~\ref{it:smoothed uj} and \ref{it:mass bdd below} of Lemma~\ref{lem:regularise uj} together with $\pi^\epsilon_{\V_1}\to0$, all approximated fluxes $j^\e_r,\tjeps_r$ are non-negative for $\epsilon$ small enough. Clearly all fluxes $j^\epsilon$ converge uniformly to $j$, since $\pi^\e_{\V_1}/\sqrt{\e}\to0$. For the initial conditions, set
\begin{align}
  u^\e_x(0)&:=\tfrac{\pi^\epsilon_x}{\pi_x}u_x(0),                          &&\text{ for all } x\in\V_{0\slow},\notag\\
  u^\e_\fc(0)= u^\e_x(0)&:= u_x(0)=u_\fc(0), &&\text{ for all } x\in\fc\in\fC, \label{eq:rec seq ic c}\\ 
  u^\e_x(0)&:=u_x(0),                             &&\text{ for all } x\in\V_1, \notag
\end{align}
and define the paths $u^\e$ by the continuity equations~\eqref{eq:cont eq}.

More precisely, by construction for $x\in V_{0\slow}$:
\begin{align*}
  \pi^\epsilon_x u^\epsilon_x(t) 
      &\stackrel{\eqref{eq:cont eq V0s}}{:=}
    \pi_x u_x(0)
      +
    \sum_{\substack{r\in\Rslow:\\r_+=x}} j^\epsilon_r[0,t]
      + 
    \sum_{\substack{r\in\Rdamped:\\r_+=x}} j^\epsilon_r[0,t]
      -
    \sum_{\substack{r\in\Rslow:\\r_-=x}} j^\epsilon_r[0,t]\\
  &=
  \pi_x u_x(t) - \mathds1_{\{x=\hat x\}} \sum_{y\in \V_1} \pi^\epsilon_y u_y(t),
\end{align*}
which is bounded away from zero (for $\e$ small enough) by the assumed properties \ref{it:smoothed uj} and \ref{it:mass bdd below} of Lemma~\ref{lem:regularise uj} together with $\pi^\epsilon_{\V_1}\to0$. Clearly $u^\e_x\to u_x$ uniformly.

For $x\in\V_1$, the densities will be constant in $\epsilon$, since:
\begin{align*}
  \pi^\epsilon_x u^\epsilon_x(t) 
    &\stackrel{\eqref{eq:cont eq V1}}{:=} \pi^\epsilon_x u_x(0)
      +
    \sum_{\substack{r\in\Rslow:\\r_+=x}} j^\epsilon_r[0,t]
      + 
    \sum_{\substack{r\in\Rdamped:\\r_+=x}} j^\epsilon_r[0,t] -
    \sum_{\substack{r\in\Rdamped:\\r_-=x}} j^\epsilon_r[0,t]
      =
    \pi^\epsilon_x u_x(t).
\end{align*}
For $x\in\fc\in\fC$, the density $u^\e_x(t)$ is defined as the solution of the coupled equations:
\begin{align}
  \pi^\epsilon_x\dot u^\epsilon_x
    &\stackrel{\eqref{eq:cont eq V0f}}{:=} 
  \sum_{\substack{r\in\Rslow:\\r_+=x}} j^\epsilon_r
    + 
  \sum_{\substack{r\in\Rdamped:\\r_+=x}} j^\epsilon_r
    -
  \sum_{\substack{r\in\Rslow:\\r_-=x}} j^\epsilon_r\notag\\
     &\hspace{7em}+
  \sum_{\substack{r\in\Rfcycle:\\r_+=x}} \Big( \tfrac{1}{\epsilon}\kappa_r \pi^\epsilon_{r_-} u^\epsilon_{r_-} + \tfrac{1}{\sqrt{\epsilon}}\tjeps_r \Big) 
     -
  \sum_{\substack{r\in\Rfcycle:\\r_-=x}} \Big( \tfrac{1}{\epsilon}\kappa_r \pi^\epsilon_{r_-} u^\epsilon_{r_-} + \tfrac{1}{\sqrt{\epsilon}}\tjeps_r \Big)
\label{eq:rec seq uC}
\end{align}
with initial condition~\eqref{eq:rec seq ic c}. Summing over $x\in\fc$ yields:
\begin{align*}
  \pi^\epsilon_\fc \dot u^\epsilon_\fc &\stackrel{\eqref{eq:cont eq sum density eps}}{:=} 
\sum_{x\in\fc} \pi^\epsilon_x\dot u^\epsilon_x
    = 
  \sum_{\substack{r\in\Rslow:\\r_+\in\fc  }} j^\epsilon_r
    + 
  \sum_{\substack{r\in\Rdamped:\\r_+\in\fc}} j^\epsilon_r
    -
  \sum_{\substack{r\in\Rslow:\\r_-\in\fc}} j^\epsilon_r\notag\\
     &\hspace{3em}+
  \sum_{\substack{r\in\Rfcycle:\\r_+\in\fc}} \Big( \tfrac{1}{\epsilon}\kappa_r \pi^\epsilon_{r_-} u^\epsilon_{r_-} + \tfrac{1}{\sqrt{\epsilon}}\tjeps_r \Big) 
     -
  \sum_{\substack{r\in\Rfcycle:\\r_-\in\fc}} \Big( \tfrac{1}{\epsilon}\kappa_r \pi^\epsilon_{r_-} u^\epsilon_{r_-} + \tfrac{1}{\sqrt{\epsilon}}\tjeps_r \Big)\\
    &\hspace{-1.3em}\stackrel{\eqref{eq:limit cont eq C},\eqref{eq:limit cont eq V0f}}{=}\pi_\fc\dot u_\fc.
\end{align*}
Together with the initial condition~\eqref{eq:rec seq ic c} this shows that $u_\fc^\epsilon \to u_\fc$ uniformly. Since all fluxes are uniformly bounded (and actually $\div\tjeps\equiv0$) and $u^\e_x(0)=u^\e_\fc(0)$ for $x\in\fc\in\fC$ we can apply Lemma~\ref{lem:strengthened equilibration} to \eqref{eq:rec seq uC} to derive that $u^\epsilon_x\to u_\fc$ uniformly on $[0,T]$ for all $x\in\fc$. Thus indeed all variables $(u^\epsilon_{\V_{0\slow}}, u^\epsilon_{\R_\fcycle}, u^\epsilon_\fC, u^\epsilon_{\V_1}, j^\epsilon_\Rslow, j^\epsilon_\Rdamped, \tjeps_\Rfcycle) \to (u_{\V_{0\slow}}, u_{\R_\fcycle}, u_\fC, u_{\V_1}, j_\Rslow, j_\Rdamped, \tilde\jmath_\Rfcycle)$ uniformly, which was to be shown. 

To show convergence of $\tilde\I^\epsilon(u^\epsilon(0))$,
\begin{align*}
  \tilde\I^\epsilon_0\big(u^\epsilon(0)\big) &= \sum_{x\in\V_{0\slow}} \pi^\epsilon_x s\big(\tfrac{\pi^\epsilon_x}{\pi_x} u_x(0)\mid 1 \big) + \sum_{\fc\in\fC}\sum_{x\in\fc} \pi^\epsilon_x s\big(u_\fc(0)\mid 1\big) + \sum_{x\in\V_1} \pi^\epsilon_x s\big(u_x(0)\mid 1\big)\\
  &\to \sum_{x\in\V_{0\slow}} \pi_x s\big( u_x(0)\mid 1 \big) + \sum_{\fc\in\fC} \pi^\fc s\big(u_\fc(0)\mid 1\big)
  = \tilde\I^0_0\big(u(0)\big).
\end{align*}
To show convergence of $\tilde\J^\epsilon(u^\epsilon,j^\epsilon)$, we use the fact that all fluxes and densities are uniformly bounded, that is for $\epsilon$ sufficiently small and all $t\in[0,T]$,
\begin{align*}
  0\leq &j^\epsilon_r(t) \leq 2\lVert j_r\rVert_{L^\infty}<\infty,                             &&r\in\Rslow\cup\Rdamped,\\
  0\leq &\tjeps_r(t) \leq 2\lVert \tilde j_r\rVert_{L^\infty}<\infty,                          &&r\in\Rfcycle,\\
  0<\tfrac12\inf_{t\in[0,T]} u_x(t) \leq &u^\epsilon_x(t) \leq 2\lVert u_x\rVert_{L^\infty},   &&x\in\V_{0\slow}\cup\V_{0\fcycle}\cup\fC\cup\V_1.
\end{align*}
The convergence of the integrals for $r\in\Rslow$ and $r\in\Rdamped$ then follows by dominated convergence:
\begin{align*}
 \int_{[0,T]}\!s\big(j^\epsilon_r(t)\mid \kappa_r\pi^\epsilon_{r_-} u^\epsilon_{r_-}(t)\big)\,dt \to \int_{[0,T]}\!s\big(j_r(t)\mid \kappa_r\pi_{r_-} u_{r_-}(t)\big)\,dt, &&r\in\Rslow,\\
 \int_{[0,T]}\!s\big(j^\epsilon_r(t)\mid \tfrac1\epsilon\kappa_r\pi^\epsilon_{r_-} u^\epsilon_{r_-}(t)\big)\,dt \to \int_{[0,T]}\!s\big(j_r(t)\mid \kappa_r\tilde\pi_{r_-} u_{r_-}(t)\big)\,dt, &&r\in\Rdamped.
\end{align*}
Similarly for $r\in\Rfcycle$, by dominated convergence,
\begin{align*}
  \int_{[0,T]}\!s\Big(\tfrac{1}{\epsilon}\kappa_r \pi^\epsilon_{r_-}\!u^\epsilon_{r_-}\!(t) + \mfrac{1}{\sqrt{\epsilon}}\tjeps_r(t)  \Bigm| \tfrac1\epsilon\kappa_r\pi^\epsilon_{r_-}\!u^\epsilon_{r_-}(t)\Big)\,dt
    \stackrel{\eqref{eq:s bdd quadr}}{\leq}
  \int_{[0,T]}\! \frac{\tjeps_r(t)^2}{\kappa_r\pi^\epsilon_{r_-}\!u^\epsilon_{r_-}(t)}\,dt
    \to
  \int_{[0,T]}\! \frac{\tilde\jmath_r(t)^2}{\kappa_r\pi^\epsilon_{r_-}\!u_{r_-}(t)}\,dt.
\end{align*}
The inequality in the other direction follows from Lemma~\ref{lem:Gamma lower bound fcycle}.
\end{proof}

\section{Spikes and damped cycles}
\label{sec:spikes}

As explained in Section~\ref{subsec:spikes}, the uniform $L^1$-bounds on the damped fluxes $j^\epsilon_{\Rdamped}$ and small concentrations $u^\epsilon_{\V_1}$ can not prevent limits from becoming measure-valued in time, that is, both may develop atomic or Cantor parts. The question when these spikes in damped fluxes may occur is answered in our Theorem~\ref{th:spikes}; this section is devoted to the proof of both statements in that theorem. The first part of Theorem~\ref{th:spikes} rules out spikes for damped fluxes that are not chained in a cycle. The second part shows that spikes may occur in damped flux cycles.

Recall the subdivision $\Rdamped=\Rdcyc\cup\Rdnocyc$ from Section~\ref{subsec:spikes}

\subsection{No spikes in damped fluxes outside cycles}

\begin{proof}[{Proof of Theorem~\ref{th:spikes}(i)}]
For this argument we again work with the fluxes in integrated form $j^\epsilon[0,t]$. Since $j^\epsilon[0,0]=0$,
\begin{align}
  j^\epsilon_\Rslow[0,\cdot] &\to j_\Rslow[0,\cdot] \text{ in } L^1([0,T];\RR^\Rslow) 
    \qquad\text{ and} \notag\\
  j^\epsilon_\Rdamped[0,\cdot] &\to j_\Rdamped[0,\cdot] \text{ in } L^1([0,T];\RR^\Rdamped).
\label{eq:damped cycle L1 conv}
\end{align}
Take an arbitrary $r^0\in\Rdnocyc$ coming out of node $r^0_-=:x\in\V_1$. By Lemma~\ref{lem:node categorisation}, all fluxes flowing out of node $x$ are damped, and all fluxes flowing into node $x$ are either slow or damped. The mild formulation of the continuity equation in $x$ now reads:
\begin{equation*}
  \sum_{\substack{r^1\in\Rdamped:\\ r^1_-=x}} j^\epsilon_{r^1}[0,t] - \sum_{\substack{r^1\in\Rslow:\\r^1_+=x}} j^\epsilon_{r^1}[0,t]
  - \sum_{\substack{r^1\in\Rdamped:\\r^1_+=x}}j^\epsilon_{r^1}[0,t]
  = -\pi^\epsilon_x u^\epsilon_x(t) + \pi^\epsilon_x u^\epsilon_x(0).
\end{equation*}
By the uniform $L^1$-bounds on $u^\epsilon$ and the vanishing $\pi^\epsilon_x u^\epsilon_x(0)$, the right-hand side above converges to zero in $L^1(0,T)$, and so by \eqref{eq:damped cycle L1 conv}:
\begin{equation*}
  \sum_{\substack{r^1\in\Rdamped:\\ r^1_-=x}} j_{r^1}[0,t] - \sum_{\substack{r^1\in\Rslow:\\r^1_+=x}} j_{r^1}[0,t]
  - \sum_{\substack{r^1\in\Rdamped:\\r^1_+=x}} j_{r^1}[0,t]
  = 0.
\end{equation*}
Therefore, by the uniqueness of derivatives of functions of bounded variation (for arbitrary sets $dt$),
\begin{equation*}
  0\leq j_{r^0}(dt)\leq\sum_{\substack{r^1\in\Rdamped:\\ r^1_-=r^0_-(=x)}} j_{r^1}(dt) = \sum_{\substack{r^1\in\Rslow:\\r^1_+=r^0_-}} j_{r^1}(t)\,dt + \sum_{\substack{r^1\in\Rdamped:\\r^1_+=r^0_-}}j_{r^1}(dt).
\end{equation*}
Applying the same inequality for each damped flux $r^1_+=r^0_-$, we get:
\begin{equation*}
  0\leq j_{r^0}(dt)\leq \sum_{\substack{r^1\in\Rslow:\\r^1_+=r^0_-}} j_{r^1}(t)\,dt + \sum_{\substack{r^1\in\Rdamped:\\r^1_+=r^0_-}}
    \Big\lbrack
      \sum_{\substack{r^2\in\Rslow:\\r^2_+=r^1_-}} j_{r^2}(t)\,dt + \sum_{\substack{r^2\in\Rdamped:\\r^2_+=r^1_-}}j_{r^2}(dt)
    \Big\rbrack.
\end{equation*}
We now apply this procedure recursively until the right-hand side contains slow fluxes only. This is possible because by assumption any damped flux that already appeared during this procedure can not reappear in the inequality. Exploiting that eventually the right-hand side is a sum over slow fluxes that are in $L^\C(0,T)$, by the Radon-Nikodym Lemma the left-hand $j_{r^0}$ also has a $L^\C$-density.
\end{proof}

\subsection{Finite-cost spikes in damped flux cycles}

We now prove that fluxes in $\Rdcyc$ may actually develop singularities.

\begin{proof}[{Proof of Theorem~\ref{th:spikes}(ii)}]
If $\Rdcyc\neq\emptyset$ then there exists a diconnected damped component $\fd\subset\V_1$ such that $\forall x,y \in\fd \quad \exists (r^k)_{k=1}^K\subset\Rdcyc, r^1_-=x, r^k_+=r^{k+1}_-,r^K_+=y$ (cf. Section~\ref{subsec:connected components}). By irreducibility and mass conservation there exists at least one $r^\mathrm{in}\in\Rslow\cup\Rdnocyc$ with $r^\mathrm{in}_+\in\fd$ and at least one $r^\mathrm{out}\in\Rdnocyc$ with $r^\mathrm{out}_-\in\fd$. We first assume 1) that all edges in $\fd$ are chained in a cycle, i.e. $\fd:=(x^k)_{k=1}^K$, $\Rdcyc\cap\{r_-\in\fd\}=(r^k)_{k=1}^K$ with $r^k_-=x^k, r^k_+=x^{k+1}, r^K_+=x^1$, 2) that  $r^\mathrm{in}_+=x^1$ and $r^\mathrm{out}_-=x^l$, and 3) that $x^0:=r^\mathrm{in}_-$ and $x^{K+1}:=r^\mathrm{out}_+$ both lie in $\V_0$, see Figure~\ref{fig:damped cycle}.

\begin{figure}[h!]
  \centering
  \begin{tikzpicture}[scale=0.6,baseline]
    \tikzstyle{every node}=[font=\scriptsize]
    \node[label=left:$x^1$](x1) at (0,0) [circle,draw,fill=gray]{};
    \node                (x2) at (2,1) [circle,draw,fill=gray]{}; 
    \node[label=right:$x^l$](x3) at (4,0) [circle,draw,fill=gray]{};
    \node                (x4) at (2,-1) [circle,draw,fill=gray]{};
    \node[label=left:$x^0$](x5) at (-1,-3) [circle,draw,fill=black]{};
    \node[label=right:$x^{K+1}$](x6) at (5,-3) [circle,draw,fill=black]{};

    \draw[-{Latex},double distance=1.5,dashed] (x1) to node[midway,sloped,anchor=south]{$r^1$} (x2);
      \begin{scope}
        \path[clip] (x2) -- (x1) -- (x1.center) -- (x2.center)-- (x2);
        \draw[-{Latex},double distance=1.5] (x1) to (x2);
      \end{scope}
      \draw[{-{Latex}[black]},white,line width=1pt] (x1) to (x2);
    \draw[-{Latex},double distance=1.5,dashed] (x2) to (x3);
      \begin{scope}
        \path[clip] (x3) -- (x2) -- (x2.center) -- (x3.center)-- (x3);
        \draw[-{Latex},double distance=1.5] (x2) to (x3);
      \end{scope}
      \draw[{-{Latex}[black]},white,line width=1pt] (x2) to (x3);
    \draw[-{Latex},double distance=1.5,dashed] (x3) to (x4);
      \begin{scope}
        \path[clip] (x4) -- (x3) -- (x3.center) -- (x4.center)-- (x4);
        \draw[-{Latex},double distance=1.5] (x3) to (x4);
      \end{scope}
      \draw[{-{Latex}[black]},white,line width=1pt] (x3) to (x4);
    \draw[-{Latex},double distance=1.5,dashed] (x4) to node[midway,sloped,anchor=north]{$r^K$}(x1);
      \begin{scope}
        \path[clip] (x1) -- (x4) -- (x4.center) -- (x1.center)-- (x1);
        \draw[-{Latex},double distance=1.5] (x4) to (x1);
      \end{scope}
      \draw[{-{Latex}[black]},white,line width=1pt] (x4) to (x1);

    \draw[-{Latex},thick] (x5) --node[midway,sloped,anchor=south]{$r^\mathrm{in}$} (x1);
    \draw[-{Latex},double distance=1.5,dashed] (x3) to node[midway,sloped,anchor=south]{$r^\mathrm{out}$} (x6);
      \begin{scope}
        \path[clip] (x6) -- (x3) -- (x3.center) -- (x6.center)-- (x6);
        \draw[-{Latex},double distance=1.5] (x3) to (x6);
      \end{scope}
      \draw[{-{Latex}[black]},white,line width=1pt] (x3) to (x6);
    \draw[-{Latex},thick,dashed] (x6) to [bend left] (x5);
  \end{tikzpicture}
\caption{A diconnected component $\fd$ of damped fluxes, chained in a cycle.}
\label{fig:damped cycle}
\end{figure}
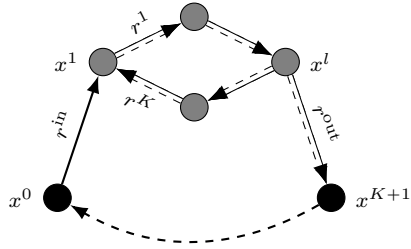

Initially we concentrate all mass in $x^0$, i.e. $u^\epsilon_{x^0}(0):=1/\pi^\epsilon_{x^0}$, and $u^\epsilon_x(0)=0$ for all other nodes $x\in\V$. The rate functional of the initial condition is indeed uniformly bounded:
\begin{equation*}
  \tilde\I_0^\epsilon\big(u^\epsilon(0)\big) = s(1\mid \underbrace{ \pi^\epsilon_{x^0}}_{\bigoh(1)}) + \underbrace{\textstyle\sum_{x\neq x^0} \pi^\epsilon_x}_{\leq 1}.
\end{equation*}

Define:
\begin{equation*}
  \Delta^\epsilon_T(t):=
    \begin{cases}
      0,                                    &0\leq t\leq \tfrac12T-\tfrac12\sqrt{\epsilon},\\
      t-\tfrac12T+\tfrac12\sqrt\epsilon,    &\tfrac12T-\tfrac12\sqrt{\epsilon} < t \leq \tfrac12T,\\
      \tfrac12T+\tfrac12\sqrt\epsilon -t,   &\tfrac12T < t \leq \tfrac12T + \tfrac12\sqrt{\epsilon},\\
      0,                                    &\tfrac12T + \tfrac12\sqrt{\epsilon} < t \leq T.
    \end{cases}
\end{equation*}
For the dynamics, we will first transport a little bit of mass from $x^0$ into each node of the cycle $\fd$, then develop a spike at $t=T/2$, and then release all mass from the cycle through $r^\mathrm{out}$. 
\begin{align*}
  j^\epsilon_{r^\mathrm{in}}(t)&:=K \mathds1_{(T/2-\sqrt{\epsilon}/2,T/2)}(t),\\
  j^\epsilon_{r^k}(t)&:=a_k\mathds1_{(T/2-\sqrt{\epsilon}/2,T/2)}(t) +  \tfrac1\epsilon \Delta^\epsilon_T(t) +  b_k\mathds1_{(T/2,T/2+\sqrt{\epsilon}/2)}(t),\\
  j^\epsilon_{r^\mathrm{out}}(t)&:=K \mathds1_{(T/2,T/2+\sqrt{\epsilon}/2)}(t),
\end{align*}
where $a_k:=K-k$ and $b_k:=k-l+K\mathds1_{\{k<l\}}$. We set all other fluxes $j^\epsilon_r,\tjeps_r$ in the network to $0$. By construction, $j^\epsilon_r\narrowto\tfrac14\delta_{T/2}$, which is singular as was to be shown.

We now show that the functional $\tilde\J^\epsilon$ is uniformly bounded. To calculate the densities, note that the $\tfrac1\epsilon\Delta_T^\epsilon$ terms in $j^\epsilon_{r^k}$ are divergence free. The mild formulation of the continuity equation~\eqref{eq:cont eq V1} thus yields for all $k=1,\hdots,K$ and $t\in[0,T]$,
\begin{align*}
  &\pi^\epsilon_{x^0} u^\epsilon_{x^0}(t) - 1 = 
    \begin{cases}
      0,                                    &0\leq t < T/2-\sqrt\epsilon/2,\\
      -K(t-\tfrac12T+\tfrac12\sqrt\epsilon),    &T/2-\sqrt{\epsilon}/2 < t \leq T/2,\\
      -\tfrac12K\sqrt\epsilon,                &t \geq T/2,
    \end{cases}\\
  &\pi^\epsilon_{x^k} u^\epsilon_{x^k}(t) = \mathds1_{\{k=1\}} j^\epsilon_{r^\mathrm{in}}[0,t] + j^\epsilon_{r^{k-1}}[0,t] - j^\epsilon_{r^k}[0,t] - \mathds1_{\{k=l\}}j^\epsilon_{r^\mathrm{out}}[0,t]
    = \Delta_T^\epsilon(t), \\
  &\pi^\epsilon_{x^{K+1}} u^\epsilon_{x^{K+1}}(t) =
    \begin{cases}
      0,                                    &0\leq t < T/2-\sqrt\epsilon/2,\\
      K(t-\tfrac12T+\tfrac12\sqrt\epsilon),    &T/2-\sqrt{\epsilon}/2 < t \leq T/2,\\
      \tfrac12K\sqrt\epsilon,                &t \geq T/2.
    \end{cases}
\end{align*}
The dynamic part of the rate functional is:
\begin{align*}
  \tilde\J^\epsilon\big(u^\epsilon_{\V_0},u^\epsilon_{\V_1},j^\epsilon_{\Rslow},j^\epsilon_{\Rdamped},\tjeps_{\Rfcycle}\big)
    =&
     \underbrace{\int_{[0,T]}\!s\big(j^\epsilon_{r^\mathrm{in}}(t) \mid \kappa_{r^\mathrm{in}} \pi^\epsilon_{x^0} u^\epsilon_{x^0}(t)\big)\,dt}_{(I)}\\
     &+ \sum_{k=1}^K \underbrace{\int_{[0,T]}\!s\big(j^\epsilon_{r^k}(t) \mid \tfrac1\epsilon\kappa_{r^k}\pi^\epsilon_{x^k}u^\epsilon_{x^k}(t)\big)\,dt}_{(II)}\\
     &+ \underbrace{\int_{[0,T]}\!s\big(j^\epsilon_{r^\mathrm{out}}(t) \mid \tfrac1\epsilon\kappa_{r^\mathrm{out}} \pi^\epsilon_{x^l} u^\epsilon_{x^l}(t)\big)\,dt}_{(III)}\\
     &+ \sum_{r^\mathrm{in}\neq r\in\Rslow:r_-= x^0} \underbrace{\int_{[0,T]}\!s\big(0 \mid \kappa_{r_-}\pi^\epsilon_{x^0}u^\epsilon_{x^0}(t)\big)\,dt}_{(IV)}\\
     &+ \sum_{r^\mathrm{out}\neq r\in\Rslow:r_-= x^{K+1}} \underbrace{\int_{[0,T]}\!s\big(0 \mid \kappa_{r_-}\pi^\epsilon_{x^{K+1}}u^\epsilon_{x^{K+1}}(t)\big)\,dt}_{(V)}.
\end{align*}
By a long but simple calculation, these integrals can be calculated explicitly:
\begin{align*}
  (I) &= \kappa_{r^\mathrm{in}}\big( T-\sqrt{\epsilon}(\tfrac12+\tfrac14 KT)\big) +\tfrac12\sqrt{\epsilon}(K\log\mfrac{K}{\kappa_{r^\mathrm{in}}}-K)+s(1-\tfrac12K\sqrt{\epsilon}\mid1)+\tfrac18\kappa_{r^\mathrm{in}}K\epsilon,\\
  (II) &=\int_0^{\sqrt{\epsilon}/2}\!s\big(a_k+\tfrac1\epsilon t\mid \tfrac1\epsilon\kappa_{r^k}\pi^\epsilon_{x^k} t\big)\,dt
          + 
         \int_0^{\sqrt{\epsilon}/2}\!s\big(b_k+\tfrac1\epsilon t\mid \tfrac1\epsilon\kappa_{r^k}\pi^\epsilon_{x^k} t\big)\,dt\\
       &=\mfrac12\epsilon a_k^2\log\Big(\frac{\epsilon a_k+\tfrac12\sqrt{\epsilon}}{\epsilon a_k}\Big)  + (\tfrac12\sqrt{\epsilon}a_k+\tfrac18)\log\big(\mfrac{2\sqrt{\epsilon}a_k+1}{\kappa_{r^k}}    \big) -\tfrac14 \sqrt{\epsilon} a_k + \tfrac18\kappa_{r^k} - \tfrac18\\
       &\quad+\mfrac12\epsilon b_k^2\log\Big(\frac{\epsilon b_k+\tfrac12\sqrt{\epsilon}}{\epsilon b_k}\Big)  + (\tfrac12\sqrt{\epsilon}b_k+\tfrac18)\log\big(\mfrac{2\sqrt{\epsilon}b_k+1}{\kappa_{r^k}}    \big) -\tfrac14 \sqrt{\epsilon} b_k + \tfrac18\kappa_{r^k} - \tfrac18,\\
  (III) &=\int_0^{\sqrt{\epsilon}/2}\!s\big(0 \mid \tfrac1\epsilon\kappa_{r^\mathrm{out}} t\big)\,dt + \int_0^{\sqrt{\epsilon}/2}\!s\big(K \mid \tfrac1\epsilon\kappa_{r^\mathrm{out}} t\big)\,dt = \mfrac12K\sqrt{\epsilon}\log\Big( \frac{2K\sqrt{\epsilon}}{\kappa_{r^\mathrm{out}}}\Big) + \mfrac14\kappa_{r^\mathrm{out}},\\
  (IV) &=\kappa_{r_-}\int_{[0,T]}\!\underbrace{\pi^\epsilon_{x^0}u^\epsilon_{x^0}(t)}_{\leq 1}\,dt \qquad\qquad\text{and}\qquad\qquad
  (V)  =\kappa_{r_-}\int_{[0,T]}\!\underbrace{\pi^\epsilon_{x^{K+1}}u^\epsilon_{x^{K+1}}(t)}_{\leq 1}\,dt.
\end{align*}
It thus follows that $\tilde\I^\epsilon_0+\tilde\J^\epsilon$ is uniformly bounded as claimed.

Recall  the three assumptions we made in the beginning of the proof. The second assumption is just notational. The first assumption, that all edges in $\fd$ are chained in a cycle, can easily be relaxed by fixing additional concentrations and damped fluxes to $0$, which keeps the rate functional finite. The third assumption would be violated if there were a chain of damped fluxes between a $\V_0$-node $x^0$ and $x^1$ or between a $\V_0$-node $x^l$ and $x^{K+1}$; in that case we can again set these fluxes equal to $j^\epsilon_{r^\mathrm{in}},j^\epsilon_{r^\mathrm{out}}$ respectively, without having the rate functional blowing up, which relaxes the last assumption.
\end{proof}

\section{Implications for large deviations and the effective dynamics}
\label{sec:density ldp and effective dynamics}

We now prove two consequences: the $\Gamma$-convergence of the density large deviations, and  the convergence of $\e$-level solutions to the solution of the effective dynamics.

\subsection{$\Gamma$-convergence of the density large deviatons}

As a consequence of our main $\Gamma$-convergence result, we obtain the $\Gamma$-convergence for the density large-deviation rate functional $\tilde\I^\epsilon_0+\tilde\I^\epsilon$ given by
\begin{align*}
  \tilde\I^\epsilon(u_{\V_0\slow},&u_{\V_{0\fcycle}}, u_{\fC}, u_{\V_1}) := \\ &\inf_{\substack{j_\Rslow\in L^\C([0,T];\RR^\Rslow),\\j_\Rdamped\in\M([0,T];\RR^\Rdamped),\\\tilde\jmath_\Rfcycle\in L^\C([0,T];\RR^\Rfcycle)}}     
    \tilde\J^\epsilon(u_{\V_0\slow},u_{\V_{0\fcycle}},u_{\fC}, u_{\V_1},j_\Rslow,j_\Rdamped,\tilde\jmath_\Rfcycle).
\end{align*}

\begin{corollary}
In $  C([0,T];\RR^{\V_{0\slow}})
    \times
  L^\infty([0,T];\RR^{\V_{0\fcycle}})
    \times
  C([0,T];\RR^{\fC})  
    \times
  \M([0,T];\RR^{\V_1})$ (equipped with the uniform, uniform, uniform,  and narrow topologies),
\begin{equation*}
  \tilde\I^\epsilon_0+\tilde\I^\epsilon \xrightarrow{\Gamma}  \tilde\I^0_0+\tilde\I^0,
\end{equation*}
where
\begin{equation*}
  \tilde\I^0(u_{\V_0},u_{\V_1}):=\inf_{\substack{j_\Rslow\in L^\C([0,T];\RR^\Rslow),\\j_\Rdamped\in\M([0,T];\RR^\Rdamped),\\\tilde\jmath_\Rfcycle\in L^\C([0,T];\RR^\Rfcycle)}}     \tilde\J^0(u_{\V_0},u_{\V_1},j_\Rslow,j_\Rdamped,\tilde\jmath_\Rfcycle).
\end{equation*}
\label{cor:density Gamma convergence}
\end{corollary}
\begin{proof}
The proof is more-or-less classic but we include it here for completeness. For brevity we write $u=(u_{\V_0\slow},u_{\V_{0\fcycle}}, u_{\fC}, u_{\V_1})$ and $j=(j_\Rslow,j_\Rdamped,\tilde\jmath_\Rfcycle)$. 

To prove the $\Gamma$-lower bound, take an arbitrary convergent sequence $u^\epsilon_{\V_0}\to u_{\V_0}$, $u^\epsilon_{\V_1}\narrowto u_{\V_1}$, and choose a corresponding sequence~$j^\epsilon$ that satisfies for each $\e>0$  the inequality
\begin{equation*}
  \tilde\J^\epsilon(u^\epsilon,j^\epsilon)
    \leq
  \inf_j \tilde\J^\epsilon(u^\epsilon,j) + \epsilon.
\end{equation*}
Without loss of generality we  assume that $\sup_{\epsilon>0} \tilde\I^\epsilon_0\big(u^\epsilon(0)\big)+\tilde\J^\epsilon(u^\epsilon,j^\epsilon)<\infty$. Hence by Corollary~\ref{cor:compactness} there exists a subsequence~$(u^\epsilon,j^\epsilon)$ (without changing notation) that converges in the sense of~\eqref{eq:conv-subseq} to a limit $(u,j)$.
From the $\Gamma$-lower bound Lemmas~\ref{lem:Gamma lower bound initial}, \ref{lem:Gamma lower bound slow} and \ref{lem:Gamma lower bound fcycle}, we find that:
\begin{align*}
  \liminf_{\epsilon\to0}\tilde\I^\epsilon_0(u^\epsilon(0))      &\geq \tilde\I^0_0(u(0)), \qquad\text{and}\\
  \liminf_{\epsilon\to0} \inf_j \tilde\J^\epsilon(u^\epsilon,j) &\geq \liminf_{\epsilon\to0} \tilde\J^\epsilon(u^\epsilon,j^\epsilon) - \epsilon \geq \tilde\J^0(u,j) \geq \inf_j \tilde\J^0(u,j).
\end{align*}
This proves the lower bound
\[
\liminf_{\e\to0} \tilde\I^\epsilon_0(u^\e(0))+\tilde\I^\epsilon(u^\e) \geq \tilde\I^0_0(u(0))+\tilde\I^0(u).
\]

For the recovery property, take an arbitrary $u$ with $\tilde\I^0_0(u(0))+\tilde\I^0(u) < \infty$, and for an arbitrary $\delta>0$, a flux $j^\delta$ such that 
\begin{equation*}
  \tilde\J^0(u,j^\delta)
    \leq
  \inf_j \tilde\J^0(u,j) + \delta = \tilde \I^0(u) + \delta.
\end{equation*}
Proposition~\ref{prop:rec seq} provides a recovery sequence $(u^\e,j^\e)$ for $(u,j^\delta)$ and the sequences $(\tilde \I^\e_0)_\e$ and $(\tilde \J^\e)_\e$, hence:
\begin{align*}
  \limsup_{\epsilon\to0}\tilde\I^\epsilon_0(u^\epsilon(0)) &\leq \tilde\I^0_0(u(0)), \qquad \text{and}\\
  \limsup_{\epsilon\to0} \tilde\I^\epsilon(u^\e)     &\leq \limsup_{\epsilon\to0} \tilde\J^\epsilon(u^\e,j^\e) \leq \tilde\J^0(u,j^\delta) \leq\tilde\I^0(u) + \delta.
\end{align*}
Since $\delta>0$ is arbitrary, the recovery property follows. 
\end{proof}

\begin{remark} By the same argument one may also contract further to obtain $\Gamma$-convergence of the functional
\begin{multline}
  u_{\V_0}\mapsto\inf_{u_{\V_1},j_\Rslow,j_\Rdamped,\tilde\jmath_\Rfcycle} 
  \tilde\I^\epsilon_0\big(u_{\V_0\slow}(0),u_{\V_{0\fcycle}}(0), u_{\fC}(0), u_{\V_1}(0)\big)\\
   + \tilde\J^\epsilon(u_{\V_0\slow},u_{\V_{0\fcycle}}, u_{\fC}, u_{\V_1},j_\Rslow,j_\Rdamped,\tilde\jmath_\Rfcycle).
\end{multline}
\end{remark}

\subsection{Convergence to the effective equations}
\label{ss:convergence-to-effective-eqns}

For any pair $(u,j)$ at which the limiting functional  $\J^0$ vanishes, the densities satisfy the following set of equations in the weak sense of~\eqref{eq:weak sense}:
\begin{subequations}
\label{eq:effective eq c}
\begin{align}
\label{eq:effective eq c V0}
    \pi_x \dot u_x &= 
    \sum_{\substack{r\in\Rslow:\\r_+=x}} \kappa_r\pi_{r_-}\! u_{r_-}
      + 
    \sum_{\substack{r\in\Rdamped:\\r_+=x}} \kappa_r\tilde\pi_{r_-}\! u_{r_-}
      -
    \sum_{\substack{r\in\Rslow:\\r_-=x}} \kappa_r\pi_x u_x 
&\qquad& \text{for } x\in\V_{0\slow},\\
    \pi_\fc \dot u_\fc &= 
    \sum_{\substack{r\in\Rslow:\\r_+\in\fc}} \kappa_r\pi_{r_-}\! u_{r_-}
      + 
    \sum_{\substack{r\in\Rdamped:\\r_+\in\fc}} \kappa_r\tilde\pi_{r_-}\! u_{r_-}
      -
    \sum_{\substack{r\in\Rslow:\\r_-\in\fc}} \kappa_r\pi_{r-} u_{r-}
    \label{eq:effective eq c c} \\
u_x &= u_\fc \qquad \text{for any }x\in \fc\in \fC,
\label{eq:effective eq c ux=uc}\\[2\jot]
    0 &= 
    \sum_{\substack{r\in\Rslow:\\r_+=x}} \kappa_r\pi_{r_-}\! u_{r_-}      + 
    \sum_{\substack{r\in\Rdamped:\\r_+=x}} \kappa_r\tilde\pi_{r_-}\! u_{r_-} -
    \sum_{\substack{r\in\Rdamped:\\r_-=x}} \kappa_r\tilde\pi_x u_x
&\qquad& \text{for } x\in\V_1.
\label{eq:limit-eq-ODE}
\end{align}
\end{subequations}
We first prove existence and uniqueness for these equations.  

\begin{lemma}
\label{l:uniqueness-limit}
Fix an initial condition $u(0)\in \RR^{\V}$ that is well-prepared, which means that 
\begin{enumerate}
\item Whenever $x,y$ are in the same connected component $\fc\in\fC$, we have $u_x=u_y$; we denote the common value by $u_\fc$;
\item $u(0)$ satisfies the condition~\eqref{eq:limit-eq-ODE}.	
\end{enumerate}
Then the system of equations~\eqref{eq:effective eq c}
has a unique solution $u\in C^\infty([0,\infty);\RR^{\V})$ with initial value $u(0)$.
\label{lem:eff eq uniqueness}
\end{lemma}


\begin{proof}
Since $\dot u_x = \dot u_\fc$ whenever $x\in \fc\in \fC$, 
equation~\eqref{eq:effective eq c c} can be rewritten as 
\begin{equation}
\label{eq:eff-modified-c}	
    \pi_\fc \dot u_x = 
    \sum_{\substack{r\in\Rslow:\\r_+\in\fc}} \kappa_r\pi_{r_-}\! u_{r_-}
      + 
    \sum_{\substack{r\in\Rdamped:\\r_+\in\fc}} \kappa_r\tilde\pi_{r_-}\! u_{r_-}
      -
    \sum_{\substack{r\in\Rslow:\\r_-\in\fc}} \kappa_r\pi_{r-} u_{r-}
    \qquad\text{for all }x\in\fc\in\fC,
\end{equation}
The right-hand side does not depend on the choice of $x$ within the same $\fc\in\fC$; therefore, under the assumption  that $u_x(0)=u_\fc(0)$ for all $x\in\fc\in\fC$,  the system~\eqref{eq:effective eq c} is equivalent to the  set of equations~\eqref{eq:effective eq c V0}--\eqref{eq:limit-eq-ODE}--\eqref{eq:eff-modified-c}.

This implies that the system~\eqref{eq:effective eq c} can be written as a differential-algebraic equation:
\begin{subequations}
\label{eq:DAE}
\begin{align}
  \dot u_{\V_0} &= A_{\V_0\to\V_0}u_{\V_0} + A_{\V_1\to\V_0} u_{\V_1},\\
   0     &= A_{\V_0\to\V_1} u_{\V_0} + A_{\V_1\to\V_1} u_{\V_1},
\end{align}
\end{subequations}
where for $x\in\V_1$,
\begin{align*}
  (A_{\V_1\to\V_1} u_{\V_1})_x &:=\sum_{\substack{r \in\Rdamped\\r_+=x}} \kappa_r\mfrac{\tilde{\pi}_{r_-}}{\tilde{\pi}_x} u_{r_-} - u_x\sum_{\substack{r\in\Rdamped\\r_-=x}} \kappa_r.
\end{align*}
By the next lemma the matrix $A_{\V_1\to\V_1}$ is invertible, and therefore \eqref{eq:DAE} can be cast in the form of a linear ordinary differential equation for $u_{\V_0}$. This equation has unique solutions with $C^\infty$ regularity, and by transforming back we find  that $u_{\V_1}$ has the same regularity as~$u_{\V_{0}}$.
\end{proof}

\begin{lemma}
\label{l:matrix-is-invertible}
Under the conditions of the previous lemma, the matrix $	A_{\V_1\to\V_1}$ is invertible.
\end{lemma}

\begin{proof}
We first note that the matrix $A_{\V_1\to\V_1}$  can be written as
\[
A_{\V_1\to\V_1} = \diag(\tilde{\pi})^{-1} \bigl(A^{\mathrm{int}}-\diag(E)\bigr)\diag(\tilde\pi),
\]
with for $x,y\in \V_1$, 
\[
A^{\mathrm{int}}_{xy} = \kappa_{y\to x} - \delta_{xy} \sum_{y'\in \V_1} \kappa_{x\to y'},
\qquad
E_x =  \sum_{y'\in \V_0} \kappa_{x\to y'}.
\]
Since $\diag(\tilde\pi)$ is invertible, it is sufficient to show that $A^{\mathrm{int}}-\diag(E)$ is invertible.

\medskip
To do this we construct a new graph $\hat{\calG} := (\V_1\cup\{\gy\}	,\Rint\cup\Rgy)$, consisting of the nodes of $\V_1$ and a single `graveyard' node $\gy$; the graveyard collects all elements of $\V_0$ into one new node. The graph $\hat{\calG}$ has edges
\begin{align*}
	\Rint &:= \Bigl\{(x\to y)\in \V_1\times \V_1\backslash\{z\to z\}: \exists r\in\Rdamped \text{ such that } r_-= x\text{ and } r_+ =y \Bigr\}\\
  \Rgy &:= \Bigl\{ (x\to\gy): x\in \V_1, \ \exists r\in \Rdamped\text{ such that } r_-= x\text{ and } r_+ \in \V_0\Bigr\}.
\end{align*}
Note that there are no fluxes out of $\gy$.

We define a new Markov jump process $Z(t)$ on this graph $\tilde{\calG}$, by specifying jump rates $\hat{\kappa}_{x\to y}$ for each edge $(x\to y)$ in $\hat{\calG}$:
\begin{equation*}
  \hat{\kappa}_{x\to y} :=
    \begin{cases}
      \sum_{r\in\Rdamped:r_-=x,r_+=y}\kappa_r, &\text{for } (x\to y)\in\Rint,\\
      \sum_{r'\in \Rdamped: r'_-=x,  r'_+ \in\V_0} \kappa_{r'}, &\text{for }(x\to y) \in \Rgy.
    \end{cases}
\end{equation*}
The generator for this jump process is the matrix $L$ given by
\[
L_{xy}:= 
\begin{cases}
	\hat{\kappa}_{x \to y} & \text{if }(x\to y)\in \Rint\cup \Rgy \text{ (which implies $x\not=y$)}\\
	-\sum_{y'\in \V_1\cup\{\gy\}} \hat{\kappa}_{x\to y'} & \text{if }x=y\in \V_1 \\
	0 &\text{otherwise.}
\end{cases}
\]
By construction the transpose $L^T$ of this generator has the following structure in terms of the splitting $\V_1\cup \{\gy\}$:
\newcommand\Tstrut{\rule{0pt}{2.6ex}}         
\newcommand\Bstrut{\rule[-0.9ex]{0pt}{0pt}}   
\[
L^T = \left(
\begin{array}{c|c}
	\Bstrut A^{\mathrm{int}} -\diag(E) & 0\\\hline
	\Tstrut E^T & 0
\end{array}\right),
\]

\medskip

Since the original graph $\calG$ is diconnected, there exists for each $x\in \V_1$ a path $x = x_0\to x_1\to\dots \to x_k$ in $\calG$ leading to some $x_k\in \V_0$; without loss of generality we assume that $x_0,x_1,\dots,x_{k-1}\in \V_1$. Since fluxes out of nodes in $\V_1$ are damped, the fluxes $(x_0\to x_1), \dots, (x_{k-1}\to x_k)$ are all in $\Rdamped$. Since these fluxes also exist as fluxes $\Rint$ in the graph $\hat{\calG}$, the path $x_0\to x_1\to \dots \to x_{k-1}$ also is a path in $\hat{\calG}$. By construction, $\Rgy$ contains a reaction $r = (x_{k-1}\to \gy)$ with positive rate $\hat{\kappa}_r$.

It follows that if the process $Z(t)$ starts at any $x\in \V_1$, then at each positive time $t>0$ there is a positive probability that $Z(t) = \gy$. Since the graveyard $\gy$ has no outgoing fluxes, the only invariant measure for the process $Z(t)$ is $\One_{\gy}:=(0,0,\dots,0,1)$, and so the kernel of $L^T$ coincides with the span of $\One_\gy$. Consequently the matrix $A^{\mathrm{int}} -\diag(E)$ is invertible because the row $E^T$ is a linear combination of the other rows of $L^T$.
\end{proof}

%

We finally derive convergence to the full effective equations.

\begin{corollary}
 For each $\epsilon>0$ let $(u^\epsilon_{\V_0},u^\epsilon_{\V_1},j^\epsilon_\Rslow,j^\epsilon_\Rdamped,\tjeps_\Rfcycle)$ in $\Theta$ solve the system of equations:
\begin{align*}
  \begin{cases}
    j^\epsilon_r(t)=\kappa_r\pi^\epsilon_{r_-}\! u^\epsilon_{r_-}\!(t), &r\in\Rslow,\\
    j^\epsilon_r(t)=\tfrac1\epsilon\kappa_r\pi^\epsilon_{r_-}\! u^\epsilon_{r_-}\!(t), &r\in\Rdamped,\\
    \tjeps_r(t)=0, &r\in\Rfcycle,\\
    \pi^\epsilon  \dot u^\epsilon(t) = - \div j^\epsilon(t),&\text{in the weak sense of \eqref{eq:cont eq},}\\
    u^\epsilon(0)= u^{\e,0},&
  \end{cases}
\end{align*}
where $u^{\e,0}$ is given. Assume that $u^{\epsilon,0}_{\V_{0\slow}},u^{\epsilon,0}_{\fC}$ converge to some $u^{0,0}_{\V_{0\slow}},u^{0,0}_{\fC}>0$, that $u^{0,0}$ is well-prepared in the sense of Lemma~\ref{l:uniqueness-limit}. In addition, assume that for each $x\in \V_1$, $\log u^{\e,0}_x	$ remains bounded. Then $(u^\epsilon_{\V_0},u^\epsilon_{\V_1},j^\epsilon_\Rslow,j^\epsilon_\Rdamped,\tjeps_\Rfcycle)$ converges in $\Theta$ to $(u_{\V_0},u_{\V_1},j_\Rslow,j_\Rdamped,\tilde\jmath_\Rfcycle)$, which is the unique solution to
\begin{align}
  \begin{cases}
    j_r(t)=\kappa_r\pi_{r_-}\! u_{r_-}\!(t), &r\in\Rslow,\\
    j_r(t)=\kappa_r\tilde\pi_{r_-}\! u_{r_-}\!(t), &r\in\Rdamped,\\
    \tilde\jmath_r(t)=0, &r\in\Rfcycle,\\
    \pi  \dot u(t) = - \div j(t), &\text{in the weak sense of \eqref{eq:limit cont eq},}\\
    u(0)= u^{0,0}.
  \end{cases}
\label{eq:effective eq}
\end{align}
\end{corollary}
\begin{proof} 
Set:
\begin{align*}
  \widecheck \I_0^\e\big(u(0)\big) :=\sum_{x\in\V} s\big(\pi^\epsilon_x u_x(0) \mid \pi^\epsilon_x u_x^{\e,0} \big)
  &&\text{and}&&
  F^\e\big(u(0)\big) := \sum_{x\in \V} \pi_x^\e \Bigl(-u_x(0)\log u_x^{\e,0}-1+u_x^{\e,0}\Bigr).
\end{align*}
Then for each $\epsilon>0$, the solution $(u^\epsilon_{\V_0},u^\epsilon_{\V_1},j^\epsilon_\Rslow,j^\epsilon_\Rdamped,\tjeps_\Rfcycle)$ minimises the modified functional $\widecheck \I_0^\e + \tilde \J^\e = \tilde \I_0^\e + F^\e + \tilde \J^\e:\Theta\to\lbrack0,\infty\rbrack$ at value zero. In particular this means that:
\[
  \sup_{\e>0}\  \tilde \I_0^\e(u^\e(0)) + \tilde \J^\e(u^\e,j^\e)  = \sup_{\e>0} \, -F^\e\big(u^\epsilon(0)\big) = \sup_{\e>0}\sum_{x\in\V} \pi_x^\epsilon\big( u_x^{\e,0}\log u_x^{\e,0} + 1 - u_x^{\e,0}\big) < \infty.
\]
By Corollary~\ref{cor:compactness}, the sequence $(u^\e,j^\e)$ has a subsequence that converges in the sense of~\eqref{eq:conv-subseq} to a limit $(u,j)$. By the assumptions on $u^{\e,0}$, the functional $F^\e$ converges along the sequence $(u^\e,j^\e)$ to the limit $F^0$, where
\[
F^0(v(0)) := \sum_{x\in \V_{0\slow}} \pi_x \Bigl( -v_x(0)\log u_x^{0,0}-1 + u_x^{0,0})\Bigr) 
+ \sum_{\fc\in\fC} \pi_\fc \Bigl(-v_\fc(0)\log u_\fc^{0,0}-1+u_\fc^{0,0}\Bigr).
\]
With the $\Gamma$-lower bound of Proposition~\ref{prop:lower bound} it follows that 
\begin{align*}
0 &= \liminf_{\e\to0} \tilde \I_0^\e\big(u^\e(0)\big) + F^\e\big(u^\e(0)\big) + \tilde \J^\e(u^\e,j^\e) \\
&\geq \tilde \I_0^0\big(u(0)\big) + F^0\big(u(0)\big) + \tilde \J^0(u,j) = \widecheck \I_0^0(u(0)) + \tilde \J^0(u,j).
\end{align*}
Here 
\[
  \widecheck \I_0^0\big(v(0)\big):=\sum_{x\in\V_{0\slow}} s\big(\pi_x u_x(0) \mid \pi_x u_x^{0,0} \big) + \sum_{\fc\in\fC} s\big(\pi_\fc u_\fc(0) \mid \pi_\fc u_\fc^{0,0} \big).
\]
It follows that the limit $(u,j)$ is a solution of the problem $\widecheck \I_0^0 + \tilde \J^0=0$, which coincides with~\eqref{eq:effective eq}.
\end{proof}




\appendix

\section{The Arzel\`a-Ascoli theorem for asymptotic uniformly equicontinuous sequences}

The classical Arzel\`a-Ascoli theorem asserts that a set of continuous functions on a compact set is precompact in the supremum norm if and only if it is uniformly bounded and uniformly equicontinuous. For countable sets such as sequences the uniform equicontinuity is equivalent to \emph{asymptotic} uniform equicontinuity, and this observation leads to the alternative version below. This is mentioned in various places in the literature (e.g.~\cite[Rem.~2.3\ (ii)]{PotscherPrucha94} or~\cite[Ex.~5.27]{Davidson94}) but since we could not find a clear statement  we state and prove it here. 

\begin{theorem}
\label{th:mod-AA}
Let $(f_n)_{n\geq 1}$ be a sequence of continuous real-valued functions on $[0,T]$ that satisfies
\begin{enumerate}
\item $\sup_{n\geq 1} \|f_n\|_\infty <\infty$;
\item There exists $\omega:[0,\infty)\to[0,\infty)$, non-decreasing, with $\lim_{\sigma\downarrow 0} \omega(\sigma)= 0$, such that, 
\[
\limsup_{n\to\infty} \sup_{|t-s|<\sigma} |f_n(t)-f_n(s)| \leq \omega(\sigma).
\]
\end{enumerate}
Then there exists a subsequence $f_{n_k}$ that converges uniformly on $[0,T]$.
\end{theorem}

\begin{proof}
We prove the result by showing that the sequence $(f_n)_n$ also is uniformly equicontinuous in the usual sense. 
Fix $\e>0$. Choose $N\geq1$ and $\sigma_0>0$ such that 
\[
\forall \, n\geq N \quad \forall\, |t-s|< \sigma_0: \qquad 
|f_n(t)-f_n(s)| < \e. 
\]
Next, choose $\sigma_1>0$ such that 
\[
\forall\, 1\leq n< N \quad \forall\, |t-s|< \sigma_1: \qquad 
|f_n(t)-f_n(s)|< \e.
\]
Then for all $n\geq 1$ and $|t-s|<\sigma_0\wedge \sigma_1$ we have $|f_n(t)-f_n(s)|<\e$. This proves that $(f_n)_n$ is uniformly equicontinuous, and therefore the result follows from the classical Arzel\`a-Ascoli theorem.
%
\end{proof}

\section{Definiteness of Markov generators}

For completeness we include the following basic result. 
\begin{lemma}
Let $0\neq A\in\RR^{d\times d}$ be a Markov generator matrix. Then
\begin{align*}
  v\tp A v \leq 0 &&\text{for all } v \in \RR^d,
\intertext{and there exists a $\lambda<0$ such that}
  v\tp A v \leq \lambda \lvert v\rvert_2^2   &&\text{for all } v \in \Col(A).
\end{align*}
\label{lem:definite generator}
\end{lemma}

\begin{proof}
Since $v\tp A v=\tfrac12 v\tp (A+A\tp) v$ we may assume without loss of generality that $A$ is symmetric, and hence diagonalisable by orthogonal matrices. If the Markov chain is irreducible, then by the Perron-Frobenius theorem the largest eigenvalue is $0$, with multiplicity $m=1$. If the chain is reducible, then by symmetry the Markov chain consists of $m>1$ disconnected irreducible components, each of which has largest eigenvalue $0$, so $A$ has largest eigenvalue $0$ with multiplicity $m$. This proves the first claim.

We order the eigenvalues in a descending fashion, and write $A=P\Lambda P\tp$ where
\begin{align*}
  \Lambda=
    \begin{bsmallmatrix}
      0 & \\    
        & \ddots \\ 
        &        & 0 \\
        &        &   & \lambda_{m+1} \\
        &        &   &               & \ddots \\
        &        &   &               &        & \lambda_{d}
    \end{bsmallmatrix}
    =
    \begin{bmatrix}
       0 & 0\\
       0 & \Lambda^\mathrm{neg}\\
    \end{bmatrix}
&&\text{and}&&
  P=
  \begin{bsmallmatrix}
     v_1 &\hdots &v_m &v_{m+1} &\hdots &v_d
  \end{bsmallmatrix}
  =
  \begin{bmatrix}
     P^0 P^\mathrm{neg}
  \end{bmatrix},
\end{align*}
and $P$ is orthonormal, and $\Lambda^\mathrm{neg}$ has only negative diagonal entries. Since $P^0$ contains only eigenvectors with zero eigenvalues, 
$\Col(A)=\Col(P^\mathrm{neg})$ and one can parametrise $\Col(A)\ni v = P^\mathrm{neg} w$ for any $w\in \RR^{d-m}$. By orthonormality, we can write
\begin{equation*}
  {P^\mathrm{neg}}\tp A P^\mathrm{neg} = \Lambda^\mathrm{neg}.
\end{equation*}
Choosing $\lambda=\lambda_{m+1}$, the largest non-zero eigenvalue, yields the second claim.
\end{proof}

\section*{Acknowledgements}

This research has been funded by the
Deutsche Forschungsgemeinschaft (DFG) through grant 
CRC 1114 "Scaling Cascades in Complex Systems", 
Project C08. We thank Robert Patterson for the useful discussions.

\bibliographystyle{alpha}
\bibliography{library} 

\end{document}